\documentclass[a4paper,11pt,oneside,reqno]{amsart}

\usepackage[english]{babel}
\usepackage[utf8x]{inputenc}    
\usepackage{xcolor}            
\usepackage{xspace}
\usepackage[pdftex]{graphicx}
\usepackage{epstopdf}
\usepackage{mathrsfs}

\usepackage{charter}

\usepackage{mathabx, stmaryrd}

\makeatletter
\newcommand{\leqnos}{\tagsleft@true\let\veqno\@@leqno}
\newcommand{\reqnos}{\tagsleft@false\let\veqno\@@eqno}
\reqnos
\makeatother

\numberwithin{equation}{section}

\usepackage{caption}           
\usepackage{url}

\usepackage[a4paper,scale={0.72,0.74},marginratio={1:1},footskip=7mm,headsep=10mm]{geometry}

\usepackage{hyperref}
\hypersetup{					
	breaklinks=true,			
	colorlinks,
    citecolor=black,
    filecolor=black,
    linkcolor=black,
    urlcolor=black
}

\newcommand{\ind}{{\sf 1}}

\newcommand{\bP}{\mathbf{P}}
\newcommand{\bZ}{\mathbf{Z}}
\newcommand{\bE}{\mathbf{E}}
\newcommand{\bbP}{\mathbb{P}}
\newcommand{\bbE}{\mathbb{E}}
\newcommand{\bbR}{\mathbb{R}}
\newcommand{\bbN}{\mathbb{N}}
\newcommand{\bbZ}{\mathbb{Z}}

\newcommand{\ent}{\mathrm{Ent}}
\newcommand{\tent}{{\widehat{\rm E} \rm nt }}

\newcommand{\sD}{\mathscr{D}}
\newcommand{\sM}{\mathscr{M}}

\newcommand{\bU}{{\ensuremath{\mathbf U}} }

\newcommand{\cA}{{\ensuremath{\mathcal A}} }

\newcommand{\cP}{{\ensuremath{\mathcal P}} }

\newcommand{\cT}{{\ensuremath{\mathcal T}} }

\newcommand{\cW}{{\ensuremath{\mathcal W}} }
\newcommand{\cB}{{\ensuremath{\mathcal B}} }

\newcommand{\cZ}{{\ensuremath{\mathcal Z}} }
\newcommand{\cG}{{\ensuremath{\mathcal G}} }
\newcommand{\cM}{{\ensuremath{\mathcal M}} }
\newcommand{\bR}{{\ensuremath{\mathbf R}} }

\definecolor{nicolor}{RGB}{0,102,0}

\newcommand{\dd}{\textrm{d}}

\renewcommand{\epsilon}{\varepsilon}
\renewcommand{\phi}{\varphi}
\renewcommand{\tilde}{\widetilde}

\newtheorem{theorem}{Theorem}[section]

\newtheorem{proposition}[theorem]{Proposition}

\newtheorem{lemma}[theorem]{Lemma}

\theoremstyle{definition}

\newcommand{\ga}{\alpha}
\newcommand{\gb}{\beta}
\newcommand{\gd}{\delta}
\newcommand{\gep}{\varepsilon}       

\newcommand{\gD}{\Delta}

\newcommand{\go}{\omega}

\newcommand{\gU}{\Upsilon}

\renewcommand{\tilde}{\widetilde}
\renewcommand{\hat}{\widehat}

\newcommand{\sumtwo}[2]{\sum_{\substack{#1 \\ #2}}}

\title{Directed polymers in heavy-tail random environment}

\author[Q. Berger]{Quentin Berger}
\address{Sorbonne Universit\'e, LPSM,
Campus Pierre et Marie Curie, case 188,
4 place Jussieu, 75252 Paris Cedex 5, France}
\email{quentin.berger@sorbonne-universite.fr}

\author[N. Torri]{Niccol\`o Torri}
\address{Sorbonne Universit\'e, LPSM,
Campus Pierre et Marie Curie, case 188,
4 place Jussieu, 75252 Paris Cedex 5, France}
\email{niccolo.torri@sorbonne-universite.fr}

\date{}

\thanks{The authors acknowledge the support of PEPS grant from CNRS.
	Q. Berger was supported by a public grant overseen by the French National Research Agency ANR-17-CE40- 0032-02.
N. Torri was supported by a public grant overseen by the French National Research Agency (ANR) as part
of the ``Investissements d'Avenir'' program (ANR-11-LABX-0020-01 and ANR-10-LABX-0098).}

\begin{document}
\maketitle

\begin{abstract}
	\ We study the directed polymer model in dimension ${1+1}$ when the environment is heavy-tailed, with a decay exponent $\alpha\in(0,2)$. We give all possible scaling limits of the model in the \emph{weak-coupling} regime, \textit{i.e.}\ when the inverse temperature temperature $\beta=\beta_n$ vanishes as the size of the system $n$ goes to infinity.  When $\alpha\in(1/2,2)$, we show that all possible transversal fluctuations $\sqrt{n} \leq h_n \leq n$ can be achieved by tuning properly $\beta_n$, allowing to interpolate between all super-diffusive scales. 
	Moreover, we determine the scaling limit of the model, answering a conjecture by Dey and Zygouras \cite{cf:DZ}---we actually identify five different regimes.
	On the other hand, when $\alpha<1/2$,  we show that there are only two regimes: the transversal fluctuations are either $\sqrt{n}$ or $n$.
	As a key ingredient, we use the \emph{Entropy-controlled Last Passage Percolation} (E-LPP), introduced in a companion paper \cite{cf:BT_ELPP}.
	\\[0.1cm]
	\\[0.1cm]
	
\end{abstract}

\setcounter{tocdepth}{1}
\tableofcontents

\section{Introduction: Directed Polymers in Random Environment}

\subsection{General setting}
We consider the directed polymer model: it has been introduced by Huse and Henley~\cite{HH85} as an effective model for an interface in the Ising model with random interactions, and is now used to describe a stretched polymer interacting with an inhomogeneous solvent.

Let $S$ be a nearest-neighbor simple symmetric random walk on $\mathbb{Z}^d$, $d\geq 1$, whose law is denoted by $\mathbf{P}$, and let $(\omega_{i,x})_{i\in \mathbb{N},\, x\in \mathbb{Z}^d}$ be a field of i.i.d.\ random variables (the \textit{environment}) with law $\bbP$ ($\go$ will denote a random variable which has the common distribution of the $\go_{i,x}$). 
The \emph{directed} random walk $(i,S_i)_{i\in\mathbb N_0}$ represents a polymer trajectory and interacts with its environment via a coupling parameter $\gb>0$ (the inverse temperature).
The model is defined through a Gibbs measure,
\begin{equation}
\label{eq:DPRE}
\frac{\dd \bP^\omega_{n,\beta}}{\dd \bP}(s)\, :=\, 
\frac{1}{\bZ^\omega_{n\, \beta}} \exp\Big( \beta\, \sum_{i=1}^n\omega_{i,s_i} \Big)\, ,
\end{equation}
where $\bZ^\omega_{n\, \beta}$ is the \emph{partition function} of the model.

One of the main question about this model is that of the localization and super-diffusivity of paths trajectories drawn from the measure $\bP^\omega_{n,\beta}$.  The transversal exponent $\xi$ describes the fluctuation of the end-point, that is $\bbE \bE^\omega_{n,\beta} |S_n| \approx n^{\xi} $ as $n\to\infty$.
Another quantity of interest is the fluctuation exponent $\chi$, that describes the fluctuations of $\log \bZ_{n,\gb}^{\go}$, \textit{i.e.}\ 
$|\log \bZ_{n,\gb}^{\go} - \bbE \log \bZ_{n,\gb}^{\go}| \approx n^{\chi}$ as $n\to\infty$.

This model has been widely studied in the physical and mathematical literature  (we refer to \cite{C17, CSY04} for a general overview), in particular when $\go_{n,x}$ have an exponential moment. 
The case of the dimension $d=1$ as attracted much attention in recent years, in particular because the model is in the 
\emph{KPZ universality class} ($\log \bZ_{n,\gb}^{\go}$ is seen as a discretization of the Hopf-Cole solution of the KPZ equation).
It is conjectured that the transversal and fluctuation exponents are $\xi=2/3$ and $\chi=1/3$ respectively.
Moreover,  it is expected that the point-to-point partition function, when properly centered and renormalized, converges in distribution to the GUE distribution.
Such scalings has been proved so far only for some special models, cf.~\cite{BQS11,Se09}. 

A recent and fruitful approach to proving universality results for this model has been to consider is \emph{weak-coupling limit}, that is
when the coupling parameter $\gb$ is close to criticality. This means that we allow $\gb=\gb_n$ to depend on $n$, with $\gb_n \to 0$ as $n\to\infty$. 
In \cite{AKQ14b,AKQ14a} and \cite{CSZ13}, the authors let $\beta_n = \hat\beta n^{-\gamma},\, \gamma=1/4$ for some fixed $\hat \beta>0$,  and they prove that the model (one may focus on its partition function $\bZ_{n,\gb_n}^{\go}$) converges to a \emph{non-trivial} (\textit{i.e.}~\emph{disordered}) continuous version of the model. This is called the \emph{intermediate disorder regime}, since it somehow interpolates between weak disorder and strong disorder behaviors.
More precisely, they showed that
\[
\log \bZ_{n,\gb_n}^{\go} - n\lambda(\beta_n) \overset{(\dd)}{\longrightarrow} \log \cZ_{\sqrt 2 \hat \beta}, \quad \text{as}\quad n\to\infty,
\] 
where $\lambda(s):=\log \bbE[e^{s\omega}]$.  
The process $\hat\beta \mapsto \log\cZ_{\sqrt 2 \hat \beta}$ is the so called \emph{cross-over process}, and interpolates between  Gaussian and GUE scalings as $\hat \beta$ goes from $0$ to $\infty$ (see \cite{ACQ11}).
These results were obtained under the assumption that $\go$ has exponential moments, but the \emph{universality of the limit} was conjectured to hold under the assumption of six moments \cite{AKQ14a}. In \cite{cf:DZ} Dey and Zygouras proved this conjecture, and they suggest that this result is a part of a bigger picture (when $\lambda(s)$ is not defined a different centering is necessary).

\subsection{The case of a heavy-tail environment}

In the rest of the paper we will focus on the dimension $d=1$ for simplicity.
We consider the case where the environment distribution $\go$
is non-negative  (for simplicity, nothing deep is hidden in that assumption) and has some heavy tail distribution: there is some $\ga>0$ and some slowly varying function $L(\cdot)$ such that
\begin{equation}
\label{eq:DisTail}
\mathbb P\left(\omega> x\right)=L(x) x^{-\alpha} \, .
\end{equation}

In the case where $\gb>0$ does not depend on $n$, the $\xi=2/3,\chi=1/3$ picture is expected to be modified, depending on the value of $\ga$.  According to the heuristics (and terminology) of \cite{BBP07,GLDBR}, three regimes should occur, with different paths behaviors:

(a) if $\ga>5$, there should be a \emph{collective} optimization and we should have $\xi=2/3$, KPZ universality class, as in the finite exponential moment case;

(b) if $\ga \in (2,5)$, the optmization strategy should be \emph{elitist}: most of the total energy collected should be via a small fraction of the points visited by the path, and we should have $\xi=\frac{\ga+1}{2\ga-1}$;

(c) if $\ga\in(0,2)$, the strategy is \emph{individual}: the polymer targets few exceptional points, and we have $\xi=1$. This case is treated in \cite{AL11, HM07}.

\smallskip

As suggested by \cite{cf:DZ}, this is part of a larger picture, when the inverse temperature $\gb$ is allowed to depend on $n$. 
Setting $\gb_n= \hat \gb n^{-\gamma}$ for some $\hat \gb > 0$ and some $\gamma \in \bbR$ then we have three different classes of coupling.  When $\gamma=0$ we recover the standard directed polymer model, when $\gamma>0$ we have a weak-coupling limit, while in the case $\gamma<0$ we have a strong-coupling limit. 
Let us stress that this last case has not been studied in the literature (for no apparent reason) and should also be of interest.
In \cite{AKQ10} and in \cite{cf:DZ}, the authors suggest that the fluctuation exponent depends on $\ga,\gamma$ in the following manner 
\begin{equation}
\label{def:xi}
\xi = \left\{
\begin{array}{ll}
\frac{2 (1-\gamma)}{3}  &\qquad \text{for } \ga \ge \tfrac{5-2\gamma}{1-\gamma} ,\  - \tfrac 12\le \gamma \le \tfrac14  \, ,\\
\frac{ 1+ \ga(1-\gamma)}{2\ga-1}  &\qquad \text{for } \ga \le \frac{5-2\gamma}{1-\gamma},\   \tfrac{2}{\ga}-1 \le  \gamma \le \tfrac{3}{2\ga} \, .
\end{array}
\right.
\end{equation}
The first part is derived in \cite{AKQ10}, based on Airy process considerations, and the second part is derived  in \cite{cf:DZ}, based on a Flory argument inspired by \cite{BBP07}.
Moreover, in the two regions of the $(\ga,\gamma)$ plane defined by \eqref{def:xi}, the KPZ scaling relation $\chi=2\xi-1$ should hold (this has been proved in the case $\gamma=0,\ga>2$ in \cite{AD13}). Outside of these regions, one should have $\xi=1/2$ ($\gamma$ large) or $\xi=1$ ($\gamma$ small).
This is summarized in Figure~\ref{fig1} below, which is the analogous of \cite[Fig.~1]{cf:DZ}.
\begin{figure}[htbp]
\begin{center}
	\includegraphics[scale=0.65 ]{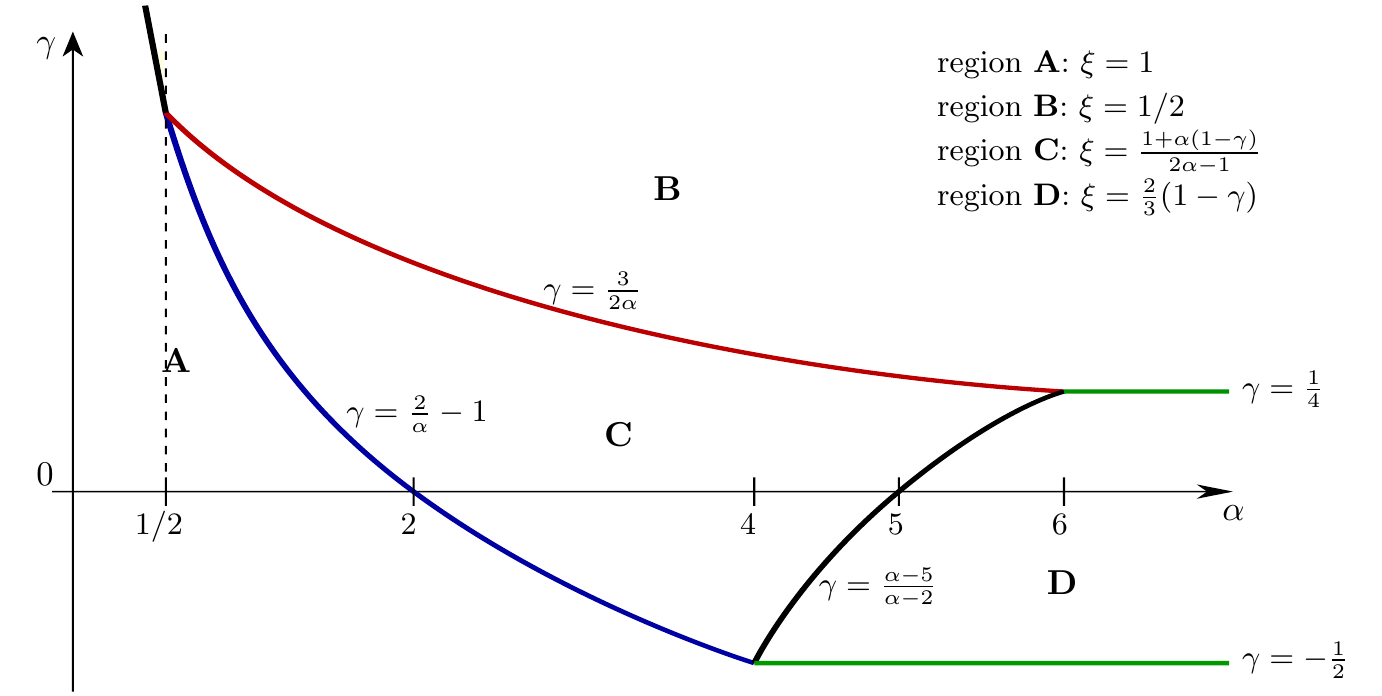}
\end{center}

\caption{\footnotesize  We identify four regions in the $(\alpha,\gamma)$ plane. 
	Region {\bf A} with $\ga<2$ is treated in \cite{AL11} and Region {\bf B} with $\ga>1/2$ in \cite{cf:DZ}. Regions {\bf C} and {\bf D} are still open, and the KPZ scaling relation $\chi=2\xi-1$ should hold in these two regions. Our main result is to settle the picture in the case $\ga\in(0,2)$. 
}
\label{fig1}
\end{figure}

This picture is far from being settled, and so far only the border cases where $\xi=1$ or $\xi=1/2$ have been studied:
Dey and Zygouras \cite{cf:DZ} proved that $\xi=1/2$ in the cases $\ga>6, \gamma =1/4$ and $\ga \in(1/2,6) , \gamma = 3/2\ga$; 
Auffinger and Louidor \cite{AL11} proved that $\xi=1$ for $\ga\in(0,2)$ and $\gamma = \frac{2}{\ga} -1$.
Here, we complete the picture in the case $\ga\in(0,2)$. For $\ga\in(1/2,2)$ we go beyond the  cases $\xi=1/2$ or $\xi=1$: we identify the correct order for the transversal fluctuations (they interpolate between $\xi=1/2$ and $\xi=1$), and we prove the convergence of $\log \bZ_{n,\gb_n}^{\go}$ in all possible intermediate disorder regimes---this proves Conjecture~1.7 in~\cite{cf:DZ}.  For $\ga<1/2$ we show that a sharp transition occurs on the line $\gamma = \frac{2}{\ga} -1$, between a regime where $\xi=1$ and a regime where $\xi=1/2$.

\section{Main results: weak-coupling limits in the case $\ga\in(0,2)$}

From now on, we consider the case of an environment $\go$ verifying \eqref{eq:DisTail} with $\ga\in(0,2)$.
For the inverse temperature, we will consider arbitrary sequences $(\gb_n)_{n\ge 1}$, but a reference example is $\gb_n = n^{-\gamma}$ for some $\gamma \in \bbR$.

For two sequences $(a_n)_{n\ge 1}, (b_n)_{n\ge 1}$, we use the notations $a_n \sim b_n$ if $\lim_{n\to\infty} a_n/b_n =1$, $a_n \ll b_n$ if $\lim_{n\to\infty} a_n/b_n =0$, and $a_n \asymp b_n$ if $0<\liminf a_n/b_n \le \limsup a_n/b_n <\infty$.

\subsection{First definitions and heuristics}

First of all, let us present a brief energy/entropy argument to justify what the correct transversal fluctuations of the polymer should be.
Let $F(x) = \bbP(\go\le x)$ be the disorder distribution, and define the function $m(x)$ by
\begin{equation}
\label{def:m}
m(x) := F^{-1} \big(1-\tfrac1x \big), \qquad \text{so we have }\ \bbP \big( \go >m (x) \big) = \frac1x  .
\end{equation}
Note that the second identity characterizes $m(x)$ up to asymptotic equivalence: we have that $m(\cdot)$ is a regularly varying function with exponent $1/\ga$.

Assuming that the transversal fluctuations are of order $h_n$ (we necessarily have $\sqrt{n}\le h_n \le n$), then the amount of weight collected by a path should be of order $m(n h_n)$ (it should be dominated by the maximal value of $\go$ in $[0,n] \times [-h_n, h_n]$).
On the other hand, thanks to  moderate deviations estimates for the simple random walk, the entropic cost of having fluctuations of order $h_n$ is roughly $h_n^2/n$ at the exponential level -- at least when $h_n \gg \sqrt{n \log n}$, see \eqref{LLT} below. It therefore leads us to define $h_n$ (seen as a function of $\gb_n$) up to asymptotic equivalence by the relation
\begin{equation}
\label{def:hn}
\gb_n m(n h_n) \sim h_n^2 /n \, .
\end{equation}

In the case $\gb_n = n^{-\gamma}$ and $\ga\in(1/2,2)$ we recover \eqref{def:xi}, that is we get that $h_n=n^{\xi+o(1)}$ with $\xi = \frac{1+\ga(1-\gamma)}{2\ga-1}$, which is in $(1/2,1)$ for $\gamma \in (\tfrac{2}{\ga}-1, \tfrac{3}{2\ga})$.
When $\ga \in(0, 1/2)$, there is no $h_n$ verifying \eqref{def:hn} with $\sqrt n \ll h_n \ll n$, leading to believe that intermediate transversal fluctuations (\textit{i.e.}\ $\xi\in(1/2,1)$) cannot occur. 
In the following, we  separate the  cases $\ga\in (1/2,2)$ and $\ga\in (0,1/2)$.

\subsection{A natural candidate for the scaling limit}
\label{sec:casealpha02}

Once we have identified in \eqref{def:hn} the scale $h_n$ for the transversal fluctuations, we are able to rescale both path trajectories and the field $(\go_{i,x})$, so that we can define the rescaled ``entropy'' and ``energy'' of a path, and the corresponding continous quantities.
The rescaled paths will be in the following set
\begin{equation}
\label{def:D}
\sD := \big\{ s: [0,1]\to \bbR\ ;\ s \text{ continuous and a.e.\ differentiable} \big\}\, ,
\end{equation}
and the (continuum) entropy of a path $s\in \sD$ will derive from the rate function of the moderate deviation of the 
simple random walk (see \cite{S67} or \eqref{LLT} below), \textit{i.e.}
\begin{equation}
\label{def:ContinuumEntropy}
\ent(s) = \frac12 \int_0^1 \big( s'(t) \big)^2 dt \qquad \text{for } s\in \sD.
\end{equation}

As far as the disorder field is concerned, we let $\cP:=\{(w_i,t_i,x_i)\}_{ i\geq 1}$ be a Poisson Point Process on  $[0,\infty)\times[0,1]\times\mathbb R $
of intensity $\mu(d w d t d x)=\frac{\alpha}{2} w^{-\alpha-1}\ind_{\{w>0\}}d w d t d x$.
For a quenched realization of $\cP$, the energy of a continuous path $s\in\sD$ is then defined by
\begin{equation}
\label{def:discrCont}
\pi(s) =\pi_{\cP}(s):=\sum_{(w,t,x)\in \cP} \, w \,\ind_{\{(t,x)\in s\}},
\end{equation}
where the notation $(t,x)\in s$ means that  $s_t=x$.

Then, a natural guess for the continuous scaling limit of the partition function is to consider an energy--entropy competition variational problem. 
For any $\gb\geq 0$ we let
\begin{equation}
\label{def:T}
\mathcal T_{\gb}  := \sup_{s\in \sD, \ent(s) <+\infty} \Big\{ \gb \pi(s) -\ent(s) \Big\} .
\end{equation}
This variational problem was originally introduced by Dey and Zygouras \cite[Conjecture~1.7]{cf:DZ}, conjecturing that it was well defined as long as  $\alpha\in (1/2,2)$ and that it was the good candidate for the scaling limit.
In \cite[Theorem~2.7]{cf:BT_ELPP} we show that the variational problem \eqref{def:T} is indeed well defined as long as $\ga\in(1/2,2)$. In Theorem \ref{thm:alpha>12} below, we prove the second part of \cite[Conjecture~1.7]{cf:DZ}.

\begin{theorem}[{\cite[Thm.~2.4]{cf:BT_ELPP}}]\label{thm:TbhatTb}
For $\ga\in (1/2,2)$ we have that $\cT_{\gb}\in (0,+\infty)$ for all $\gb>0$ a.s. On the other hand, for  $\ga\in(0,1/2]$ we have $\cT_{\gb} =+\infty$ for all $\gb>0$~a.s. 
\end{theorem}

Let us  mention here that in \cite{AL11}, the authors consider the case of transversal fluctuations of order $n$. The natural candidate for the limit is $\hat \cT_{\gb}$, defined analogously to \eqref{def:T} by $\hat \cT_{\gb} = 0$ for $\gb=0$, and for $\gb>0$
\begin{equation}
\label{def:hatT}
\hat \cT_{\gb} =\sup_{s\in \mathrm{Lip}_1} \Big \{ \pi(s) -  \frac1\gb \tent(s) \Big\}.  
\end{equation}
Here the supremum is taken over the set $\mathrm{Lip}_1$ of $1$-Lipschitz functions, and the entropy $ \tent(s)$ derives from the rate function of the large deviations for the simple random walk,~\textit{i.e.}
\[ \tent(s) = \int_0^1 e\big( s'(t) \big) dt \quad \text{with } e(x) = \tfrac12 (1+x) \log (1+x) + \tfrac12 (1-x)\log (1- x)\, .\]

\subsection{Main results I : the case $\ga\in(1/2,2)$}
\label{sec:resultsI}

Our first result deals with the transversal fluctuations of the polymer: we prove that $h_n$ defined in \eqref{def:hn} indeed gives the correct order for the transversal fluctuations.
\begin{theorem}\label{thm:fluctu}
Assume that $\ga\in(1/2,2)$, that $\gb_n m(n^{2}) \to 0$ and that $\gb_n m(n^{3/2}) \to +\infty$, and define $h_n$ as in \eqref{def:hn}: then $\sqrt{n} \ll h_n \ll n$.
Then, there are constants $c_1,c_2$ and $\nu>0$ such that for any sequences $A_n\ge 1$ we have for all $n\ge 1$
\begin{equation}
\label{eq:hscaling}
\bbP\left( \bP_{n,\gb_n}^{\go} \big( \max_{i\leq n} |S_i| \geq A_n \,   h_n \big) \geq n\, e^{- c_1 A_n^2  h_n^2/n} \right) \leq c_2\, A_n^{-\nu} \, .
\end{equation}
\end{theorem}

In particular, this proves that if $h_n$ defined in \eqref{def:hn} is larger than a constant times $\sqrt{n\log n}$, then $n e^{- c_1 A  h_n^2/n}$ goes to $0$ as $n\to\infty$ provided that $A$ is large enough: the transversal fluctuations are at most $A h_n$, with high $\bbP$-probability. On the other hand, if $h_n$ is much smaller than $\sqrt{n\log n}$, then this theorem does not give sharp information: we still find that the transversal fluctuations must be smaller than $A\sqrt{n\log n}$, with high $\bbP$-probability.
Anyway, in the course of the demonstration of our results, it will be clear that the main contribution to the partition function comes from trajectories with transversal fluctuations of order exactly $h_n$.

We stress that the cases $\gb_n m(n^2) \to \gb \in(0,+\infty]$ and \ $\gb_n m(n^{3/2}) \to \gb \in[0,\infty)$ have already been considered by Auffinger and Louidor \cite{AL11} and \ Dey and Zygouras \cite{cf:DZ} respectively: they find that the transversal fluctuations are of order $n$, resp.\ $\sqrt{n}$. We state their results below, see Theorem~\ref{thm:AL} and  Theorem~\ref{thm:DZ} 
respectively.
Our first series of results consist in identifying three new regimes for the transversal fluctuations ($\sqrt{n\log n} \ll h_n \ll n$, $h_n \asymp \sqrt{n\log n}$, and $\sqrt{n} \ll h_n \ll \sqrt{n\log n}$), that interpolate between the Auffinger Louidor regime ($h_n \asymp n$) and the  Dey Zygouras regime ($h_n \asymp \sqrt{n}$).
We now describe more precisely these five different regimes.

\subsubsection*{\underline{Regime 1}: transversal fluctuations of order $n$}
Consider the case where 

\begin{equation}
\label{reg1}
\tag{R1}
\gb_n n^{-1}m(n^2) \to \gb \in (0,\infty]\, ,
\end{equation}
which corresponds to having transversal fluctuations of order~$n$. Auffinger and Louidor showed that, properly rescaled, $\log \bZ_{n,\gb_n}^{\go}$ converges to $\hat \cT_\gb$ defined in \eqref{def:hatT}.

\begin{theorem}[Regime 1, \cite{AL11}]
\label{thm:AL}
Assume $\ga\in(0,2)$, and consider $\gb_n$ such that \eqref{reg1} holds.
Then we have the following convergence
\[
\frac{1}{\beta_n  m(n^2)} \log \bZ_{n,\gb_n}^{\go}   \stackrel{\rm (d)}{\longrightarrow} \hat \cT_{\gb} \quad \text{as } n\to\infty ,
\]
with $\hat \cT_{\gb}$  defined  in \eqref{def:hatT}.
For $\ga\in [1/2,2)$, we have  $\hat \cT_{\gb}>0$ a.s. for all $\gb>0$.
\end{theorem}

\subsubsection*{\underline{Regime 2}: $ \sqrt{n \log n} \ll h_n \ll n$}
Consider the case when

\begin{equation}
\label{reg2}
\tag{R2}
\gb_n n^{-1}m(n^2) \to 0 \quad \text{ and } \quad \gb_n  \log n^{-1} m(n^{3/2} \sqrt{\log n}) \to \infty\, ,
\end{equation}
which corresponds to having transversal fluctuations $\sqrt{n \log n} \ll h_n \ll n$, see \eqref{def:hn}. We find that, properly rescaled, $\log \bZ_{n,\gb_n}^{\go}$ converges to $\cT_1$ defined in \eqref{def:T}---this proves Conjecture 1.7 in \cite{cf:DZ}.

\begin{theorem}[Regime 2]
\label{thm:alpha>12}
Assume that $\ga\in(1/2,2)$, and consider $\gb_n$ such that \eqref{reg2} holds. 
Defining $h_n$ as in \eqref{def:hn}, then $\sqrt{n\log n}\ll h_n \ll n$, and we have
\begin{equation}
\label{eq:hscaling3}
\frac{1}{\beta_n m(n h_n)} \Big( \log\bZ^\omega_{n,\beta_n} - n \gb_n \bbE[\go] \ind_{\{\ga \ge 3/2\}}  \Big) \  \stackrel{\rm (d)}{\longrightarrow} \ \cT_1 \quad \text{as } n\to\infty,
\end{equation}
with $\cT_1$  defined in \eqref{def:T}.
\end{theorem}
We stress here that we need to recenter $\log\bZ^\omega_{n,\beta_n}$ by $n \gb_n \bbE[\go] $ only when necessary, that is when $n/m(nh_n)$ does not go to $0$: in terms of the picture described in Figure \ref{fig1}, this can happen only when $\gamma \ge 4-2\ga$, and in particular when $\ga \ge 3/2$ (this is stressed in the statement of the theorem).

\subsubsection*{\underline{Regime 3}: $h_n \asymp \sqrt{n\log n}$}
Consider the case
\begin{equation}
\label{reg3}
\tag{R3}
\gb_n  \log n^{-1} m(n^{3/2} \sqrt{\log n}) \to \gb \in(0,\infty)\, ,
\end{equation}
which from \eqref{def:hn} corresponds to  transversal fluctuations $h_n\sim \gb^{1/2}
\sqrt{n\log n}$, see \eqref{def:hn}. We find the correct scaling of $\log \bZ_{n,\gb_n}^{\go}$, which can be of two different natures (and go to $+\infty$ or~$0$), see Theorems~\ref{thm:cas3}-\ref{thm:cas3bis} below.

We first need to introduce a few more notations. For a quenched continuum energy field $\cP$ (as defined in Section~\ref{sec:casealpha02}), we define for a path $s$ the number of weights $w$ it collects:
\begin{equation}\label{def:N}
N(s) : = \sum_{(w,t,x)\in \cP} \ind_{\{ (t,x) \in s\}} \, .
\end{equation}
Then, we define a new energy-entropy variational problem: for a fixed realization of $\cP$, define for any $k\ge 1$
\begin{equation}
\label{def:tildeT}
\begin{split}
\tilde \cT_{\gb}^{(k)}   = \tilde \cT_{\gb}^{(k)} (\cP)&:= \sup_{s\in \sD , N(s) =k} \Big\{ \pi(s) - 
\ent(s) - \frac{k}{2\gb}  \Big\}, \\
\text{and }\quad   \tilde \cT_{\gb}^{(\ge r)}&:=\sup_{k\ge r} \tilde \cT_{\gb}^{(k)} .
\end{split}
\end{equation}
When $r=0$ we denote by $\tilde\cT_{\gb}$ the quantity $\tilde \cT_{\gb}^{(\ge 0)}$.
In Proposition \ref{propTtilde} below, we prove that these quantities are well defined, and that there exists $\gb_c = \gb_c(\mathcal P)\in (0,\infty)$ such that $ \tilde \cT_{\gb} \in (0,\infty)$ if $\gb>\gb_c$ and $\tilde \cT_{\gb} =0$ if $\gb<\gb_c$. 

\begin{theorem}[Regime 3-a]
\label{thm:cas3}
Assume that $\ga\in (1/2,2)$, and consider $\gb_n$ such that \eqref{reg3} holds. 
Then from \eqref{def:hn} we have $h_n \asymp \sqrt{n\log n}$,  and
\begin{equation}
\frac{1}{\gb_n m(nh_n)}  \Big( \log\bZ^\omega_{n,\beta_n} - n \gb_n \bbE[\go] \ind_{\{\ga \ge 3/2\}}  \Big) \stackrel{\rm (d)}{\longrightarrow}  \tilde \cT_{\gb}  \quad \text{as } n\to\infty \, .
\end{equation}
(Recall that $\gb_n m(n h_n)\sim h_n^2/n \sim \gb \log n$.)
\end{theorem}

If $\tilde \cT_{\gb} >0$ ($\gb>\gb_c$) the scaling limit is therefore well identified, and $\log \bZ_{n,\gb_n}^{\go}$ (when recentered) grows like $\gb \tilde \cT_{\gb}\log n$ with $ \gb\tilde \cT_{\gb}>0$. On the other hand, if $\tilde \cT_{\gb} =0$, then the above theorem gives only a trivial limit.
By an extended version of Skorokhod representation theorem \cite[Corollary 5.12]{K97}, one can couple the discrete environment and the continuum field $\cP$ in order to obtain an almost sure convergence in Theorem~\ref{thm:cas3} above. Hence, it makes sense to work conditionally on $\tilde \cT_\gb^{\ge 1}<0$ ($\gb<\gb_c$), even at the discrete level.
Our next theorem says that for $\gb<\gb_c$, $\log \bZ_{n,\gb_n}$ decays polynomially, with a random exponent $\gb \tilde\cT_{\gb}^{(\ge 1)} \in (-1/2,0)$.

\begin{theorem}[Regime 3-b,  $\tilde \cT_{\gb} =0$, $\gb<\gb_c$]
\label{thm:cas3bis}
Assume that $\ga\in(1/2,2)$ and that \eqref{reg3} holds. Then,  conditionally on $\{\tilde \cT_{\gb}^{(\ge 1)} <0 \}$ (i.e.\ $\gb<\gb_c$), 
\[   \frac{1}{\gb_n m(n h_n)}\log   \Big( \log\bZ^\omega_{n,\beta_n} - n \gb_n \bbE[\go \ind_{\{\go \le 1/\gb_n\}}] \ind_{\{\ga \ge 1\}}  \Big)  \stackrel{\rm (d)}{\longrightarrow} \tilde \cT_{\gb}^{(\ge 1)}   \quad \text{as } n\to\infty \, .\]
\end{theorem}
Recalling that $\gb_n m(n h_n)\sim h_n^2/n \sim \gb \log n$, we note that  $\exp(\gb \tilde \cT_{\gb}^{(\ge 1)} \log n)$ goes to $0$ as a (random) power $\gb \tilde \cT_{\gb}^{(\ge 1)}$ of $n$, with $\gb \tilde \cT_{\gb}^{(\ge 1)} \in (-1/2,0)$.

\subsubsection*{\underline{Regime 4}: $\sqrt{n} \ll h_n \ll \sqrt{n\log n}$}
Consider the case
\begin{equation}
\label{reg4}
\tag{R4}
\gb_n m(n^{3/2}) \to \infty \quad  \text{ and } \quad \gb_n  \log n^{-1} m(n^{3/2} \sqrt{\log n} ) \to 0 \, ;
\end{equation}
which corresponds to having transversal fluctuations $\sqrt{n} \ll h_n \ll \sqrt{n\log n}$, see \eqref{def:hn}. 
Let us define 
\begin{equation}
\label{def:W}
W_{\gb} := \tilde \cT_{\gb}^{(1)} + \frac{1}{2\gb}:= \sup_{(w,x,t) \in \cP} \Big\{ w -  \frac{x^2}{2 \gb t} \Big\}  \, ,
\end{equation}
which is a.s. positive and finite if $\ga\in(1/2,2)$, see Proposition \ref{prop:W} below.
\begin{theorem}[Regime 4]
\label{thm:cas4}
Assume that $\ga\in (1/2,2)$, and 
consider $\gb_n$ such that \eqref{reg4} holds. Defining $h_n$ as in \eqref{def:hn},
then $\sqrt{n}\ll h_n \ll \sqrt{n\log n}$, and we have
\[
\frac{1}{\gb_n m(nh_n)}\log \bigg( \sqrt{n}    \Big( \log\bZ^\omega_{n,\beta_n} - n \gb_n \bbE[\go \ind_{\{\go \le 1/\gb_n \}} ] \ind_{\{\ga \ge 1\}}  \Big)  \bigg)  \stackrel{\rm (d)}{\longrightarrow}  W_1 
\]
as $n\to\infty$.
\end{theorem}
Recalling that $\gb_n m(nh_n) \sim h_n^2/n \ll \log n$, we note that $\exp\big( W_1 h_n^2/n \big) $ goes to infinity (at some random rate), but slower than any power of $n$.

\subsubsection*{\underline{Regime 5}: transversal fluctuations of order $\sqrt{n}$}
Consider the case
\begin{equation}
\label{reg5}
\tag{R5}
\gb_n m(n^{3/2}) \to \gb \in [0,\infty) \, ;
\end{equation}
this corresponds to having transversal fluctuations $h_n$ of order $\sqrt n$.
Here, we state one of the results obtained by Dey and Zygouras, \cite[Theorem 1.4]{cf:DZ}. 

\begin{theorem}[Regime 5, \cite{cf:DZ}]
\label{thm:DZ}
Assume that $\ga\in(1/2,2)$, and consider $\gb_n$ such that \eqref{reg5} holds, that is $\gb_n m(n^{3/2}) \to \gb \in [0,\infty)$. Then
\[ \frac{\sqrt{n}}{ \gb_n m(n^{3/2})}  \Big( \log \bZ_{n,\gb_n} - n \gb_n \bbE\big[ \go \ind_{\{\go\le m(n^{3/2})\}} \big] \ind_{\ga\ge 1}\Big) \stackrel{\rm (d)}{\longrightarrow}   2 \cW_\gb^{(\ga)} \quad \text{as } n\to\infty \, . \]
Here, $\cW_{\gb}^{(\ga)}$ is  some specific $\ga$-stable random variable (defined in \cite[p.~4011]{cf:DZ}).
\end{theorem}

\subsubsection*{Some comments about the different regimes}
The  regimes 2-3-4 have different behavior due to the different behaviors for the \emph{local moderate deviation}, see \cite[Theorem~3]{S67}. We indeed have that for ${\sqrt n \ll h_n \ll n}$
\begin{equation}
p_n(h_n) :=\bP(S_n = h_n) = \frac{c}{\sqrt{n}}\exp\Big( - (1+o(1))   \, \frac{h_n^2}{2 n}  \Big) \, ,
\label{LLT}
\end{equation}
so that we identify three main possibilities:
if $h_n \ll \sqrt{n\log n}$, then $p_n(h_n) = n^{-1/2 +o(1)}$;
if $h_n \sim c \sqrt{n\log n} $ then $p_n(h_n)  =n^{-(c^2+1)/2 +o(1)}$; 
if $h_n \gg  \sqrt{n\log n}$ then $p_n(h_n) = e^{-(1+o(1)) h_n^2/n}$ which decays faster than any power of $n$.

This is actually reflected in the behavior of the partition function. Let us denote $\bar \bZ_{n,\gb_n}^{\go} =  e^{-  n \gb_n C_\ga}\times  \bZ_{n,\gb_n}^{\go}$ be
the renormalized (when necessary) partition function. We recall that $C_\ga$ is equal either to $\bbE[\go] \ind_{\{\ga\ge 3/2\}}$ (Regime 2 and 3-a) or to 
$\bbE[\go \ind_{\{\go\le 1/\gb_n\}}] \ind_{\ga\ge 1}$ (Regime 3-b and 4).
Then we have

\textbullet\ In Regimes 1 and 2, transversal fluctuations are $h_n \gg \sqrt{n\log n}$, and $\bar \bZ_{n,\gb_n}$ grows faster than any power of $n$: roughly, it is of order $e^{\gb \hat \cT_{\gb} n}$ in Regime 1 (for $\gb<\infty$), and of order $e^{\cT_1 h_n^2/n}$ in Regime 2.

\textbullet\ In Regime 3, transversal fluctuations  are $h_n\asymp \sqrt{n\log n}$, and $\bar\bZ_{n,\gb_n}$ goes to infinity polynomially in Regime~3-a, and it goes to $1$ with a polynomial correction in Regime~3-b. This could be summarized as $\bar \bZ_{n,\gb_n}\approx 1+ n^{ \gb \tilde \cT_{\gb}^{(\ge 1)} }$, with $\gb \tilde \cT_{\gb}^{(\ge 1)} >-1/2$: the transition between regime 3-a and 3-b occurs as $\gb \tilde \cT_{\gb}^{(\ge 1)}$ changes sign, at $\gb=\gb_c$ (note that $\gb\tilde \cT_{\gb}^{(\ge 1)}$ keeps a mark of the local limit theorem, see \eqref{def:tildeT} and \eqref{LLT}).

\textbullet\ In Regime 4, $\bar \bZ_{n,\gb_n}$ goes to $1$ with a correction of order $n^{-1/2} \times e^{W_{1} h_n^2/n}$, with $e^{W_{1} h_n^2/n}$ going to infinity slower than any power of $n$: this  corresponds to the cost for a trajectory to visit a single site, at which the supremum in $W_1$ is attained. In Regime 5, $\bar \bZ_{n,\gb_n}^{\go}$ goes to $1$ with a correction of order $ n^{-1/2}$.

%

\subsection{Main results II : the case $\ga \in(0, 1/2)$}
\label{sec:resultsII}

In this case, since we have $n^{-1} m(n^2) /m(n^{3/2}) \to \infty$, there is no sequence $\gb_n$ such that  $\gb_n n^{-1} m(n^2) \to 0$ and $\gb_n m(n^{3/2}) \to +\infty$.
First of all, Theorem \ref{thm:AL} already gives a result, but a phase transition has been identified in \cite{AL11,T14} when $\ga\in(0,1/2)$.

\begin{theorem}[\cite{AL11,T14}]
\label{thm:T14}
When $\ga\in (0,1/2)$, $\hat \cT_{\gb}$ defined in \eqref{def:hatT} undergoes a phase transition: there exists some $\hat \gb_c = \hat \gb_c(\cP)$ with $\hat \gb_c \in(0,\infty)$ $\bbP$-a.s., such that $ \hat \cT_{\gb} =0$ if $\gb\le \hat \gb_c$ and $ \hat \cT_{\gb}>0$ if $\gb> \hat \gb_c$. 
\end{theorem}
The fact that $\hat \cT_{\hat \gb_c} =0$ was not noted in \cite{AL11,T14}, but simply comes from the (left) continuity of $\gb \mapsto \hat \cT_{\gb}$ (the proof is identical to that for $\gb \mapsto \cT_{\gb}$, see \cite[Section 4.5]{cf:BT_ELPP}).

In view of Theorem \ref{thm:AL}, the scaling limit of $\log \bZ_{n,\gb_n}^{\go}$ is identified when $\hat \cT_{\gb}>0$,  and it is trivial when $ \hat \cT_{\gb} =0$. Again, by an extended version of Skorokhod representation theorem \cite[Corollary 5.12]{K97}, we can obtain an almost sure convergence in Theorem~\ref{thm:AL}. Hence, it makes sense to work conditionally on $\hat \cT_{\gb}>0$ or $\hat \cT_\gb =0$, even at the discrete level.
We show here that only two regimes can hold:
if $\hat \cT_{\gb} >0$, then fluctuations are of order~$n$, and properly rescaled,  $\log \bZ_{n,\gb_n}^{\go}$ converges to $\hat \cT_{\gb}$ (this is Theorem \ref{thm:AL});
if $\hat \cT_{\gb} =0$, then fluctuations are of order $\sqrt{n}$, and properly rescaled, $\log \bZ_{n,\gb_n}^{\go}$ converges in distribution (conditionally on $\hat \cT_{\gb} =0$).

\begin{theorem}
\label{thm:alpha<12}
Assume $\ga\in(0,1/2)$, and consider $\gb_n$ with $\gb_n n^{-1} m(n^{2}) \to \gb\in[0,+\infty)$. Then, on the event $\{ \hat \cT_{\gb}=0\}$ ($\gb\le \hat \gb_c<\infty$), transversal fluctuations are of order $\sqrt{n}$. More precisely,  for any $\gep>0$, there exists some $c_0,\nu>0$ such that, 
for any  sequence $C_n>1$ we have
\begin{equation}
\label{eq:alpha<12_one}
\mathbb P\bigg(\bP_{n,\gb_n}^{\go} \Big( \max_{i\leq n} |S_i| \ge C_n \sqrt{n} \Big) \ge e^{-c_0 C_n^2\wedge n^{1/2}}\ \Big| \  \hat \cT_{\gb} =0\bigg)\leq \gep . 
\end{equation}
Moreover, conditionally on $\{\hat \cT_{\gb}=0\}$, we have that
\begin{equation}
\label{eq:alpha<12_two}
\frac{\sqrt{n}}{\gb_n m(n^{3/2})}  \log \bZ^\omega_{n,\beta} \stackrel{\rm (d)}{\longrightarrow}   2 \cW_0^{(\ga)} \, , \qquad \text{as } n\to+\infty \, .
\end{equation}
where $\cW_0^{(\ga)}:= \int_{\bbR_+\times \bbR \times [0,1]} w \rho(t,x) \cP(\dd w ,\dd x,\dd t)$ with $\cP$ a realization of the Poisson Point Process defined in Section~\ref{sec:casealpha02}, and $\rho(t,x) = ({2\pi t})^{-1/2} e^{-x^2/2t}$ is the Gaussian Heat kernel. 
\end{theorem}
Note that  $\cW_0^{(\ga)}$ is well defined  and has an $\ga$-stable distribution, with explicit characteristic function, see Lemma~1.3 in \cite{cf:DZ}.
Theorem~\ref{thm:alpha<12} therefore shows that, when $\ga<1/2$, a very sharp phase transition occurs on the line $\gb_n \sim \gb n/m(n^2)$: for $\gb \le \hat \gb_c$, transversal fluctuations are of order $\sqrt{n}$ whereas for $\gb> \hat \gb_c$ they are of order $n$.

\subsection{Some comments and perspectives}
\label{sec:comments}
We now present some possible generalizations, and we discuss some open questions.

\subsubsection{About the case $\ga=1/2$}
We excluded above the case $\ga=1/2$. In that case,  both $n^{-1} m(n^{2})$ and $m(n^{3/2})$ are regularly varying with index $3$, and there are mostly two possibilities.

(1) If $\frac{n^{-1}m(n^2)}{m(n^{3/2})} \to 0$ (for instance if $L(x) = e^{- (\log x)^{b}}$ for some $b\in (0,1)$),  there are sequences $(\gb_n)_{n\ge 1}$ with $\gb_n n^{-1} m(n^2)\to 0$ and $\gb_n m(n^{3/2}) \to +\infty$.
The situation should be similar to that of Section~\ref{sec:resultsI}: there should be five regimes, with transversal fluctuations $h_n$ interpolating between $\sqrt{n}$ and~$n$.

(2) If $\frac{n^{-1}m(n^2)}{m(n^{3/2})} \to c  \in (0,\infty]$ (for instance if $L(n) = (\log x)^b$ for some~$b$), there is no sequence 
$(\gb_n)_{n\ge 1}$ with $\gb_n n^{-1} m(n^2)\to 0$ and $\gb_n m(n^{3/2}) \to +\infty$. Then, the situation should be similar to that of Section~\ref{sec:resultsII}: there should be only two regimes, with transversal fluctuations either $\sqrt{n}$ or $n$.

%

\subsubsection{Toward the case $\ga\in (2,5)$}

When $\ga\in(2,5)$ (more generally in region $\mathbf{C}$ in Figure \ref{fig1}), an important difficulty is to find the correct centering term for $\log \bZ_{n,\gb_n}^{\go}$.  Another problem is that the variational problem $\cT_{\gb}$ defined in \eqref{def:T} is $\cT_{\gb}=+\infty$ a.s., since paths that collect many small weights bring an important contribution to $\cT_{\gb}$.
The main objective is therefore to prove a result of the type: there exists a function $f(\cdot)$ such that, for  $\ga\in (2,6)$ and  any $\gb_n$ in region $\mathbf{C}$ of Figure~\ref{fig1}
\[ \frac{1}{\gb_n m(nh_n)}  \Big( \log \bZ_{n,\gb_n}^{\go} -  f(\gb_n) \Big) \stackrel{({\rm d})}{\longrightarrow} \check \cT_{1} \, ,\]
with $h_n$ defined as in \eqref{def:hn} and where $\check \cT_{1}$ is somehow a ``recentered'' version of the variational problem \eqref{def:T} (that is in which the contribution of the small weights has been canceled out).
The difficulties are however serious: one needs (i)~to identify the centering term $f(\gb_n)$, (ii) to make sense of the variational problem $\check \cT_1$.

\subsubsection{Path localization}

We mention that in \cite{AL11}, Auffinger and Louidor show some path localization: they  prove that, under $\bP_{n,\gb_n}^{\go}$, path trajectories concentrate around the (unique) maximizer $\gamma^*_{n,\gb_n}$ of the discrete analogue of the variational problem \eqref{def:hatT}, see Theorem~2.1 in \cite{AL11}; moreover this maximizer $\gamma^*_{n,\gb_n}$ converges in distribution to the (unique) maximizer $\hat\gamma_{\gb}^*$ of the variational problem \eqref{def:hatT}. 
This could theoretically be done in our setting: in \cite[Section 4.6]{cf:BT_ELPP} we prove the existence and uniqueness of the maximizer of the continuous variational problem \eqref{def:T}. Then similar techniques to those of \cite{AL11} could potentially be used, and one would obtain a result analogous to \cite[Thm.~2.1]{AL11}

\subsubsection{Higher dimensions}

Similarly to \cite{AL11}, our methods should work in any dimension $1+d$ (one temporal dimension, $d$ transversal dimensions). The relation \eqref{def:hn} is replaced by $\gb_n m(n h_n^d) \sim h_n^2/n$: for paths with transversal scale $h_n$, the energy collected should be of order $\gb_n m ( n h_n^d)$ while the entropy cost should remain of order $h_n^2/n$, at the exponential level.
For $\ga\in (0,1+d)$, and choosing $\gb_n = n^{-\gamma}$, we should therefore find that in dimension $d$ a similar picture to Figure~\ref{fig1} hold:

\smallskip
\begin{center}
\begin{minipage}{0.37\textwidth}
	\begin{center}
		Case $\ga \in (0,d/2)$\\[0.1cm]
		\begin{tabular}{|c|c|}
			\hline
			\footnotesize $\gamma < \frac{1+d}{\ga} -1$ &  \footnotesize $\gamma > \frac{1+d}{\ga} -1$  \\
			\hline 
			\footnotesize $\xi=1$ & \footnotesize $\xi =1/2$\\
			\hline
		\end{tabular}
	\end{center}
\end{minipage}
\quad
\begin{minipage}{0.57\textwidth}
	\begin{center}
		Case $\ga\in (d/2 ,1+d)$\\[0.1cm]
		\begin{tabular}{|c|c|c|}
			\hline
			\footnotesize$\gamma \le \frac{1+d}{\ga} -1$ &\footnotesize $\frac{1+d}{\ga} -1 <\gamma < \frac{2+d}{2\ga}$ & \footnotesize$\gamma\ge \frac{2+d}{2\ga}$  \\
			\hline 
			\footnotesize$\xi=1$ & \footnotesize$\xi = \frac{1+(1-\gamma)\ga}{2\ga -d}\in (\frac12,1)$  & \footnotesize$\xi =\frac12$\\
			\hline
		\end{tabular}
	\end{center}
\end{minipage}
\end{center}


\subsection{Organization of the rest of the paper}
We present an overview of the main ideas used in the paper, and describe how the proofs are organized.

\smallskip
$\ast$ In Section~\ref{sec:3}, we recall some of the notations and results of the Entropy-controlled Last-Passage Percolation (E-LPP) developed in \cite{cf:BT_ELPP}, which will be a central tool for the rest of the paper. In particular, we introduce a discrete energy/entropy variational problem \eqref{def:discreteELPP} (which is the discrete counterpart of \eqref{def:T}), and state its convergence toward \eqref{def:T} in Proposition~\ref{prop:ConvVP}.

\smallskip
$\ast$ In Section~\ref{sec:fluctu}, we prove Theorem \ref{thm:fluctu}, identifying the correct transversal fluctuations. 
In order to make our ideas appear clearer, we first treat the case when no centering is needed (\textit{i.e.}\ $\ga<3/2$) in Section~\ref{sec:nocentering}.
In Section~\ref{sec:centeringneeded} we adapt the proof to the case where it is needed. 
In the first case, we use a rough bound $\bP_{n,\gb_n}^{\go} \big( \max_{i\le n}|S_i| \ge A_n h_n \big) \le \bZ_{n,\gb_n}^{\go} \big( \max_{i\le n}|S_i| \ge A_n h_n \big) $,  the second term being the partition function restricted to trajectories with $ \max_{i\le n}|S_i| \ge A_n h_n $.
The key idea is to decompose this quantity into sub-parts where trajectories have a ``fixed'' transversal fluctuation
\[ \bZ_{n,\gb_n}^{\go} \big( \max_{i\le n}|S_i| \ge A_n h_n \big) =  \sum_{k=\log_2 A_n +1}^{\log_2(n/h_n)} \bZ_{n,\gb_n}^{\go} \Big(  \max_{i\le n}|S_i|  \in [2^{k-1} h_n, 2^k h_n) \Big) \, . \]
Then, we control each term separately. Forcing the random walk to reach the scale $2^{k-1} h_n$ has an entropy cost $\exp ( - c 2^{2k} h_n^2/n  )$ so  we need to understand if the partition function, when restricted to trajectories with $\max_{i\le n} |S_i| \le 2^k h_n$, compensates this cost (cf. \eqref{eq:CauchySchwarz}): we need to estimate the probability of having $ \bZ_{n,\gb_n}^{\go}(\max_{i\le n} |S_i| \le 2^k h_n) \ge e^{c 2^{2k} h_n^2/n}$. This is the purpose of  Lemma~\ref{lem:Zmax}, which is the central estimate of this section, and which tediously uses  estimates derived in~\cite{cf:BT_ELPP} (in particular Proposition~2.6).

\smallskip
$\ast$ In Section~\ref{secProofeq:hscaling}, we consider Regimes 2 and 3-a, and we prove Theorems~\ref{thm:alpha>12}-\ref{thm:cas3}. The proof is decomposed into three steps. In the first step (Section~\ref{sec:reduction2}), we use Theorem~\ref{thm:fluctu} in order to restrict the partition function to path trajectories that have transversal fluctuations smaller than $A h_n$ (for some large $A$ fixed).
In a second step (Section~\ref{largeweights2}), we show that we can keep only the largest weights in the box of height $A h_n$ (more precisely a finite number of them), the small-weights contribution being negligible.
Finally, the third step (Section~\ref{sec:2-Step3}) consists  in proving the convergence of the large-weights partition function, and relies on the convergence of the discrete variational problem of Section~\ref{sec:3}.

\smallskip
$\ast$ In Section~\ref{sec:3b4}, we treat Regime 3-b and Regime 4, and we prove Theorems~\ref{thm:cas3bis}-\ref{thm:cas4}. We proceed in four steps. In the first step (Section~\ref{sec:4-step1}), we again use Theorem~\ref{thm:fluctu} to restrict the partition function to trajectories with transversal fluctuations smaller than $A \sqrt{n\log n}$ (for some large $A$ fixed).
The second step (Section~\ref{largeweights3}) consists in showing that one can restrict to large weights. In the third step (Section~\ref{reductionlogZ}), we observe that since we consider a regime  $\log \bZ_{n,\gb_n}^{\go} \to 0$, it is equivalent to studying the convergence of $\bZ_{n,\gb_n}^{\go}-1$: we  reduce to showing the convergence of a finite number of terms of the polynomial chaos expansion of $\bZ_{n,\gb_n}^{\go}-1$, see Lemmas~\ref{lem:cas3-mainterm}-\ref{lem:cas4-mainterm}.
We prove this convergence in a last step: in Section~\ref{sec:prooflem3b}, we show the convergence in Regime 3-b (Lemma \ref{lem:cas3-mainterm}), relying on the convergence of a discrete variational problem. In Section~\ref{sec:regime4}, we show the convergence in Regime 4 (Lemma~\ref{lem:cas4-mainterm}), which is slightly more technical since we first need to reduce to trajectories with transversal fluctuations of order $h_n \ll \sqrt{n\log n}$.

\smallskip
$\ast$ In Section~\ref{sec:alpha12}, we consider the case $\ga\in(0,1/2)$, and we prove Theorem~\ref{thm:alpha<12}.
First, in Section~\ref{sec:fluctualpha12}, we prove \eqref{eq:alpha<12_one} \textit{i.e.}\ that there cannot be intermediate transversal fluctuations between $\sqrt{n}$ and $n$. We use mostly the same ideas as in Section~\ref{sec:fluctu},  decomposing the contribution to the partition function according to the scale of the path, and controlling the entropic cost vs.\ energy reward for each term. Here, some simplifications occur: one can bound the maximal energy collected by a path at a given scale by the sum of all weights in a box containing the path,  this sum being roughly dominated by the maximal weight in the box (this is true for $\ga<1$).
We then turn to the convergence of the partition function in Section~\ref{sec:convalpha12}.
The idea is similar to that of \cite[Section~5]{cf:DZ}, and consists in several steps: first we reduce the partition function to trajectories that stay at scale $\sqrt{n \log n}$; then we perform a polynomial chaos expansion of $\bZ_{n,\gb_n}^{\go}-1$ and we show that only the first term contributes; finally, we prove the convergence of the main term, see Lemma~\ref{lem:convergence},  showing in particular that the main contribution comes from trajectories that stay at scale~$\sqrt{n}$.

%
\section{Discrete energy-entropy variational problem}
\label{sec:3}
We introduce here a few necessary notations, and state some useful results from \cite{cf:BT_ELPP}.
Let us consider a box $\Lambda_{n,h} = \llbracket 1,n \rrbracket \times \llbracket -h,h \rrbracket$.
For any set $\Delta \subset \Lambda_{n,h}$, we define the (discrete) energy collected by $\Delta$ by
\begin{equation}
\label{energydiscrete}
\Omega_{n,h} (\Delta) := \sum_{(i,x) \in \Delta} \go_{i,x} \, .
\end{equation}
We can also define the (discrete) entropy of a finite set $\Delta=\big\{ (t_i,x_i) ; 1\le i\le j \big\} \subset \bbR^2$ with $|\Delta|=j\in \mathbb N$ and with $0\le t_1\le  t_2\le \cdots \le t_j$ (with $t_0=0,x_0=0$)
\begin{equation}
\label{def:entdelta}
\ent(\Delta) := \frac12 \sum_{i=1}^j \frac{(x_i-x_{i-1})^2}{t_i-t_{i-1}} \, ,
\end{equation}
By convention, if  $t_i=t_{i-1}$ for some $i$, then $\ent(\Delta)=+\infty$.
The set $\Delta$ is seen as a set of points a (continuous or discrete) path has to go through:
if $\Delta \subset \bbN \times \bbZ$ a standard calculation gives that $\bP(\Delta \subset S) \le e^{-\ent(\Delta)}$  ($\Delta\subset S$ means that $S_{t_i}=x_i$ for all $i\le |\Delta|$), where we use that $\bP(S_i=x) \le e^{- x^2/2i}$ by a standard Chernoff bound argument.

We are interested in the (discrete) variational problem, analogous to \eqref{def:T}
\begin{equation}
\label{def:discreteELPP}
T_{n,h}^{\gb_{n,h}} := \max_{ \Delta \subset \Lambda_{n,h}} \big\{  \gb_{n,h} \Omega_{n,h} (\Delta) - \ent(\Delta)  \big\} \, ,
\end{equation}
with $\gb_{n,h}$ some function of $n,h$ (soon to be specified).

We may rewrite the disorder in the region $\Lambda_{n,h}$, using the \emph{ordered statistic}: we let 
$M_r^{(n,h)}$ be the $r$-th largest value of $(\omega_{i,x})_{(i,x)\in \Lambda_{n,h}}$ and $Y_r^{(n,h)}\in \Lambda_{n,h}$ its position.
In such a way 
\begin{equation}
(\omega_{i,j})_{(i,j)\in \Lambda_n}{=}(M_r^{(n,h)},Y_r^{(n,h)})_{r=1}^{|\Lambda_{n,h}|} \, .
\end{equation}
In the following we refer to $(M_r^{(n,h)})_{r=1}^{|\Lambda_{n,h}|}$ as the \emph{weight} sequence. 
Note also that $(Y_r^{(n,h)})_{r=1}^{|\Lambda_{n,h}|}$ is simply a random permutation of the points of $\Lambda_{n,h}$. 

The ordered statistics allows us to redefine the energy collected by a set $\Delta \subset \Lambda_{n,h}$, and its contribution by the first $\ell$ weights (with $1\le \ell \le |\Lambda_{n,h}|$) by
\begin{equation}
\label{def:Omega}
\Omega_{n,h}^{(\ell)} (\Delta)  := \sum_{r=1}^{\ell} M_r^{(n,h)} \ind_{\{ Y_r^{(n,h)} \in \Delta\}}\, , \qquad \Omega_{n,h} (\Delta)  := \Omega_{n,h}^{(|\Lambda_{n,h}|)} (\Delta)  \, .
\end{equation}
We also set $ \Omega_{n,h}^{(>\ell)} (\Delta) :=  \Omega_{n,h} (\Delta)-  \Omega_{n,h}^{(\ell)}(\Delta)$.
We then define analogues of \eqref{def:discreteELPP} with a restriction to the $\ell$ largest weights, or beyond the $\ell$-th  weight
\begin{equation}
\label{def:discrELPPell}
\begin{split}
T_{n,h}^{\gb_{n,h},(\ell)} &:= \max_{ \Delta \subset \Lambda_{n,h}} \big\{  \gb_{n,h} \Omega_{n,h}^{(\ell)} (\Delta) - \ent(\Delta)  \big\} \, ,\\
T_{n,h}^{\gb_{n,h},(>\ell)} &:= \max_{ \Delta \subset \Lambda_{n,h}} \big\{  \gb_{n,h} \Omega_{n,h}^{(>\ell)} (\Delta) - \ent(\Delta)  \big\} \, .
\end{split}
\end{equation}
Estimates on these quantities are given in \cite[Prop.~2.6]{cf:BT_ELPP} (most useful in Section~\ref{sec:fluctu}). The following convergence in distribution is given in \cite[Thm.~2.7]{cf:BT_ELPP}, and plays a crucial role for the convergence  in Theorems \ref{thm:alpha>12}---\ref{thm:cas4} .
\begin{proposition}\label{prop:ConvVP}
Suppose that $\frac{n}{h^2}\gb_{n,h} m(nh) \to \nu \in[0,\infty)$ as $n,h\to\infty$. For every $\ga\in (1/2,2)$ and for any $q>0$ we have
\begin{equation}
\label{def:TA}
\frac{n}{h^2}\,  T_{n,qh}^{\beta_{n,h}} \stackrel{(\dd)}\longrightarrow \cT_{\nu,q} :=\sup_{s\in \sM_q}\big\{\nu \pi(s)-\ent(s) \big\} \quad \text{as } n\to\infty,
\end{equation}
with $\sM_q:=\{s\in \sD , \ent(s)<\infty, \max_{t\in [0,1]} |s(t)|\le q\}.$
We also have
\begin{equation}
\label{conv:largeweights}
\frac{n}{h^2}\,  T_{n,qh}^{\beta_{n,h}, (\ell)} \stackrel{(\dd)}\longrightarrow \cT_{\nu,q}^{(\ell)} :=\sup_{s\in \sM_q}\big\{\nu \pi^{(\ell)}(s)-\ent(s) \big\} \quad \text{as } n\to\infty,
\end{equation}
where $\pi^{(\ell)}:= \sum_{r=1}^{\ell} M_r \ind_{\{Y_r\in s\}}$ with $\{(M_r,Y_r)\}_{r\ge 1}$ the ordered statistics of $\cP$ restricted to $[0,1]\times [-q,q]$, see \cite[Section~5.1]{cf:BT_ELPP} for  details.

Finally, we have $\cT_{\nu,q}^{(\ell)} \to \cT_{\nu,q}$ as $\ell\to\infty$, and $\cT_{\nu,q} \to \cT_{\nu}$ as $q\to\infty$, a.s.
\end{proposition}

\section{Transversal fluctuations: proof of Theorem~\ref{thm:fluctu}}
\label{sec:fluctu}

In this section, we have $\ga\in (1/2,2)$.

First, we partition the interval $[A_n h_n,n]$ into blocks  
\begin{equation}
B_{k,n}:=[2^{k-1}h_n,2^{k} h_n),\quad k=\log_2 A_n+1,\dots, \log_2 (n/h_n)+1.
\end{equation}
In such a way,
\begin{equation}\label{alph12EQ1}
\bP^\omega_{n,\beta_n}\big( \max\limits_{i\leq n}  | S_i | \ge A_n h_n\big)=
\sum_{k=\log_2 A_n+1}^{\log_2 (n/h_n)}\bP^\omega_{n,\beta_n}\big( \max\limits_{i\leq n} | S_i |\in B_{k,n}\big).
\end{equation}

We first deal with the case where $n/m(nh_n) \stackrel{n\to\infty}{\to} 0$ for the sake of clarity of the exposition: in that case, $\log \bZ_{n,\gb_n}^{\go}$ does not need to be recentered. We treat the remaining case (in particular we have $\ga\ge 3/2$) in a second step.

\subsection{Case $n/m(nh_n) \stackrel{n\to\infty}{\to} 0$}
\label{sec:nocentering}
We observe that the assumption  $\omega\geq 0$ implies that the partition function $\bZ^\omega_{n,\beta_n}$ is larger than one.
Therefore, 
\[\bP^\omega_{n,\beta_n}\big( \max\limits_{i\leq n}  | S_i |\in B_{k,n}\big)\leq 
\bZ^\omega_{n,\beta_n}\big( \max\limits_{i\leq n} | S_i |\in B_{k,n}\big).\]
By using Cauchy-Schwarz inequality, we get that 
\begin{equation}
\label{eq:CauchySchwarz}
\bZ^\omega_{n,\beta_n}\big( \max\limits_{i\leq n}  | S_i |\in B_{k,n}\big)^2 \le \bP \big( \max\limits_{i\leq n} \big| S_i\big| \ge 2^{k-1} h_n \big) \times 
\bZ^\omega_{n,2\beta_n}\big( \max\limits_{i\leq n} | S_i | \le 2^k h_n \big) \, .
\end{equation}

The first probability is bounded by $ 2\bP( |S_n|\ge h_n )\le 4 \exp(-  2^{2k} h_n^2/2n)$ (by Levy's inequality and a standard Chernov's bound).
We are going to show the following lemma, which is the central estimate of the proof.
\begin{lemma}
\label{lem:Zmax}
There exist some constant $q_0>0$ and some $\nu>0$, such that for all $q \ge q_0$ we have
\begin{equation}
\bbP\Big( \bZ^\omega_{n,2\beta_n}\big( \max\limits_{i\leq n} | S_i | \le qh_n \big) \ge  e^{ \frac{1}{4} q^2  \frac{h_n^2 }{n} } \Big) \le   q^{-\nu}\Big( 1 +    1\wedge \frac{n}{m(nh_n)} \Big) \, .
\end{equation}
\end{lemma}
Therefore, if $n/m(nh_n) \stackrel{n\to\infty}{\to} 0$,
this lemma gives that for $c_0 =1/8$ and for $k$ large enough (\textit{i.e.}\ $A_n$ large enough), 
using \eqref{eq:CauchySchwarz}, 
\[\bbP\Big(  \bZ^\omega_{n,\beta_n}\big( \max\limits_{i\leq n}  | S_i |\in B_{k,n}\big)  \ge 4e^{- c_0  2^{2k} h_n^2/n } \Big)  \le   (2^{k})^{-\nu}\, .\]
Then,
using that $\sum_{k> \log_2 A_n} 4e^{- c_0 2^{2k} h_n^2 /n } \le e^{- c_1 A_n^2 h_n^2/n}$, we get that  by a union bound
\begin{align}
\bbP\Big( \bP^\omega_{n,\beta}\big( & \max\limits_{i\leq n} \big| S_i\big| \ge A_n h_n \big)  \ge e^{- c_1 A_n^2 h_n^2/n} \Big) \notag \\
&\le \sum_{k=\log_2 A_n +1}^{\log_2(n/h_n)} \bbP\Big(  \bZ^\omega_{n,\beta}\big( \max\limits_{i\leq n} \big| S_i\big|\in B_{k,n}\big) \ge 4e^{- c_0 2^{2k} h_n^2/n}  \Big) \notag \\
&\le \sum_{k >  \log_2 A_n }  2^{-\nu k} \le c A_n^{-\nu} \, .
\label{unionboundBkn}
\end{align}
We stress that in the case when $n/m(nh_n) \stackrel{n\to\infty}{\to} 0$, we  do not need the additional $n$ in front of $e^{- c_1 A_n^2 h_n^2/n}$ in \eqref{eq:hscaling}.

\begin{proof}[Proof of Lemma \ref{lem:Zmax}]
For simplicity, we assume in the following that $qh_n$ is an integer. We fix $\gd>0$ such that $(1+\gd)/\ga <2$ and $(1-\gd)/\ga >1/2$, and let
\begin{equation}\label{eq:defT}
\mathtt T=\mathtt T_n(qh_n) = \frac{h_n^2}{n}  q^{1/\ga} ( q^2 h_n^2/n)^{-(1-\gd)^{3/2}/\ga} \vee 1
\end{equation}
be a truncation level. Note that if $\ga \le (1-\gd)^{3/2}$ then we have $\mathtt T =1$.
We decompose the partition function into three parts: thanks to H\"older's inequality, we can write that
\begin{align}
\log \bZ^\omega_{n,2\beta_n}\big( \max\limits_{i\leq n} | S_i | \le qh_n \big)
\le \frac13 \log \bZ_{n,6\gb_n}^{(>\mathtt T)} 
+ \frac13 \log   \bZ_{n,6\gb_n}^{((1,\mathtt T])}
+ \frac13 \log \bZ_{n,6\gb_n}^{(\le1)}\, ,
\label{eq:threeparts}
\end{align}
where the three partition functions correspond to three ranges for the weights $\gb_n \go_{i,S_i}$:
\begin{align}\label{def:ZbigT}
\bZ_{n,6\gb_n}^{(>\mathtt T)} &:=  \bE\Big[ \exp\Big(\sum_{i=1}^n 6 \gb_n \go_{i,S_i} \ind_{\{ \gb_n \go_{i,S_i} >\mathtt  T\}}  \Big)  \ind_{\{ \max\limits_{i\leq n} | S_i | \le qh_n\}}\Big]  \\
\bZ_{n,6\gb_n}^{((1,\mathtt T])} &:=   \bE\Big[ \exp\Big(\sum_{i=1}^n 6 \gb_n \go_{i,S_i} \ind_{\{ \gb_n \go_{i,S_i} \in (1, \mathtt T]\}}  \Big)  \ind_{\{ \max\limits_{i\leq n} | S_i | \le qh_n\}}\Big]  
\label{def:Zin1T}\\
\bZ_{n,6\gb_n}^{(\le1)} & :=   \bE\Big[ \exp\Big(\sum_{i=1}^n 6 \gb_n \go_{i,S_i} \ind_{\{ \gb_n \go_{i,S_i} \le 1\}}  \Big) 
\ind_{\{ \max\limits_{i\leq n} | S_i | \le qh_n\}}
\Big] \, .
\label{def:Zsmall1}
\end{align}
We now show that with high probability, these three partition functions cannot be large. Note that when $\mathtt T=1$, the second term is equal to $1$ and we do not have to deal with it.


\medskip
{\bf Term 1.}
For \eqref{def:ZbigT}, we prove that for any $\nu<2\ga-1$, for $q$ sufficiently large, for all $n$ large enough we have
\begin{equation}
\label{aim:part1}
\bbP\Big( \log \bZ_{n,6\gb_n}^{(>\mathtt T)} \ge  c_0 q^2 \frac{h_n^2}{n} \Big) \le q^{-\nu} \, .
\end{equation}
We compare this truncated partition function with the partition function where we keep the first $\ell$ weights in the ordered statistics $(M_i^{(n,qh_n)})_{1\le i\le nqh_n}$.
Define
\begin{equation}
\label{def:ell0}
\ell = \ell_n(qh_n):= (q^2 h_n^2/n)^{1-\gd} \, , \quad \text{so  }\ \mathtt T = \frac{h_n^2}{n}  q^{1/\ga} \times \ell^{-(1-\gd)^{1/2}/\ga} \, ,
\end{equation}
and set
\begin{equation}
\label{Zell}
\bZ_{n,6\gb_n}^{(\ell)} := \bE\Big[ \exp\Big(\sum_{i=1}^{\ell}   6\gb_n M_i^{(n,qh_n)} \ind_{\{Y_i^{(n,q h_n)} \in S\}}\Big) \Big] \, .
\end{equation}
Remark that, with the definition of $\mathtt T$ and thanks to the relation \eqref{def:hn} verified by $\gb_n$, we have  that for $n$ large enough
\begin{align*}
\bbP\Big( \gb_n M_{\ell}^{(n,qh_n)} >\mathtt  T  \Big) &\le  \bbP\Big( M_{\ell}^{(n,qh_n)}  \ge \frac12  q^{1/\ga} \ell^{-(1-\gd)^{1/2}/\ga}  m(nh_n) \Big)
\end{align*}
Then, since we have $q/\ell\le 1$ (see \eqref{def:ell0}), we can use Potter's bound to get that for $n$ sufficiently large 
\[m\big( n q h_n /\ell \big) \le (q/\ell)^{(1-\gd^2)/\ga} m(nh_n) \, ,\]
and we obtain that provided that $\gd$ is small enough
\begin{align*}
\bbP\Big( \gb_n M_{\ell}^{(n,qh_n)} >\mathtt  T  \Big) &\le  \bbP\Big( M_{\ell}^{(n,qh_n)}  \ge  c_0  q^{\gd^2/\ga}\ell^{\gd^2/\ga}  m\big( n q h_n /\ell \big) \Big) \le (c q \ell)^{- \gd^2 \ell/2} \, ,
\end{align*}
where we used \cite[Lemma 5.1]{cf:BT_ELPP} for the last inequality.  We therefore get that, with probability larger than $1- (c \ell )^{- \gd \ell/2}$ (note that $\ell^{- \gd \ell/2} \le q^{-\gd \ell/2} \le q^{-4}$ for $n$ large enough),
we have that 
\begin{equation}
\label{eq:inclusion}
\Big\{ (i,x) \in \llbracket 1,n \rrbracket \times \llbracket -qh_n ,qh_n \rrbracket ; \gb_n \go_{i,x} > \mathtt T \Big\} \subset \gU_\ell := \big\{ Y_1^{(n,qh_n)}, \ldots, Y_{\ell}^{(n,qh_n)}\big\} \, ,
\end{equation}
and hence $\bZ_{n,6\gb_n}^{(>T)} \le  \bZ_{n,6\gb_n}^{(\ell)}$.

We are therefore left to focus on the term $\bZ_{n,6\gb_n}^{(\ell)}$: recalling the definitions \eqref{def:Omega} and \eqref{def:discrELPPell}, we get that
\begin{equation}\label{ZTleZell}
\begin{split}
\bZ_{n,6\gb_n}^{(\ell)} &=  \sum_{\gD \subset \gU_\ell} e^{ 6 \gb_n \Omega_{n,qh_n}^{(\ell)}(\Delta) } \bP\big( S \cap \gU_m =\Delta \big) \\
& \le \sum_{\gD \subset \gU_\ell}  \exp\big( 6\gb_n \Omega_{n,qh_n}(\Delta) - \ent(\Delta) \big) \le 2^{\ell} \exp \Big(  T_{n,qh_n}^{6\gb_n,(\ell)}\Big) \, ,
\end{split}
\end{equation}
where we used that $\bP(\Delta \subset S)\le \exp(-\ent(\Delta))$ as noted below \eqref{def:entdelta}.

Note that we have $\ell \le \frac12 c_0 q^2 h_n^2/n$ for $n$ large enough (and $q\ge 1$), so we get that
\[\bbP\Big( \log \bZ_{n,6\gb_n}^{(\ell)}  \ge c_0 q^2 \frac{h_n^2}{n}  \Big) \le  \bbP\Big(  T_{n,qh_n}^{6\gb_n, (\ell)} \ge \frac12 c_0 q^2 \frac{h_n^2}{n}\Big) \, .\]
Then, by the definition \eqref{def:hn} and thanks to Potter's bound, for any $\eta>0$ there exists a constant $c_{\eta}$ such that for any $q\ge 1$
\[ \frac{\big(6\gb_n m(nqh_n) \big)^{4/3}}{ ( q^2 h_n^2/n)^{1/3}} \le  c_{\eta} q^{(1+\eta) \frac{4}{3\ga} -\frac{2}{3} }\,  \frac{h_n^2}{n}  = c_{\eta}  (q^{4/3})^{(1+\eta)/\ga -2}    \times q^2 \frac{h_n^2}{n} \, ,\]
where we used that for any $\eta>0$, $m(nqh_n) \le c'_{\eta} q^{(1+\eta)/\ga} m(nh_n)$ provided that $n$ is large enough (Potter's bound).
Therefore, provided that $\eta$ is small enough so that $(1+\eta)/\ga < 2$, an application of \cite[Prop.~2.6]{cf:BT_ELPP} gives that for $q$ large enough (so that $b_q:= \tfrac{c_0}{2 c_{\eta}} (q^{4/3})^{2- (1+\eta)/\ga}$ is large), 
\begin{equation}
\bbP\Big(  T_{n,qh_n}^{6\gb_n, (\ell)} \ge \frac12 c_0 q^2 \frac{h_n^2}{n}\Big) \le \bbP\Big(  T_{n,qh_n}^{6\gb_n, (\ell)} \ge b_q  \times\frac{\big(6\gb_n m(nqh_n) \big)^{4/3}}{ ( q^2 h_n^2/n)^{1/3}}   \Big)  \le c q^{- \nu} \, ,
\end{equation}
with $\nu= 2\ga-1-2\eta$.
This gives \eqref{aim:part1}, since $\eta$ is arbitrary.


\medskip
{\bf Term 2.} We now turn to \eqref{def:Zin1T}
We consider only the case $\mathtt{T} >1$ (and in particular we have $\ga >(1-\gd)^{3/2}$). We show that for any $\eta>0$, there is a constant $c_{\eta}>0$ such that for $q$ large enough and $n$ large enough,
\begin{equation}
\label{aim:part2}
\bbP\Big( \log  \bZ_{n,6\gb_n}^{((1,\mathtt T])} \ge c_0 \big( q^2 h_n^2 /n  \big)^{1-\eta}  \Big) \le  \exp\big( - c_{\eta}  (q^2 h_n^2/n)^{1/3}  \big)\, .
\end{equation}
Again, we need to decompose $\bZ_{n,6\gb_n}^{((1,\mathtt T])}$ according to the values of the weights.
We set $\theta := (1-\gd)2/\ga >1$, and let
\begin{align}
\ell_j &:= ( q^2 h_n^2/n)^{ \theta^j (1-\gd)} = (\ell_0)^{\theta^j}\, , \quad \text{ with } \ell_0=\ell =(q^2 h_n^2/n)^{1-\gd} \text{ as in \eqref{def:ell0} } \label{def:ellj}\\
\mathtt{T}^{(j)} &:= \frac{h_n^2}{n}  q^{1/\ga} \times (q^2 h_n^2/n)^{- \theta^j (1-\gd)^{3/2}/\ga} = \frac{h_n^2}{n}  q^{1/\ga} \big( \ell_j \big)^{-(1-\gd)^{1/2} /\ga}
\label{def:Tj}
\end{align}
for $j \in \{ 0,\ldots, \kappa \}$ with $\kappa$ the first integer such that $ \theta^{\kappa}  >\ga /(1-\gd)^{3/2}$. We get that $\mathtt{T}^{(0)}=\mathtt T$, and $\mathtt T^{(\kappa)} <1$.
Then,  thanks to H\"older inequality we may write
\begin{align*}
\log \bZ_{n,6\gb_n}^{((1,T])} &\le \frac{1}{\kappa} \sum_{j=1}^{\kappa}  \log  \bZ_{n, 6 \kappa \gb_n}^{((\mathtt{T}^{(j)}, \mathtt{T}^{(j-1)}])} \, ,  \quad \text{with}\\
\bZ_{n, 6 \kappa \gb_n}^{((\mathtt{T}^{(j)}, \mathtt{T}^{(j-1)}])} &:= \bE\Big[ \exp\Big(\sum_{i=1}^n 6 \kappa \gb_n  \go_{i,S_i} \ind_{\{ \gb_n \go_{i,S_i} \in (\mathtt{T}^{(j)}, \mathtt{T}^{(j-1)}]\}}  \Big)  \ind_{\{ \max\limits_{i\leq n} | S_i | \le qh_n\}}\Big]\, .
\end{align*}
To prove \eqref{aim:part2}, it is therefore enough to prove that for any $1\le j\le \kappa$, since $\ell_j \ge (q^2 h_n^2/n)^{1-\gd}$,
\begin{equation}
\label{aim:part2-intermed}
\bbP\Big( \log \bZ_{n, 6 \kappa \gb_n}^{((\mathtt{T}^{(j)}, \mathtt{T}^{(j-1)}])}  \ge 8 \kappa  \big( q^2 h_n^2/n \big) \ell_j^{-\gd/10} \Big) \le  \exp\big( - c (  q^2 h_n^2/n)^{1/3}  \big) \, .
\end{equation}
First of all, we notice that in view of \eqref{def:ellj}-\eqref{def:Tj}, with the same computation leading to \eqref{eq:inclusion},   we have that with probability larger than $1- (c\ell_j )^{-\gd \ell_j/4}$
\begin{align}
\label{eq:inclusionj}
\Big\{ (i,x) \in \llbracket 1,n \rrbracket \times \llbracket -qh_n ,qh_n \rrbracket ; &\gb_n \go_{i,x} > \mathtt{T}^{(j-1)} \Big\} \notag\\
& \subset \gU_{\ell_j} := \big\{ Y_1^{(n,qh_n)}, \ldots, Y_{\ell_j}^{(n,qh_n)}\big\} \, .
\end{align}
On this event, and using that $\ell_j = (\ell_{j-1})^{(1-\gd) 2/\ga}$ and 
\[\mathtt{T}^{(j-1)} = \frac{h_n^2}{n} q^{1/\ga} \ell_j^{-  (1-\gd)^{-1/2}/2} \le\frac{h_n^2}{n} q^{1/\ga} \ell_j^{-  1/2 -\gd/5} \] (if $\gd$ is small), we have
\begin{align}\label{eq:estimatebZ6kappa}
\bZ_{n, 6 \kappa \gb_n}^{((\mathtt{T}^{(j)}, \mathtt{T}^{(j-1)}])}  &\le \bE\Big[   \exp\Big(  6\kappa \mathtt{T}^{(j-1)} \sum_{i=1}^{\ell_j} \ind_{\{ Y_i^{(n,qh_n)} \in S\} }    \Big)   \Big] \\
\nonumber
& \leq  e^{ 6 \kappa q^2 \frac{h_n^2}{n} \ell_j^{-\gd/10}   } + \mathcal{H}_j
\end{align}
with
\begin{align*}
\mathcal{H}_j& := \sum_{k= q^{2-\frac{1}{\ga} } \ell_{j}^{1/2 + \gd/10} }^{\ell_j} \sum_{\Delta \subset \gU_{\ell_j} ; |\Delta|=k} e^{ 6\kappa \frac{h_n^2}{n} q^{1/\ga} \ell_j^{-1/2- \gd/5} k  } \ \bP\big( S\cap \gU_{\ell_j} =\Delta \big)\\
& \le  \sum_{k= q^{2-\frac{1}{\ga} }\ell_{j}^{1/2 + \gd/10} }^{\ell_j} \binom{\ell_j}{k} \exp\Big( 6\kappa \frac{h_n^2}{n} q^{1/\ga } \ell_j^{-1/2 - \gd/5 } k - \inf_{\Delta\subset \gU_{\ell_j} , |\Delta| =k} \ent(\Delta) \Big) \, .
\nonumber
\end{align*}
Then, we may bound $\binom{\ell_j}{k}\le e^{ k \log \ell_j }$. We notice from the definition of $\kappa$ (and since $\theta \in (1,2)$) that there exists some $\eta>0$ such that $\ell_{j} \le \ell_{\kappa} \le (q^2h_n^2/n)^{2-\eta}$ for any $1\le j\le \kappa$: it shows in particular that $\log \ell_j \le \ell_j^{\gd^2} \le q^2 \frac{h_n^{2}}{n}  \ell_j^{-1/2- \gd/5}$, provided that $n$ is sufficiently large and $\gd$ has been fixed sufficiently small. We end up with the following bound
\begin{equation*}\label{eq:estimatebZ6kappa2}
\mathcal{H}_j \le  \sum_{k=  q^{2-\frac{1}{\ga} } \ell_{j}^{1/2+\gd/10} }^{\ell_j} \exp\Big( c q^{2} \frac{h_n^2}{n} \ell_j^{-1/2- \gd/5 } k - \inf_{\Delta\subset \gU_{\ell_j} , |\Delta| =k} \ent(\Delta) \Big)  .
\end{equation*}
Then, we may use relation (2.5) of \cite{cf:BT_ELPP} (with $m=\ell_j$, $h=qh_n$) to get that, for any $ k\ge  q^{2-\frac{1}{\ga} } \ell_{j}^{1/2+\gd/10} $
\begin{align}\label{eq:entDelta>q}
\bbP\Big( \inf_{\Delta\subset \gU_{\ell_j} , |\Delta| =k} \ent(\Delta) \le  2 c q^{2} \frac{h_n^2}{n}  \ell_j^{-1/2- \gd/5 } k \Big) & \le  \bigg( \frac{C_0 (2c\ell_j^{-1/2- \gd/5} k)^{1/2} \ell_j}{k^2} \bigg)^k  \nonumber\\
& \le \big( c q^{\frac{3}{2\ga} - 3} \ell_j^{ -\gd/4} \big)^k \le \big( c \ell_j \big)^{-  \gd k /4}\, .
\end{align}
For the last inequality, we used that  $q^{\frac{3}{2\ga} - 3}\le 1$, since $\ga>1/2$ and $q\ge 1$.
Since we have that $q^2 \frac{h_n^2}{n}  \ell_j^{-1/2 - \gd/5} \ge 1$, we get that there is a constant $c'>0$ such that
\[\sum_{k\ge  q^{2-\frac1\ga} \ell_j^{1/2 +\gd/10} } e^{- c q^2 \frac{h_n^2}{n}  \ell_j^{-1/2 - \gd/5} k} \le c' e^{ - c q^2 \frac{h_n^2}{n}  \ell_j^{-\gd/10}} \le c'. \]
Using \eqref{eq:entDelta>q}, we therefore obtain, via a union bound (also recalling \eqref{eq:inclusionj}), that provided that $n$ is large enough
\begin{align*}
\bbP\Big(  \bZ_{n, 6 \kappa \gb_n}^{((\mathtt{T}^{(j)}, \mathtt{T}^{(j-1)}])} \ge e^{8 \kappa q^2 \frac{h_n^2}{n}  \ell_j^{-\gd/10}} \Big)& \le (c \ell_j)^{-\gd \ell_j/4} + \sum_{k\ge  q^{2-\frac1\ga} \ell_j^{1/2 +\gd/10} } \big( c \ell_j \big)^{-  \gd k /4}\\
&  \le  \big( c \ell_j \big)^{- c_{\gd} \ell_j^{1/2}} \, .
\end{align*}
This proves \eqref{aim:part2-intermed} since $\ell_j \ge \ell_0  = (q^2 h_n^2/n)^{1-\gd}$.


\medskip

{\bf Term 3.}
For the last part \eqref{def:Zsmall1}, we prove that for arbitrary $\eta>0$,
\begin{equation}
\label{aim:part3}
\bbP\Big(  \log  \bZ_{n,6\gb_n}^{(\le1)} \ge c_0 q^2 \frac{h_n^2}{n}  \Big) \le c q^{-2} \times
\begin{cases}
\frac{n}{m(nh_n)} & \quad \text{ if } \ga>1 \, ,\\
\frac{n}{ m(nh_n)^{(1-\eta)\ga}}  & \quad \text{ if } \ga\le 1 \, .
\end{cases}
\end{equation}
Let us stress that in the case $\ga\le 1$ we get that for $n$ large $m(nh_n)^{(1-\eta)\ga} \ge (nh_n)^{1-2\eta}$, therefore $n/(nh_n)^{(1-\eta)\alpha}$ goes to $0$ provided that $\eta$ is small enough, since we are considering the case when  $h_n\ge \sqrt{n}$. Hence, we can replace the upper bound in \eqref{aim:part3} by $1\wedge (n/m(nh_n))$.

To prove \eqref{aim:part3}, we use that $e^{6 x \ind_{\{x\le 1\}}} \le 1+ e^6 x \ind_{\{x\le 1\}}$ for any $x$, and we get that
\begin{align}
\label{EZnsmall}
\bZ_{n,6\gb_n}^{(\le1)} & \le \bE\Big[ \prod_{i=1}^{n} \big (1+ 6e^6 \gb_n  \go_{i,s_i} \ind_{\{ \gb_n \go_{i,s_i} \le 1\}} \big)  \Big]\, ,  \\
\text{and } \ \bbE\bZ_{n,6\gb_n}^{(\le1)} & \le \bE\Big[ \prod_{i=1}^{n} \big( 1+ 6e^6 \gb_n \bbE \big[ \go  \ind_{\{  \go  \le 1/\gb_n\}} \big] \big)  \Big] \le e^{ 6e^6 n \gb_n \bbE [ \go  \ind_{\{  \go  \le 1/\gb_n\}}]  } \, .\notag
\end{align}
Therefore, by Markov inequality and Jensen inequality,
\begin{align}
\bbP\Big(  \log  \bZ_{n,6\gb_n}^{(\le1)} \ge c_0 q^2 \frac{h_n^2}{n}  \Big)& \le \frac{1}{c_0 q^2} \frac{n}{h_n^2} \log \bbE \bZ_{n,6\gb_n}^{(\le1)}  \le C q^{-2} \frac{n^2 \gb_n }{h_n^2} \bbE \big[ \go  \ind_{\{  \go  \le 1/\gb_n\}} \big] \, .
\label{MarkovJensen}
\end{align}
It remains to estimate $\bbE \big[ \go  \ind_{\{  \go  \le 1/\gb_n\}} \big]$.
If $\ga>1$ then it is bounded by  $\bbE[\go]<+\infty$: this gives the first part of \eqref{aim:part3}, using also \eqref{def:hn}. If $\ga\le1$ then for any $\gd>0$, for $n$ large enough we have   $\gb_n \bbE \big[ \go  \ind_{\{  \go  \le 1/\gb_n\}} \big]\le  \gb_n^{(1-\eta)\ga} $ for $n$ large: by using \eqref{def:hn} together with $h_n^2/n \geq 1$,  this gives the second part of \eqref{aim:part3}.

\smallskip
The conclusion of Lemma~\ref{lem:Zmax} follows by collecting the estimates \eqref{aim:part1}-\eqref{aim:part2}-\eqref{aim:part3} of the three terms in \eqref{eq:threeparts}.
\end{proof}

\subsection{Remaining case ($\ga\ge 3/2$)}
\label{sec:centeringneeded}
We now consider the remaining case, \textit{i.e.}\ when we do not have that $n/m(nh_n) \stackrel{n\to\infty}{\to} 0$. In particular, we need to have that $\ga\ge 3/2$, and hence $\bbE[\go]=:\mu <+\infty$.
Then, we do not simply use that $\bZ_{n,\gb_n}^{\go} \ge 1$ to bound $\bP_{n,\gb_n}^{\go} \big( \max_{i\le n} |S_i| \in B_{k,n} \big)$, but instead we use a re-centered partition function $\bar \bZ_{n,\gb_n}^{\go} = e^{- n\gb_n \mu}\bZ_{n,\gb_n}^{\go}$, so that we can write
\begin{align}
\bP_{n,\gb_n}^{\go} \big( \max_{i\le n} |S_i| \in B_{k,n} \big)& = \frac{1}{ \bar \bZ_{n,\gb_n}^{\go}} \bE\Big[   \exp\Big(  \sum_{i=1}^n \gb_n (\go_{i,s_i} -\mu)  \Big) \ind_{\{ \max_{i\le n} |S_i| \in B_{k,n}\}}\Big] \nonumber \\
& =: \frac{1}{ \bar \bZ_{n,\gb_n}^{\go}}\,  \bar \bZ_{n,\gb_n}^{\go} \big( \max_{i\le n} |S_i| \in B_{k,n} \big)\, .
\label{P=Zbar}
\end{align}
First, we need to get a lower bound on $\bar \bZ_{n,\gb_n}^{\go}$.
\begin{lemma}
\label{lem:barZlowerbound}
For any $\gd>0$, there is a constant $c>0$ such that for any positive sequence $\gep_n \le 1$ with  $\gep_n \ge  n^{-1/2} (h_n^2/n)^{\ga-3/2+\gd}$ (this goes to $0$ for $\gd$ small enough), and any $n\ge 1$ 
\begin{equation}
\bbP\Big( \bar \bZ_{n,\gb_n}^{\go} \ge  n^{-1} \,  e^{ \gep_n  \frac{ h_n^2}{n} } \Big) \ge 1- e^{- c /\gep_n^{\ga-1/2-\gd} } - e^{-  c \gep_n h_n^2/n}. 
\end{equation}
\end{lemma}
We postpone the proof of this lemma to the end of this subsection, and we now complete the proof of Theorem \ref{thm:fluctu}-\eqref{eq:hscaling}.
Lemma~\ref{lem:barZlowerbound} gives that $\bar \bZ_{n,\gb_n}^{\go} \ge n^{-1}$ with overwhelming probability: using \eqref{alph12EQ1} combined with \eqref{P=Zbar}, we get, analogously to \eqref{unionboundBkn},
\begin{align}
\label{unionboundBkn-bis}
\bbP&\Big( \bP^\omega_{n,\beta}\big( \max\limits_{i\leq n} \big| S_i\big| \ge A_n h_n \big)  \ge n e^{- c_1 A_n^2 h_n^2/n} \Big) \\
&\le \bbP\big( \bar \bZ_{n,\gb_n}^{\go} \le  n^{-1}  \big)  + \sum_{k=\log_2 A_n +1}^{\log_2(n/h_n)+1} \bbP\Big(  \bar \bZ^\omega_{n,\beta}\big( \max\limits_{i\leq n} \big| S_i\big|\in B_{k,n}\big) \ge 4e^{- c_0 2^{2k} h_n^2/n}  \Big)\, .\notag
\end{align}
Then, we have a lemma which is the analogous of Lemma~\ref{lem:Zmax} for $\bar \bZ_{n,\gb_n}^{\go}$.
\begin{lemma}
\label{lem:barZmax}
There exist some constant $q_0>0$ and some $\nu>0$, such that for all $q \ge q_0$ we have
\begin{equation}
\bbP\Big( \bar \bZ^\omega_{n,2\beta_n}\big( \max\limits_{i\leq n} | S_i | \le qh_n \big) \ge  e^{ \frac{1}{4} q^2  \frac{h_n^2 }{n} } \Big) \le   q^{-\nu} \, .
\end{equation}
\end{lemma}
\begin{proof}
The proof follows the same lines as for Lemma \ref{lem:Zmax}: \eqref{eq:threeparts} still holds, with $\gb_n \go_{i,S_i}$ replaced by $\gb_n (\go_{i,S_i}-\mu)$ (outside of the indicator function).
The bounds \eqref{aim:part1}-\eqref{aim:part2} for terms 1 and 2 still hold, since one fall back to the same estimates by using that $(\go_{i,S_i}-\mu)\le \go_{i,S_i}$.
It remains only to control only the third term: we prove that when $\mu:=\bbE[\go]<\infty$, then for any $\gd>0$, provided that $n$ is large enough,
\begin{equation}
\label{aim:part3-bar}
\bbP\Big( \log \bar \bZ_{n,6\gb_n}^{(\le 1)}  \ge c_0 q^2 \frac{h_n^2}{n} \Big) \le c q^{-2} \times n^{-1/2} \Big( \frac{h_n^2}{n}\Big)^{\ga-\frac32 +\gd} , 
\end{equation}
where we set analogously to \eqref{eq:threeparts}
\begin{equation}\label{def:barZsmall1}
\bar \bZ_{n,6\gb_n}^{(\le 1)} :=   \bE\Big[ \exp\Big(\sum_{i=1}^n 6 \gb_n (\go_{i,S_i}-\mu) \ind_{\{ \gb_n \go_{i,S_i} \le 1\}}  \Big) 
\Big] \, . 
\end{equation}
Then, using $h_n^2/n\le n$ (if $\ga\ge 3/2$, the upper bound in \eqref{aim:part3-bar} is bounded by $c q^{-2} n^{\ga-2 +\gd}$ which is smaller than $q^{-2}$ provided that $\gd$ had been fixed small enough.

To prove \eqref{aim:part3-bar}, we use that there is a constant $c$ such that $e^{x} \le 1+x+c x^2$ as soon as $|x|\le 6$, so that we get similarly to \eqref{EZnsmall} that
\begin{align}
\bbE\bZ_{n,6\gb_n}^{(\le1)} & \le  \Big( 1+ \gb_n \bbE \big[ (\go-\mu)  \ind_{\{  \go  \le 1/\gb_n\}} \big] + c\gb_n^2 \bbE \big[ (\go-\mu)^2  \ind_{\{  \go  \le 1/\gb_n\}} \big] \Big)^n \notag \\
&\le \exp\Big( c n L(1/\gb_n) \gb_n^{\ga} \Big) \le \exp\Big(   \frac{c}{h_n}  (h_n^2/n)^{\ga+\gd} \Big) \, .
\label{EZnsmall-bis}
\end{align}
For the second inequality, we used that $\bbE \big[ (\go-\mu)  \ind_{\{  \go  \le 1/\gb_n\}} \big] \le 0$ (as soon a $1/\gb_n \ge \mu$), and also that
$\bbE \big[ (\go-\mu)^2  \ind_{\{  \go  \le 1/\gb_n\}} \big] \le c L(1/\gb_n) \gb_n^{\ga-2}$, thanks  to \eqref{eq:DisTail}. The last inequality holds for any fixed $\gd$, provided that $n$ is large enough, and comes from using Potter's bound and the relation \eqref{def:hn} to get that $L(1/\gb_n) \gb_n^{\ga}\le c' \bbP(\go>1/\gb_n) \le (nh_n)^{-1} (h_n^2/n)^{\ga+\gd}$.
Then, applying Markov and Jensen inequalities as in \eqref{MarkovJensen}, we get that 
\begin{align*}
\bbP\Big( \log \bar \bZ_{n,6\gb_n}^{(\le 1)}  \ge c_0 q^2 \frac{h_n^2}{n} \Big) & \le  c q^{-2} \frac{n}{h_n^3} \Big( \frac{h_n^2}{n}\Big)^{\ga+\gd}  \, ,
\end{align*}
which proves \eqref{aim:part3-bar}.
\end{proof}

With Lemma~\ref{lem:barZmax} in hand, and using Cauchy-Schwarz inequality as in \eqref{eq:CauchySchwarz}, we get that 
\[ \bbP\Big(  \bar \bZ^\omega_{n,\beta}\big( \max\limits_{i\leq n} \big| S_i\big|\in B_{k,n}\big) \ge 2e^{- c_0 2^{2k} h_n^2/n}  \Big) \le (2^{k})^{-\nu} \, .\]
Plugged into \eqref{unionboundBkn-bis}, this concludes the proof of Theorem~\ref{thm:fluctu}-\eqref{eq:hscaling}. It therefore only remains to prove Lemma \ref{lem:barZlowerbound}.

\begin{proof}[Proof of Lemma \ref{lem:barZlowerbound}]
We need to obtain a lower bound on $\bar \bZ_{n,\gb_n}$, so we use Cauchy-Schwarz inequality \emph{backwards}: we apply Cauchy Schwarz inequality to 
\begin{align*}
\bar \bZ_{n,\gb_n/2}^{(>1)}&:=\bE\Big[ \exp\Big( \sum_{i=1}^n \frac{\gb_n}{2} (\go_{i,s_i} -\mu) \ind_{\{ \gb_n \go_{i,s_i} >1 \}} \Big) \Big] \\
&\le (\bar \bZ_{n,\gb_n} \big)^{1/2} \bE\Big[ \exp\Big( \sum_{i=1}^n -\gb_n (\go_{i,s_i} -\mu) \ind_{\{ \gb_n \go_{i,s_i} >1 \}} \Big)\Big]^{1/2}\\
&\hspace{6cm} =:
(\bar \bZ_{n,\gb_n} \big)^{1/2} (\bar \bZ_{n,-\gb_n}^{(\le 1)} \big)^{1/2} ,
\end{align*}
so that
\begin{equation}
\bar \bZ_{n,\gb_n}  \ge  \big( \bar \bZ_{n,\gb_n/2}^{(>1)}\big)^2 \Big/ \bar \bZ_{n,-\gb_n}^{(\le 1)}\, .
\end{equation}
Hence, we get that
\begin{equation}
\label{twotermsbarZmax}
\bbP\Big( \bar \bZ_{n,\gb_n}^{\go} \le n^{-1} \,  e^{ \gep_n \frac{ h_n^2}{n} } \Big) \le  \bbP\Big( \bar \bZ_{n,-\gb_n}^{(\le 1)} \ge e^{ \gep_n \frac{ h_n^2}{2n} } \Big) + \bbP\Big( \bar \bZ_{n,\gb_n/2}^{(>1)} \le n^{-1/2} \,  e^{ \gep_n \frac{ h_n^2}{4n} } \Big)\, ,
\end{equation}
and we deal with both terms separately.

For the first term, we use that analogously to \eqref{EZnsmall-bis} we have
\begin{equation}\label{EZnegativebeta}
\begin{split}
\bbE \bar \bZ_{n,-\gb_n}^{(\le 1)} &\le \Big( 1- \gb_n \bbE \big[ (\go-\mu)  \ind_{\{  \go  \le 1/\gb_n\}} \big] + c\gb_n^2 \bbE \big[ (\go-\mu)^2  \ind_{\{  \go  \le 1/\gb_n\}} \big] \Big)^n \\
&\le \Big( 1+ c L(1/\gb_n) \gb_n^{\ga}\Big)^n \le \exp\Big( \frac{c}{h_n} \big( h_n^2/n \big)^{\ga+\gd/2} \Big) \, ,
\end{split}
\end{equation}
Here, the difference with \eqref{EZnsmall-bis} is that we use for the second inequality that $- \bbE \big[ (\go-\mu)  \ind_{\{  \go  \le 1/\gb_n\}} \big] = \bbE\big[ (\go-\mu) \ind_{\{  \go  > 1/\gb_n\}}\big] \le c L(1/\gb_n) \gb_n^{\ga-1}$, thanks to \eqref{eq:DisTail}. Again, the second inequality holds for any fixed $\gd$, provided that $n$ is large enough.
Using Markov's inequality, one therefore obtains that the first term in \eqref{twotermsbarZmax} is bounded by
\begin{equation}\label{EZnegativebeta2} \bbP\Big( \bar \bZ_{n,-\gb_n}^{(\le 1)} \ge e^{ \gep_n \frac{\gep_n h_n^2}{2n} } \Big)  \le  \exp\Big(  \frac{c}{h_n} \big( h_n^2/n \big)^{\ga+\gd}   - \gep_n \frac{ h_n^2}{2n}  \Big) \le \exp\Big(    - \gep_n \frac{ h_n^2}{4n}  \Big) \, ,
\end{equation}
the second inequality holding provided that $\gep_n$ is larger than $n^{-1/2} \Big( \frac{h_n^2}{n}\Big)^{\ga-\frac32 +\gd}$.

As far as the second term in \eqref{twotermsbarZmax} is concerned, we find a lower bound on $\bZ_{n,\gb_n}^{(\ge 1)}$ by restricting to a particular set of trajectories. 
Consider the set 
\[ \mathcal{O}_n :=\Big\{ (i,x) \in \llbracket  n/2,n \rrbracket\times \llbracket  \gep_n^{1/2}  h_n, 2 \gep_n^{1/2} h_n \rrbracket ; \gb_n \go_{i,x} \ge 2 x^2 /i  \Big\} \, .\]
If the set $\mathcal{O}_n$ is non-empty, then pick some $(i_0,x_0)\in \mathcal{O}_n$, and consider trajectories which visit this specific site: since all other weights are non-negative ($(\go-\mu)\ind_{\{\gb_n \go >1\}} \ge 0$ provided $\mu<1/\gb_n$), we get that
\begin{align}
\bar \bZ_{n,\gb_n}^{(\ge 1)} & \ge e^{\gb_n (\go_{i_0,x_0} -\mu)} \bP\big( S_{i_0} = x_0 \big) \notag \\
& \ge  \frac{c}{\sqrt{n}}  \exp\Big( \gb_n \go_{i_0,x_0}  - \frac{x_0^2}{i_0}\Big)  \ge \frac{c}{\sqrt{n}}  e^{ \gep_n \frac{h_n^2}{n} }\, .
\label{target}
\end{align}
We used Stone's local limit theorem \cite{S67} for the second inequality (valid provided that $n$ is large, using also that $i_0\ge n/2$). For the last inequality, we used the definition of $\mathcal{O}_n$ to bound the argument of the exponential by $x_0^2/i_0 \ge \gep_n h_n^2/n$.
Therefore, we get that
\begin{align*}
\bbP\Big( \bar \bZ_{n,\gb_n}^{(\ge 1)} \le   \frac{c}{\sqrt{n}}  e^{ \gep_n \frac{h_n^2}{n} } \Big)  & \le \bbP\big( \mathcal{O}_n = \emptyset \big) = \prod_{i=n/2}^n \prod_{x = \gep_n^{1/2} h_n}^{2\gep_n^{1/2} h_n} \Big( 1- \bbP \big( \gb_n\go >2 x^2/i  \big) \Big)\\
&\le \Big( 1- \bbP\big( \go > 4 \gep_n m(n h_n)\big) \Big)^{\gep_n^{1/2} n h_n}  \, .
\end{align*}
For the second inequality we used that $x^2/i \ge \gep_n h_n^2/n$ for the range considered, together with the relation \eqref{def:hn} characterizing $\gb_n$. Then, we use the definition of $m(nh_n)$ together with Potter's bound to get that for any fixed $\gd>0$, we have $\bbP\big( \go > 4 \gep_n m(n h_n)\big) \ge c \gep_n^{-\ga +\gd} (nh_n)^{-1}$, provided that $n$ is large enough.
Therefore, we obtain that 
\begin{equation}
\label{631}
\bbP \Big( \bar \bZ_{n,\gb_n}^{(\ge 1)} \le   \frac{c}{\sqrt{n}}  e^{ \gep_n \frac{h_n^2}{n} } \Big) \le   \exp\Big( - c\,  \gep_n^{\frac12 -\ga +\gd}\Big) \, ,
\end{equation}
which bounds the second term in \eqref{twotermsbarZmax}.
\end{proof}

\section{Regime 2 and regime 3-a}
\label{secProofeq:hscaling}
In this section we prove Theorem \ref{thm:alpha>12} and Theorem \ref{thm:cas3}. 
We decompose the proof in three steps, Step $1$ and Step $2$ being the same for both theorems.
For the third step, we give the details in regime 2, and adapt the reasoning to regime 3-a.

\subsection{Step 1: Reduction of the set of trajectories}
\label{sec:reduction2}
Recalling $\mu=\bbE[\go]$ (which is finite for $\alpha>1$), we define
\begin{equation}\label{def:Zbar}
\bar\bZ_{n,\beta_n}^\go := \bE\Big[\exp\Big(\sum_{i=1}^n \beta_n \big(\go_{i,S_i}-\mu \ind_{\{\alpha\ge 3/2\}}\big)\Big)\Big]
\end{equation}

We show that to prove Theorem \ref{thm:alpha>12} and Theorem \ref{thm:cas3}
we can reduce the problem to the random walk trajectories belonging to $\Lambda_{n, A h_n}$ for some $A>0$ (large). For any $A>0$, we define
\begin{equation}\label{defBAset}
\cB_n(A):=\Big \{(i,S_i)_{i=1}^n \colon \max_{i\leq n} |S_i|\le A h_n \Big \}
\end{equation}
and we let
\begin{equation}
\bar \bZ_{n,\beta_n}^\omega(\cB_n(A)):= \bE\Big[\exp\Big(\sum_{i=1}^n \beta_n \big(\go_{i,s_i}-\mu \ind_{\{\alpha\ge 3/2\}}\big)\Big)
\ind_{\cB_n(A)}\Big].
\end{equation}

Relation \eqref{eq:hscaling} gives that 
$\mathbb P \Big(\bP_{n,\beta_n}^\omega\big(\cB_n(A)\big) \ge n e^{-c_1 A^2 h_n^2/n} \Big)\le c_2 A^{-\nu_1}$, uniformly on $n\in \mathbb N$. This implies that
\begin{equation}
\label{eq:step1}
\bbP\bigg(\Big|\log \bar \bZ_{n,\beta_n}^\omega-\log \bar \bZ_{n,\beta_n}^\omega(\cB_n(A))\Big | \ge ne^{-c_1' A^2 h_n^2/n} \bigg) \leq c_2 A^{-\nu_1},
\end{equation}
uniformly on $n\in\mathbb N$. Let us observe that in {Regime 2 and regime 3-a} we have that $h_n^2/n \ge c_{\gb}\log n$, therefore $ne^{-c_1' A^2 h_n^2/n}$ goes to $0$ as $n$ gets large, provided $A$ is sufficiently large.

In such a way relation \eqref{eq:step1} implies
\begin{equation}\label{rel1thm23a}
\lim_{n\to\infty} \frac{n}{h_n^2}\log \bar \bZ_{n,\beta_n}^\omega=\lim_{A\to\infty}\lim_{n\to\infty}\frac{n}{h_n^2}\log \bar \bZ_{n,\beta_n}^\omega(\cB_n(A)).
\end{equation}

\subsection{Step 2: Restriction to large weights}
\label{largeweights2}
In the second step of the proof we show that we can only consider the partition function $\bZ_{n,\beta_n}^{\go, (\texttt L)}$ truncated to a finite number $\mathtt{L}$ of large weights,  iwth$\texttt L$ independent of $n$. 
We need some intermediate truncation steps.

We start by removing the small weights. Using the notations introduced in (\ref{def:ZbigT} -- \ref{def:Zsmall1}) and \eqref{def:barZsmall1},  H\"older's inequality gives that for any $\eta \in(0,1)$
\begin{align}\label{eq:step2-reg23a}
\Big(   \bar \bZ_{n,(1-\eta)\beta_n}^{(>1)} &\Big)^{\frac{1}{1-\eta}}  \Big(  \bar \bZ_{n,-(\eta^{-1}-1)\beta_n}^{(\le 1)} \Big)^{ - \frac{\eta}{1-\eta} }  
\\
&  \le \bar \bZ_{n,\beta_n}^\omega(\cB_n(A)) \le  
\Big(   \bar \bZ_{n,(1+\eta)\beta_n}^{(>1)} \Big)^{\frac{1}{1+\eta}}  \Big(  \bar \bZ_{n,(1+\eta^{-1})\beta_n}^{(\le 1)} \Big)^{ \frac{\eta}{1+\eta} } \, ,
\notag
\end{align}
We observe that 
the condition $\beta_n\omega>1$ implies (if $\mu<\infty$)
\begin{equation}\label{barreplacement}
(1-2\eta)\beta_n\omega \le (1-\eta)\beta_n(\omega-\mu) 
\ \text{ and }\
(1+\eta)\beta_n(\omega-\mu)\le (1+\eta) \beta_n \omega,
\end{equation}
provided $n$ is large enough. In such a way, we can safely replace $\bar\bZ_{n,(1-\eta)\beta_n}^{(>1)}$ by $\bZ_{n,(1-2\eta)\beta_n}^{(>1)}$ and $\bar\bZ_{n,(1+\eta)\beta_n}^{(>1)}$ by $\bZ_{n,(1+\eta)\beta_n}^{(>1)}$ in \eqref{eq:step2-reg23a}.
The next lemma shows that the contribution given by $\log\bar \bZ_{n,\rho\beta_n}^{(\le 1)}$ is negligible. 
\begin{lemma}\label{lem:ZBless1neg}
Let $\rho\in \mathbb R$. Then,
\begin{equation}
\frac{n}{h_n^2}\log\bar \bZ_{n,\rho\beta_n}^{(\le 1)}\overset{\bbP}{\to} 0, \quad \text{as}\, n\to\infty. 
\end{equation}
\end{lemma}
\begin{proof}
The case $\rho>0$ is a consequence of the estimate in \eqref{EZnsmall} and \eqref{MarkovJensen}, while the case $\rho<0$ is a consequence of the estimate in \eqref{EZnegativebeta} and \eqref{EZnegativebeta2}
\end{proof}


We can further reduce the partition function $\bZ_{n,\nu\beta_n}^{(>1)} $ to even (intermediate) larger weights (with $\nu>0$). 

We fix some $\gd>0$ small, and define 
$\ell :=  ( A^2 h_n^2/ n )^{1-\gd}$ and  also $\mathtt T = A^{1/\ga} \frac{h_n^2}{n} \ell^{-(1-\gd)^{1/2}/\ga}$
as in \eqref{def:ell0}: then, H\"older's inequality gives that for any $\eta \in(0,1)$
\[ \log \bZ_{n,\nu\beta_n}^{(>\mathtt T)}  \le \log \bZ_{n,\nu\beta_n}^{(>1)} \le \frac{1}{1+\eta}  \log \bZ_{n,(1+\eta)\nu\beta_n}^{(>\mathtt T)}  + \frac{\eta}{1+\eta}  \log \bZ_{n,(1+\eta^{-1})\nu\beta_n}^{((1,\mathtt T])}  \, . \]
Then, \eqref{aim:part2} gives that for any fixed $A\ge 1$, and since $h_n^2/n \to\infty$, we have that for any $\rho>0$,
\begin{equation}
\frac{n}{h_n^2} \log \bZ_{n,\rho \beta_n}^{((1,\mathtt T])} \overset{\bbP}{\to} 0, \quad \text{as}\, n\to\infty. 
\end{equation}

Finally we show that we can only consider a finite number of large weights.
We consider $\gU_\ell = \big\{ Y_1^{(n,Ah_n)}, \ldots, Y_{\ell}^{(n,Ah_n)}\big\}$ with $\ell$ chosen above. Using \eqref{eq:inclusion}, with probability larger $1- (c\ell)^{-\delta \ell/2}$ (with $\ell\to\infty$ as $n\to\infty$) we have that 
\[
\Xi_\texttt{T}:=
\Big\{ (i,x) \in \llbracket 1,n \rrbracket \times \llbracket -Ah_n ,Ah_n \rrbracket ; \gb_n \go_{i,x} > \mathtt T \Big\} \subset \gU_\ell
\]
and thus $ \bZ_{n,\nu\beta_n}^{(>\mathtt T)}\le \bZ_{n,\nu \beta_n}^{(\ell)}$ with high probability.
We let $\texttt L\in \mathbb N$ be a fixed (large) constant. Since $|\Xi_\texttt{T} |\to\infty$ as $n\to\infty$ in probability, we have that 
$\gU_{\texttt L}\subset \Xi_\texttt{T}$ so that, $\bZ_{n,\nu\beta_n}^{( \texttt L)}\le \bZ_{n,\nu\beta_n}^{(>\mathtt T)}$ for large $n$, with high probability.
By using H\"older's inequality we get, 
\[
\bZ_{n,\nu\beta_n}^{(\texttt L)}\le \bZ_{n,\nu \beta_n}^{(>\texttt T)}\le  \Big(\bZ_{n,\nu(1+\eta) \beta_n}^{(\texttt L)}\Big)^{\frac{1}{1+\eta}} \Big(\bZ_{n,\nu(1+\eta^{-1}) \beta_n}^{(\texttt L, \ell)}\Big)^{\frac{\eta}{1+\eta}},
\]
where
\begin{equation}
\bZ_{n, \beta_n}^{(\texttt L, \ell)}:=\bE\Big[ \exp\Big(\sum_{i=\texttt L+1}^{\ell}   \gb_n M_i^{(n,qh_n)} \ind_{\{Y_i^{(n,q h_n)} \in S\}}\Big) \Big] .
\end{equation}

We now show that the contribution of $\bZ_{n,\nu(1+\eta^{-1}) \beta_n}^{(\texttt L, \ell)}$ is negligible.
\begin{lemma}
For any $\epsilon\in (0,1)$ and for any $\mathtt L\in \mathbb N$ and $\rho>0$ there exists $\delta_{\mathtt L}$ such that for all $n$
\begin{equation}
\bbP\Big(\frac{n}{h_n^2}\log \bZ_{n,\rho \beta_n}^{(\mathtt L, \ell)}>\epsilon\Big)\leq \delta_{\mathtt L}   ,
\end{equation}
with $\gd_{\mathtt L} \to 0$ as $\mathtt L \to\infty$.
\end{lemma}
\begin{proof}
We let $\rho>0$. Recalling the definition \eqref{def:Omega}, and using that $\bP(\Delta \subset S)\le e^{\ent(\Delta)}$, we have that
\begin{align*}
\bZ_{n,\rho \beta_n}^{(\texttt L, \ell)} &\le  \sum_{\gD \subset \gU_\ell} e^{  \rho \gb_n \Omega_{n,qh_n}^{(>\texttt L)}(\Delta) } \bP\big( S \cap \gU_\ell =\Delta \big) \\
& \le \sum_{\gD \subset \gU_\ell}  \exp\Big( \rho\gb_n \Omega_{n,qh_n}^{(>\texttt L)}(\Delta) - \ent(\Delta) \Big) \le 2^{\ell} \exp \Big(  T_{n,Ah_n}^{\rho \gb_n,(>\texttt L)}\Big) \, .
\end{align*}
Using that $\ell=o(h^2/n)$ and relation (5.5) of \cite{cf:BT_ELPP}, we conclude the proof.
\end{proof}
Collecting the above estimates, we can conclude that
\begin{equation}\label{finallimit3}
\lim_{n\to\infty}\frac{n}{h_n^2} \log \bar \bZ_{n,\beta_n}^\omega(\cB_n(A)) 
=\lim_{\nu\to 1}\lim_{\texttt L\to\infty}\lim_{n\to\infty}\frac{n}{h_n^2}\log \bZ_{n,\nu\beta_n}^{(\texttt L)}  \, .
\end{equation}

\subsection{Step 3: Regime 2. Convergence of the main term}
\label{sec:2-Step3}
It remains to show the convergence of the partition function restricted to the large weights.

\begin{proposition}\label{prop:convBZbig1}
For any $\nu>0$, and $\mathtt L>0$
\begin{equation}\label{convBZbig1}
\frac{n}{h_n^2}\log \bZ_{n,\nu \beta_n}^{(\mathtt{L})} \stackrel{(\dd)}{\to} 
\begin{cases}
\cT_{\nu,A}^{(\mathtt L)} &\quad \text{in Regime 2}, \\ 
\tilde \cT_{\gb,\nu,A}^{(\mathtt L)}&\quad  \text{in Regime 3-a}, 
\end{cases}
\end{equation}
where $\cT_{\gb,A}^{(\mathtt L)}$ was introduced in \eqref{conv:largeweights} and $\tilde \cT_{\beta,\nu,A}^{(\mathtt L)}$ is defined in \eqref{def:tildeTA} below.
\end{proposition}

One readily verifies that

$\ast$ $\nu\mapsto \cT_{\nu,A}^{(\mathtt L)}$ (resp. $\nu\mapsto \tilde\cT_{\gb,\nu,A}^{(\mathtt L)}$) is a continuous function;

$\ast$ $\cT_{1,A}^{(\mathtt L)}\to \cT_{1,A}$ (resp. $\tilde \cT_{\gb,1,A}^{(\mathtt L)}\to \tilde \cT_{\gb, 1,A}$) as $\texttt L\to\infty$ (see Proposition~\ref{prop:ConvVP}, resp.~Proposition~\ref{prop:ConvVPtilde});

$\ast$  $\cT_{1,A}\to \cT_{1}$ (resp. $\tilde \cT_{\gb,1,A}\to \tilde \cT_{\gb}$) as $A\to\infty$ (see Proposition~\ref{prop:ConvVP}, resp. Proposition~\ref{prop:ConvVPtilde}).

Therefore, the proof of Theorem \ref{thm:alpha>12} and Theorem \ref{thm:cas3} is a consequence of relations \eqref{rel1thm23a}, \eqref{finallimit3} and \eqref{convBZbig1}.

\begin{proof}
We detail the proof for the Regime 2. The Regime 3-a follows similarly using the results in Section \ref{complement3a} below.
To keep the notation lighter we let $\nu=1$.

\emph{Lower bound.} 
For any
${\texttt L}\in \mathbb N$ we consider a set $\Delta_{\texttt L}\subset \gU_{\texttt L}$ which achieves the maximum of $T_{n, Ah_n}^{\beta_n, ({\texttt L})} $, resp. of $\tilde T_{n,Ah_n}^{\gb_n,({\texttt L})}$ defined below in \eqref{def:discreteELPPtilde} for Regime 3-a. We have 
\[
\bZ_{n,\beta_n}^{({\texttt L})}\ge \exp\Big(\beta_n\Omega_{n, A h_n}(\Delta_{\mathtt L})\Big)\, \bP \big( S \cap \gU_{\texttt L}=\Delta_{\mathtt L} \big) \, .
\]
Since ${\texttt L}$ is fixed, we realize that any pair of points $(i,x),(j,y)\in \gU_{\texttt L}$ satisfies the condition $|i-j|\ge \gep n$ and $|x-y|\ge\gep h_n$ with probability at least $1-c_\gep$ with $c_\gep\to 0$ as $\gep\to 0$.
In such a way, we can use the Stone local limit theorem \cite{S67} to get 
that $\bP(S \cap \gU_{\texttt L}=\Delta_{\mathtt L} ) = n^{-\frac{|\Delta_{\mathtt L}|}{2} +o(1)}e^{-\ent(\Delta_{\mathtt L})}$. In the Regime 2, in which $\ent(\Delta_{\mathtt L}) \asymp h_n^2/n \gg \log n$, this implies that
\begin{equation}
\label{end-lowerbound}
\bZ_{n,\beta_n}^{({\texttt L})}
\ge
\exp\Big((1+o(1))T_{n,Ah_n}^{\gb_n,({\texttt L})}\Big).
\end{equation}
To conclude, we use Proposition \ref{prop:ConvVP}-\eqref{conv:largeweights} to obtain
that $T_{n,Ah_n}^{\gb_n,({\texttt L})}$ converges in distribution to $ \cT_{1 ,A}^{(\texttt L)}$, concluding the lower bound.

In Regime 3-a, \eqref{end-lowerbound} is replaced by
\begin{equation}
\bZ_{n,\beta_n}^{({\texttt L})}
\ge
\exp\Big( (1+o(1)) \Big\{ \beta_n\Omega_{n, A h_n}(\Delta_{\mathtt L}) - \ent(\Delta_{\mathtt L}) - \frac{|\Delta_{\mathtt L}|}{2} \log n \Big\} \Big),
\end{equation}
so that
$T_{n,Ah_n}^{\gb_n,({\texttt L})}$ is replaced by $\tilde T_{n,Ah_n}^{\gb_n,({\texttt L})}$ defined in \eqref{def:discreteELPPtilde}. Then the conclusion follows by Proposition \ref{prop:ConvVPtilde}-\eqref{def:tildeTAL} below.

\emph{Upper bound.}
We have
\begin{align*}
\bZ_{n,\beta_n}^{(\texttt L)} &=  \sum_{\gD \subset \gU_{\texttt L}} e^{  \gb_n \Omega_{n,qh_n}^{(\texttt L)}(\Delta) } \bP\big( S \cap \gU_{\texttt L} =\Delta \big) 
\end{align*}
Using the Stone local limit theorem \cite{S67} we have that $\bP(S \cap \gU_{\texttt L}=\Delta ) = n^{-\frac{|\Delta|}{2} +o(1)} e^{-\ent(\Delta)}$ uniformly for all $\Delta \subset \gU_{\mathtt L}$. 
Since we have only a finite number of sets, we obtain that
\begin{equation}
\label{end-upperbound}
\bZ_{n,\beta_n}^{(\texttt L)}\le 2^{\texttt L} \exp \Big( (1+o(1)) T_{n,Ah_n}^{\gb_n,(\texttt L)}\Big),
\end{equation}
which concludes the proof of the upper bound, again thanks to the convergence proven in Proposition~\ref{prop:ConvVP}-\eqref{conv:largeweights}.
In Regime 3-a, using the Stone local limit theorem, we can safely replace $T_{n,Ah_n}^{\gb_n,(\texttt L)}$ by 
$\tilde T_{n,Ah_n}^{\gb_n,(\texttt L)}$ defined below in \eqref{def:discreteELPPtilde}, and also conclude thanks to Proposition~\ref{prop:ConvVPtilde}-\eqref{def:tildeTAL}.
\end{proof}

\subsection{Step 3: Regime 3.a. Complements for the convergence of the main term}\label{complement3a}
We end here the proof of Theorem~\ref{thm:cas3} by stating the results needed to complete Step 3 above in the case of regime 3.a.
In analogy with  \eqref{def:discreteELPP}, and in view of the local limit theorem \eqref{LLT}, we define 
\begin{equation}
\label{def:discreteELPPtilde}
\begin{split}
&\tilde T_{n,h}^{\gb_{n,h}} := \max_{ \Delta \subset \Lambda_{n,h}} \big\{  \gb_{n,h} \Omega_{n,h} (\Delta) - \ent(\Delta) - \frac{|\Delta|}{2}\log n  \big\} \, ,\\
& \tilde T_{n,h}^{\gb_{n,h}, (\ell)} := \max_{ \Delta \subset \Lambda_{n,h}} \big\{  \gb_{n,h} \Omega_{n,h}^{(\ell)} (\Delta) - \ent(\Delta) - \frac{|\Delta|}{2}\log n  \big\} \, 
\end{split}
\end{equation}
In the next result we state the convergence of $\frac{n}{h^2}\tilde T_{n,h}^{\gb_{n,h}}$ and $\frac{n}{h^2}\tilde T_{n,h}^{\gb_{n,h},(\ell)}$, analogously to Proposition~\ref{prop:ConvVP}. 
\begin{proposition}\label{prop:ConvVPtilde}
Suppose that $  \frac{n}{h^2}\gb_{n,h} m(nh) \to \nu \in (0,\infty)$ as $n,h\to\infty$ and $h\sim \beta^{1/2}\sqrt{\log n}$, with $\beta>0$.  Then, for every $\ga\in (1/2,2)$ and for any $q>0,\, \ell\in \mathbb N$ we have the following convergence in distribution, as $n\to\infty$
\begin{equation}
\label{def:tildeTA}
\frac{n}{h^2}\, \tilde T_{n,qh}^{\beta_{n,h}} \stackrel{\rm (d)}\longrightarrow \tilde\cT_{\gb,\nu,q}:= \sup_{s\in\sM_q}\Big\{\nu \pi(s)-\ent(s)-\frac{N(s)}{2\beta} \Big\} \, ,
\end{equation}
with $\sM_q$ as defined in Proposition~\ref{prop:ConvVP}.
We also have, as $n\to\infty$
\begin{equation}
\label{def:tildeTAL}
\frac{n}{h^2}\, \tilde T_{n,qh}^{\beta_{n,h}, (\ell)} \stackrel{\rm (d)}\longrightarrow \tilde\cT_{\gb,\nu,q}^{(\ell)} :=\sup_{s\in\sM_q}\Big\{\nu \pi^{(\ell)}(s)-\ent(s)-\frac{N(s)}{2\beta} \Big\}\, .
\end{equation}
Moreover, we have
$\tilde\cT_{\gb,\nu,q}^{(\ell)}\overset{(\dd)}{\to} \tilde\cT_{\gb,\nu,q}$ 
as $\ell\to\infty$, and $\tilde\cT_{\gb,\nu,q}\overset{(\dd)}{\to} \tilde\cT_{\gb,\nu}$  as $q\to\infty$.
\end{proposition}
The proof is identical to the proof of Proposition \ref{prop:ConvVP} (cf.~proof of \cite[Theorem~2.7]{cf:BT_ELPP}, using also that $\frac{n}{h_n^2}\log n \to \frac{1}{\gb}$ in regime~3), for this reason it is omitted.
To conclude, let us show that $\tilde \cT_\beta^{(\ge r)}$ defined in \eqref{def:tildeT} is well defined. 
\begin{proposition}\label{propTtilde} 
For any $r\ge 0$ the quantities $\tilde \cT_\beta^{(\ge r)}$ are well defined and for any $\gb > 0$
\begin{equation}
\label{sandwichTgb}
-\frac{1}{2\gb} <  \tilde \cT_\beta^{(\ge 1)}\le  \tilde \cT_\beta <\infty.
\end{equation}
Moreover  $\tilde \cT_{\gb} \ge 0$, and we have $\tilde \cT_\gb >0$ if and only if $\tilde \cT_\gb^{(\ge 1)} >0$. Finally, there is a critical value
$\gb_c = \inf\{ \gb \colon \tilde \cT_\gb >0 \}  \in  (0,\infty).$
\end{proposition}
\begin{proof}
Since $\tilde \cT_{\gb}^{(0)} =0$, we obtain that $\tilde \cT_{\gb} \in [0,\infty)$. As a by-product we also have that
$\tilde \cT_{\gb}>0 $ if and only if $ \tilde \cT_{\gb}^{(\ge 1)}>0$; and in that case $\tilde \cT_{\gb} =\tilde \cT_{\gb}^{(\ge 1)}$.
Additionally, we have
\[W_\beta-\frac{1}{2\beta}\le  \tilde \cT_\beta^{(\ge 1)}\le \tilde \cT_\beta\le   \Big( \cT_1 -\frac{1}{2\beta} \Big)\vee 0 ,\]
with $W_\beta$ and $\cT_1$ defined in \eqref{def:W} and \eqref{def:T} respectively. Proposition \ref{prop:W} and Theorem \ref{thm:TbhatTb} ensure that for $\gb>0$, $W_\beta \in (0,\infty)$ and $\tilde \cT_{1} <\infty$, showing \eqref{sandwichTgb}.

It remains to show that $\gb_c \in (0,\infty)$, by observing that $\beta\mapsto \beta W_\beta$ and  $\gb \mapsto (\gb \cT_1 -1/2 )\vee 0 $ are monotone functions which converge to $0$ as $\beta\to 0$.
\end{proof}

\section{Regime 3-b and regime 4}
\label{sec:3b4}

In this section we prove Theorem \ref{thm:cas3bis} and Theorem \ref{thm:cas4}. 
We decompose the proof in three steps (analogously to what is done in Section~\ref{secProofeq:hscaling}), Step $1$ and Step $2$ being the same for both regimes 3-b and 2.
For the third step, we separate regime 3-b and regime 4, which have different behaviors.
Note that in both regimes there is a constant $c_{\gb}>0$ such that  $h_n \le  c \sqrt{ n\log n}$ (in regime 4, we have $h_n \ll \sqrt{n\log n}$).

Let us  define here, analogously to \eqref{def:Zbar}, the recntered partition function
\begin{equation}
\label{def:Zbar2}
\bar \bZ_{n,\gb_n}^{\go} := \bE\Big[ \exp\Big( \sum_{i=1}^n  \gb_n \big( \go_{i,s_i} - \bbE[\go \ind_{\go\le 1/\gb_n}] \ind_{\{\ga\ge 1\}} \big) \Big] \, . 
\end{equation}
Then, roughly speaking, we show that  $\log \bar \bZ_{n,\gb_n}^{\go}$ is of order  $n^{-1/2}\exp( X h_n^2/n) $, with  $X = \tilde \cT_{\gb}^{(\ge 1)} +\frac{1}{2\gb}$ in the regime 3-b (where $h_n^2/n \sim \gb \log n$), and with $X=W_1$ in regime~4. In all cases, we will have $\log \bar \bZ^{\go}_{n,\gb_n} =o(1)$ (recall that in regime 3-b, $\tilde \cT_{\gb}^{(\ge 1)}<0$).

\subsection{Step 1. Reduction of the set of trajectories}
\label{sec:4-step1}

We proceed as for Step 1 in Section \ref{secProofeq:hscaling}: for any $A>0$ (fixed large in a moment), we define
\begin{equation}
\label{def:An}
\cA_n := \Big\{ (i,S_i) \, : \, \max_{i\le n} |S_i| \le A \sqrt{ n \log n}\Big\} \, .
\end{equation}
Then, we let $\bar\bZ_{n,\gb_n}^{\go}(\cA_n) $ be the (normalized) partition function restricted to trajectories in~$\cA_n$. Relation \eqref{eq:hscaling} gives that, analogously to \eqref{eq:step1}
\begin{equation}
\bbP\bigg(\Big|\log \bar \bZ_{n,\beta_n}^\omega-\log \bar \bZ_{n,\beta_n}^\omega(\cA_n)\Big | \ge n e^{-c_1 A^2 \log n } \bigg) \leq c_2 A^{-\nu_1}\, .
\end{equation} 
Hence, we fix $A$ large enough so that $e^{- c_0 A^2 \log n} \le n^{-3}$.
This shows that with high probability $\log  \bar \bZ_{n,\beta_n}^\omega = \log \bar \bZ_{n,\beta_n}^\omega(\cA_n) + O(n^{-2})$.
In such a way, in the following we can safely focus only on the partition function with trajectories restricted to $\cA_n$.

\subsection{Step 2. Restriction to large weights}
\label{largeweights3}

We now fix  $\eta \in(0,1)$, small. The same H\"older inequalities as in \eqref{eq:step2-reg23a} hold for $\bZ_{n,\beta_n}^\omega(\cA_n) $, 
so that we can write, with similar notations as in \eqref{def:ZbigT}-\eqref{def:Zsmall1}  (the restriction to trajectories in $\cA_n$ does not appear in the notations)
\begin{align}
\label{Holder:reg34}
\log  \bar \bZ_{n,\beta_n}^\omega(\cA_n)\, \left\{  
\begin{aligned}
&\le  \frac{1}{1+\eta} \log  \bZ_{n,(1+\eta)\beta_n}^{(>1)}  + \frac{\eta}{1+\eta} \log  \bar \bZ_{n,(1+\eta^{-1})\beta_n}^{(\le1)} \, 
, \\ 
&\ge   \frac{1}{1-\eta} \log  \bZ_{n,(1-2\eta)\beta_n}^{(>1)} - \frac{\eta}{1-\eta} \log  \bar \bZ_{n,-(\eta^{-1}-1)\beta_n}^{(\le1)} \, .
\end{aligned}
\right. 
\end{align}
We used also \eqref{barreplacement} to be able to bound below $ \bar \bZ_{n,(1-\eta)\beta_n}^{(>1)}$ by $ \bZ_{n,(1-2\eta)\beta_n}^{(>1)}$ (using that $\gb_n  \bbE[\go \ind_{\{\go\le 1/\gb_n\}}] \ll 1$ when $\ga\ge 1$).
Then, we need to get a more precise statement than Lemma~\ref{lem:ZBless1neg} to deal with $\bar \bZ_{n,\rho \gb_n}^{(\le 1)}$.

\begin{lemma}
\label{lem:logbarZ}
For any $\rho \in \bbR$,
\[ \Big( \frac{h_n^2}{n} \Big)^{-3\ga } \sqrt{n}\log \bar \bZ_{n,\rho \gb_n}^{(\le 1)} \stackrel{\bbP}{ \to} 0 \, , \qquad \text{as } n\to\infty \, .\]
\end{lemma}
\begin{proof}
We will simply control the first moment of $\bar \bZ_{n,\rho \gb_n}^{(\le 1)} -1$. The idea is similar to that used to obtain \eqref{aim:part3} and \eqref{aim:part3-bar}.
We divide the proof into two cases: when $\ga<1$ so that there is no renormalization necessary in \eqref{def:Zbar2}, and when $\ga\in [1,2)$.

Let us start with the case $\ga<1$: using that
$ |\rho| \gb_n \go_{i,S_i} \le |\rho|$ on the event $\{\gb_n \go_{i,S_i} \le 1\}$, we get that there exists a constant $c_{\rho}$ such that 
\begin{equation}
\label{develop:exp}
e^{ \sum_{i=1}^n \rho \gb_n \go_{i,S_i} \ind_{\{\gb_n \go_{i,S_i} \le 1\}} } \le \prod_{i=1}^n \big(1+ c_{\rho}   \gb_n \go_{i,S_i} \ind_{\{\gb_n \go_{i,S_i} \le 1\}}  \big) \, .
\end{equation}
By independence, and since $\bbP(\go >t )$ is regularly varying, we get that for $n$ sufficiently large
\begin{align}
\bbE[ \gb_n \go_{i,x} \ind_{\{\gb_n \go_{i,x} \le 1\}} ] &\le \int_{0}^{1/\gb_n} \gb_n \bbP(\go >t) dt \le c \, L(1 /\gb_n ) \gb_n^{\ga}  \nonumber \\
& \le  c \bbP \big( \go >1/\gb_n \big) \le  \frac{c'}{ n h_n } \Big( \frac{h_n^2}{n}\Big)^{2\ga} \, .
\label{truncmean}
\end{align}
For the last inequality we used Potter's bound, and the definition of $\gb_n$, \textit{i.e.}\ the fact that  $\gb_n \sim \frac{h_n^2}{n} m(n h_n)$.
Therefore, in view of \eqref{develop:exp} and using that $h_n\ge \sqrt{n}$, we get that for $n$ sufficiently large (how large depends on $\rho$)
\begin{equation}
\label{eq:EZ-1}
\bbE\big[\bar \bZ_{n,\rho \gb_n}^{(\le 1)} -1  \big] \le \Big(1+ c'_{\rho} \frac{\big(h_n^2/n\big)^{2\ga} }{ n^{3/2}  }  \Big)^n -1 \le  2 c'_{\rho} n^{-1/2}   \Big( \frac{h_n^2}{n} \Big)^{2\ga }\, .
\end{equation}
This concludes the proof in the case $\ga<1$ by using Markov's inequality, since $h_n^2/n \to +\infty$.

In the case $\ga\in[1,2)$, we use the expansion $e^x \le 1+x+ c_{\rho} x^2$ for all $|x|\le 2 |\rho|$, to get analogously to \eqref{develop:exp}, and setting $\mu_n:= \bbE[\go \ind_{\{\go \le 1/\gb_n\}}] \ll 1/\gb_n$,
\begin{align*}
\bbE \Big[\bar \bZ_{n,\rho\gb_n}^{(\le1)}\Big] & \le  \Big( 1+ \rho \gb_n \bbE \big[ (\go-\mu_n)  \ind_{\{  \go  \le 1/\gb_n\}} \big] + c_{\rho}\gb_n^2 \bbE \big[ (\go-\mu_n)^2  \ind_{\{  \go  \le 1/\gb_n\}} \big] \Big)^n \\
&\le \exp\Big( c\, n \bbP(\go >1/\gb_n) \Big) \le 1+ c n^{-1/2}   \Big( \frac{h_n^2}{n} \Big)^{2\ga }\, ,
\end{align*}
obtaining the same upper bound as in \eqref{eq:EZ-1}.
To obtain the above inequality, we used that
\begin{align*}
&\bbE[ (\go -\mu_n) \ind_{\{ \go \le 1/\gb_n\}} ] = \mu_n \bbP(\go>1/\gb_n) \le \gb_n^{-1} \bbP(\go>1/\gb_n)  \, ,\\
& \bbE[ (\go -\mu_n)^2 \ind_{\{  \go \le 1/\gb_n\}} ] \le \bbE[ \go^2 \ind_{\{  \go \le 1/\gb_n\}} ] \le c L(1/\gb_n) \gb_n^{\ga-2} \, ,
\end{align*}
where the last inequality follows similarly to \eqref{truncmean}. 
One concludes that \eqref{eq:EZ-1} also holds when $\ga\ge 1$, and the lemma follows by Markov's inequality.
\end{proof}

Therefore, in view of \eqref{Holder:reg34} and Lemma \ref{lem:logbarZ}, we have that for both regimes 3-b and 4
\begin{equation}
\label{eq:finallimit3b}
\lim_{n\to\infty}  \frac{n}{h_n^2}\log \Big( \sqrt{n} \log  \bar \bZ_{n,\beta_n}^\omega(\cA_n) \Big) =  \lim_{\nu\to 1} \lim_{n\to\infty} \frac{n}{h_n^2}\log \Big(\sqrt{n} \log   \bZ_{n,\nu\beta_n}^{(>1)} \Big)\, .
\end{equation}
Note that in the case of regime 3-b, $h_n^2/n \sim \gb \log n$, so the limit is that of 
\[\frac{1}{\gb \log n} \log \Big( \log \bZ_{n,\nu\beta_n}^{(>1)}   \Big) + \frac{1}{2\gb} \, .\]
For simplicity of notations, we will consider only the case $\nu=1$ in the following.

\subsection{Step 3. Reduction of the main term}
\label{reductionlogZ}
In both regimes 3-b and 4, we show that $\log \bZ_{n,\gb_n}^{(>1)}$ goes to $0$, and we identify at which rate: to do so, it is equivalent to identify the  rate at which $\bZ_{n,\gb_n}^{(>1)}-1$ goes to $0$.
The behavior for regimes~3-b and 4 are different, since the
main contribution to $\bZ_{n,\gb_n}^{(>1)}-1$  may come from
several large weights in \mbox{regime~3-b,} whereas it comes from a
single large weight in regime 4, as it will be reflected in the proof.

Let us define  $\ell =\ell(\go)$ the number of $(i,x) \in  \Lambda_{n,A_n}=\llbracket 1,n \rrbracket \times  \llbracket -A_n , A_n\rrbracket$ (with the notation  $A_n =A\sqrt{n\log n}$ for simplicity) such that $\gb_n \go_{i,x} \ge 1 $, and let us denote
\begin{align}
\label{def:gUell}
\Big\{ (i,x) \in \Lambda_{n,A_n} & ; \gb_n \go_{i,x} \ge 1 \Big\} = \gU_\ell := \big\{ Y_1^{(n,A_n)}, \ldots, Y_{\ell}^{(n,A_n)}\big\} \, ,
\end{align}
with $Y_i^{(n,A_n)}$ the ordered statistic, as in Section~\ref{sec:3}.
We have that
\begin{equation}\label{Eell}
\bbE[\ell] = \sum_{(i,x)\in \Lambda_{n,A_n} } \bbP(\gb_n \go_{i,x} \ge 1)   \le 2A n^{3/2} \sqrt{\log n} \Big(\frac{h_n^2}{n} \Big)^{2\ga} \frac{1}{n h_n} \,,
\end{equation}
where we used that $ \bbP(\go \ge 1/\gb_n) \le (h_n^2/n)^{2\ga} (nh_n)^{-1}$ for $n$ large enough, thanks to \eqref{def:hn} and  Potter's bound.
Since $h_n^2/n \le c \log n$, $h_n \gg \sqrt{n}$, \eqref{Eell} implies that $\ell\le (\log n)^{3\ga}$ with probability going to $1$ (we also used that $\frac12+2\ga<3\ga$). 

Hence, decomposing $\bZ_{n,\gb_n}^{(>1)}$ according to the number of sites in $\gU_{\ell}$ visited, we can write  for any fixed $k_0>0$,
\begin{align}\label{eq:lubZ3a}
& \sum_{k=1}^{k_0} \bU_k  \le \bZ_{n,\gb_n}^{(>1)}-1 =\sum_{k=1}^{\ell} \bU_k\, , \\
\text{with }\  & \bU_k :=  \sum_{\Delta \subset \gU_\ell , |\Delta|=k} e^{\gb_n\Omega_{n, A_n} (\Delta)} \bP\big( S\cap \gU_\ell = \Delta \big) \, .\notag
\end{align}
In regime 3-b, the main contribution comes from one of the $\bU_k$'s for some $k\ge 1$, whereas in regime 4 only the term $\bU_1$ will contribute.



Let us now show that, with high probability, we can replace the upper bound in \eqref{eq:lubZ3a} by considering only a finite number of terms. 
For this purpose, notice that $\ell\le (\log n)^{3\ga}$ and $\min\{ |i-j| , (i,x) \neq (j,y) \in \gU_\ell \}  \ge  n/ (\log n)^{10\ga}$ with probability going to $1$. Then,
we can use the Stone local limit theorem \cite{S67} to have that for any $\Delta\subset \gU_\ell$
\[
\bP\big( S\cap \gU_\ell = \Delta \big)\le c n^{-(\frac12-\eta) |\Delta|}e^{-\ent(\Delta)}\, ,
\]
where $\eta>0$ is independent of $\Delta$ and can be chosen
arbitrary small (by changing the value of the constant $c$). 

As a consequence, using that $\binom{\ell}{k}\le \ell^k$ and $\ell\le (\log n)^{3\ga}$, we have for any $1\le k_1\le \ell$
\begin{align}\label{eq:morethenk1-3a}
\sum_{k=k_1}^{\ell} \bU_k &=\sum_{k=k_1}^\ell\sum_{\Delta \subset \gU_\ell , |\Delta|=k}  e^{\gb_n\Omega_{n, A_n} (\Delta)} \bP\big( S\cap \gU_\ell = \Delta \big)   \\\nonumber
&\le  e^{T_{n,A_n}^{\beta_n}} \sum_{k=k_1}^\ell \ell^k \, n^{-k(\frac12-\eta)} \le  c\, e^{T_{n,A_n}^{\beta_n}} \, n^{-k_1(\frac12-\eta')}.
\end{align}
Recalling Proposition \ref{prop:ConvVP} (and the fact that $h_n^2/n \le c \log n$) we have that $T_{n,A_n}^{\beta_n}\le C\log n$ with probability going to $1$ as $C\to\infty$. Therefore, we obtain that \eqref{eq:morethenk1-3a} is $O(n^{-2})$  with probability close to~$1$, provided that $k_1$ is sufficiently large  -- this will turn out to be negligible, see Lemma~\ref{lem:cas3-mainterm}.
Hence, we have shown that with probability close to $1$, we can keep a finite number of terms in \eqref{eq:lubZ3a}.

This can actually be improved in regime 4,  where we can keep only one term: indeed, since in that case $h_n^2/n = o(\log n) $, we get that for any fixed $\gamma>0$, $T_{n,A_n}^{\beta_n}\le \gamma \log n$ with probability going to one. Hence, we get that in regime 4, we can take $k_1=2$ in \eqref{eq:morethenk1-3a} and obtain that  
$\sum_{k=2}^\ell \bU_k  = O(n^{-3/4})$ with probability close to $1$, which will turn out to be negligible, see Lemma ~\ref{lem:cas4-mainterm}.

\smallskip
It remains to show the following lemmas, proving the convergence of the main term in regimes 3-b and 4.
\begin{lemma}
\label{lem:cas3-mainterm}
In regime 3 \eqref{reg3} (recall $h_n^2/n \sim \gb \log n$), for any $K>0$ we have that
\begin{equation}
\label{eq:convcas3}
\frac{n}{h_n^2} \log \Big( 
\sum_{k = 1}^{K} 
\bU_k \Big) \stackrel{({\rm d})}{\longrightarrow}  \sup_{1\le k\le K}   \tilde  \cT_{\gb,A}^{(k)}\, ,
\end{equation}
where
$\tilde\cT_{\gb,A}^{(k)} := \sup_{s \in \sM_A , N(s)=k} \big\{ \pi(s)  - \ent(s) - \frac{k}{2\gb}\big\}\, , $
with $\sM_A$ defined below \eqref{def:tildeTA}. 
\end{lemma}
Note that we have $ \sup_{k\ge 1} \tilde  \cT_{\gb,A}^{(k)} <0$ in regime 3-b: this lemma proves that $\sum_{k = 1}^{K} 
\bU_k$ goes to $0$ in probability, and hence $\bZ_{n,\gb_n}^{(>1)}-1$ also goes to $0$ in probability. This is needed to replace the study of $\log \bZ_{n,\gb_n}^{(>1)}$ by that of $\bZ_{n,\gb_n}^{(>1)}-1$, and it is actually the only place where the definition of regime 3-b is used. 

\begin{lemma}
\label{lem:cas4-mainterm}
In regime 4 \eqref{reg4}, we have that
\begin{equation}
\frac{n}{h_n^2} \log \Big( \sqrt{n}\, \bU_1 \Big)  \stackrel{({\rm d})}{\longrightarrow}  W_1 \, ,
\end{equation}
with $W_1$ defined in \eqref{def:W}.
\end{lemma}
Here also, this proves that $\bU_1 \to 0$ in probability, and hence so does $\bZ_{n,\gb_n}^{(>1)}-1$.

\subsection{Regime 3-b: convergence of the main term}
\label{sec:prooflem3b}

In this section, we prove Lemma \ref{lem:cas3-mainterm}.

\subsubsection{Reduction to finitely many weights}
First of all, we fix some $\mathtt L$ large 
and show that the main contribution comes from the $\texttt L$ largest weights.
We define
\begin{equation}
\bU_k^{(\mathtt L)} := \sum_{\Delta \subset \gU_{\mathtt L}, |\Delta| =k} e^{\gb_n \Omega_{n,A_n}(\Delta)} \bP\big( S\cap \gU_{\ell} =\Delta \big) \, ,
\end{equation}
where $\gU_{\mathtt L} = \{ Y_1^{n,A_n}, \ldots, Y_{\mathtt L}^{n,A_n}\}$ is the set of $\mathtt L$ largest weights in $\Lambda_{n,A_n}$ (note that $\gU_{\mathtt L} \subset \gU_{\ell}$ for $n$ large enough).
Then we have that $\bU_k \ge \bU_k^{(\mathtt L)}$, and  $ \sum_{k=1}^K \bU_k $ is bounded by
\begin{align*}
&\sum_{k=1}^{K} \sum_{\Delta \subset \gU_{\mathtt{L}} , |\Delta| =k}  \sum_{ \Delta' \subset \gU_{\ell}\setminus \gU_{\mathtt L}, | \Delta'| \le K} 
e^{\gb_n \Omega_{n,A_n}(\Delta) +\gb_n \Omega_{n,A_n}(\Delta')} \bP\big( S\cap \gU_{\ell} =\Delta \cup \Delta' \big) \\
& \le \sum_{k=1}^{K} \sum_{\Delta \subset \gU_{\mathtt{L}} , |\Delta| =k} e^{\gb_n \Omega_{n,A_n}(\Delta)} \bP\big( S\cap \gU_{\mathtt L} =\Delta \big) \times \exp\Big(  K \gb_n M_{\mathtt L}^{(n,A_n)}\Big)\, \\
&=\exp\Big(  K \gb_n M_{\mathtt L}^{(n,A_n)}\Big) \sum_{k=1}^K \bU_k^{(\mathtt L)}.
\end{align*}
In the second inequality, we simply bounded $\Omega_{n,A_n}(\Delta')$ by $K M_{\mathtt L}^{(n,A_n)}$ uniformly for $\Delta' \subset \gU_{\ell}\setminus \gU_{\mathtt L}$, with $|\Delta'|\le K$.
Then, since  $\gb_n\sim c_{\gb} (\log n) / m(n h_n) \sim c_{\gb,A} (\log n) / m(nA_n)$ as $n\to\infty$, 
we get that $K \gb_n M_{\mathtt L}^{(n,A_n)}$ is bounded above by $ 2c_{\gb,A} K M_{\mathtt L}^{(n,A_n)} /m(nA_n) \times \log n $.
For any fixed $\gep> 0$, we can fix $\mathtt L$ large enough so that for large $n$  we have
$M_{\mathtt L}^{(n,A_n)} /m(nA_n) \le \gep /(2Kc_{\gb,A})$  with probability larger than $1-\gep$. We conclude that there exists some $\gep_{\mathtt L}$ with $\gep_{\mathtt L} \to 0$ as $\mathtt L\to\infty$ such that 
\begin{align*}
0\le \sum_{k=1}^K (\bU_k -\bU_k^{(\mathtt L)} ) \le   n ^{\gep_{\mathtt L}} \sum_{k=1}^{K}   \bU_k^{(\mathtt L)} \, .
\end{align*}
Since $h_n^2/n \sim \gb \log n$, this proves that
\begin{equation}
\lim_{n\to\infty} \frac{n}{h_n^2} \log \Big( \sum_{k=1}^K \bU_k \Big) = \lim_{\mathtt L \to\infty} \lim_{n\to\infty}   \frac{n}{h_n^2} \log \Big( \sum_{k=1}^K \bU_k^{(\mathtt L)} \Big) \, . 
\end{equation}

\subsubsection{Convergence of the remaining term}
We finally prove that
\begin{equation}
\label{cas3-bfinal}
\frac{n}{h_n^2}\log \Big( \sum_{k=1}^K \bU_k^{(\mathtt L)} \Big) \stackrel{({\rm d})}{\longrightarrow}   \max_{1\le k\le K}\tilde  \cT_{\gb,A}^{(k,\mathtt L)}
\end{equation}
where $\tilde  \cT_{\gb,A}^{(k,\mathtt L)}$ is the restriction of $\tilde\cT_{\gb,A}^{(k)}$ to the $\mathtt L$ largest weights in $[0,1]\times [-A,A]$, that is
\[\tilde  \cT_{\gb,A}^{(k,\mathtt L)} := \sup_{s \in \sM_A , N(s)=k} \Big\{ \pi^{(\mathtt L)}(s)  - \ent(s) - \frac{k}{2\gb}\Big\} \]
In analogy with Proposition \ref{prop:ConvVPtilde}, one shows that $\tilde  \cT_{\gb,A}^{(k,\mathtt L)} \to \tilde  \cT_{\gb,A}^{(k)}$ as $\mathtt L \to\infty$, which completes the proof.

The proof of \eqref{cas3-bfinal} comes from the rewriting
\begin{align*}
\sum_{k=1}^K \bU_k^{(\mathtt L)} &= \sum_{\Delta \subset \gU_{\mathtt L}, |\Delta| \le K} e^{\gb_n \Omega_{n,A_n}(\Delta)} \bP\big(  S\cap \gU_{\mathtt L} = \Delta\big) \\
&=   \sum_{\Delta \subset \gU_{\mathtt L}, |\Delta| \le K}  \exp\Big( \gb_n \Omega_{n,A_n}(\Delta) - \ent(\Delta) -  \frac{|\Delta|}{2} \log n  + o (K) \Big)\, ,
\end{align*}
where for the last inequality we used Stone local limit theorem \cite{S67} (using that any two points in $\gU_{\mathtt L}$ have abscissa differing by at least $\gep n$ with probability going to $1$ as $\gep \to 0$) to get that $\bP\big(  S\cap \gU_{\mathtt L} = \Delta\big) = n^{-\frac{|\Delta|}{2} +o(1)} e^{-\ent(\Delta)}$ uniformly for $\Delta \subset \gU_{\mathtt L}$.
Since there are finitely many terms in the sum, we get that analogously to \eqref{end-lowerbound}-\eqref{end-upperbound},
\[\sum_{k=1}^K \bU_k^{(\mathtt L)}  = e^{o(\log n)} \times  \exp\Big( \max_{\Delta \subset \gU_{\mathtt L}, |\Delta |\le K }\Big\{  \gb_n \Omega_{n,A_n}(\Delta) - \ent(\Delta) -  \frac{|\Delta|}{2} \log n  \Big\}\Big) . \]
At this stage we write
\[\begin{split}
&\max_{\Delta \subset \gU_{\mathtt L}, |\Delta |\le K }\Big\{  \gb_n \Omega_{n,A_n}(\Delta) - \ent(\Delta) -  \frac{|\Delta|}{2} \log n  \Big\}=
\max_{1\le k\le K} \tilde T_{n,h}^{\gb_{n,h}, (k,\mathtt L)},\\ 
&\text{where}\qquad
\tilde T_{n,h}^{\gb_{n,h}, (k,\mathtt L)}:=
\max_{\Delta \subset \gU_{\mathtt L}, |\Delta |=k }\Big\{  \gb_n \Omega_{n,A_n}(\Delta) - \ent(\Delta) -  \frac{k}{2} \log n  \Big\}
\end{split}
\]
To complete the proof of \eqref{cas3-bfinal} we only have to show that
\begin{equation}
\label{convtildeTkl}
\frac{n}{h_n^2}\log \Big( \sum_{k=1}^K \bU_k^{(\mathtt L)} \Big)  = o(1)+
\frac{n}{h_n^2}\max_{1\le k\le K} \tilde T_{n,h}^{\gb_{n,h}, (k,\mathtt L)}
\xrightarrow[]{(\dd)}  \max_{1\le k\le K}\tilde  \cT_{\gb,A}^{(k,\mathtt L)}.
\end{equation}
In analogy with \eqref{def:discreteELPPtilde} and Proposition \ref{prop:ConvVPtilde}, we have that for any fixed $k$, 
\[
\frac{n}{h_n^2} \tilde T_{n,h}^{\gb_{n,h}, (k,\mathtt L)}
\xrightarrow[]{(\dd)} \tilde  \cT_{\gb,A}^{(k,\mathtt L)} \, .
\]
As for the convergence of  \eqref{conv:largeweights}, since we have only a finite number of points,  
the proof is a consequence of (5.1) and (5.2) of \cite{cf:BT_ELPP} and the Skorokhod representation theorem---we use also that $\frac{n}{h_n^2} \log n \to \frac{1}{\gb}$.
Since the maximum is taken over a finite number of terms, this shows \eqref{convtildeTkl} and concludes the proof.

\subsection{Regime 4: convergence of the main term}
\label{sec:regime4}

First of all, we show briefly that $W_{\gb}$ is well defined, before we turn to the proof of Lemma \ref{lem:cas4-mainterm}.
One of the difficulties here is that  the reduction to trajectories operated in Section~\ref{sec:4-step1} (to trajectories with $\max_{i\le n} |S_i| \le A\sqrt{n\log n}$) is not adapted here, since the transversal fluctuations are of order  $h_n \ll \sqrt{n\log n}$. Therefore, we have to further reduce the set of trajectories in $\bU_1$.

\subsubsection{Well-posedness and properties of $W_{\gb}$}
We prove the following proposition.
\begin{proposition}
\label{prop:W}
Assume that $\ga\in(1/2,1)$. Then for every $\gb>0$, $W_\gb\in (0,\infty)$ almost surely.
\end{proposition}

\begin{proof} 
Recalling the definition \eqref{def:W} of $W_{\gb}$.
We fix a region $\mathcal{D}_\gep:= [\frac{1}{2},1]\times [-\gep,\gep]$, for $\gep>0$. 
In such a way we have that 
\begin{equation}\label{lbWb}
W_\gb\ge \sup_{(w,t.x)\in \cP ; (t,x) \in \mathcal{D}_\gep}\big\{\, w\, \big\}- \frac{\gep^2}{\gb}.
\end{equation}
We observe that
\[
\max_{(w,t,x)\in \cP; (t,x )\in\mathcal{D}_{\gep}}\big\{\, w\, \big\}\overset{(\dd)}{=}
(2 \gep)^{1/\ga } \mathrm{Exp}(1)^{-1/\alpha}.
\] 
Therefore, since $\frac{1}{\ga}<2$, the r.h.s. of \eqref{lbWb} is a.s. positive provided $\epsilon$ is sufficiently small.

For an upper bound, we simply observe that $W_{\gb} \le \cT_{\gb} <\infty$ a.s.
\end{proof}

\subsubsection{Proof of Lemma \ref{lem:cas4-mainterm}}

We denote $p(i,x) := \bP(S_i=x)$ for the random walk kernel.
For $A>0$ fixed and $\gd>0$, we split $\sqrt{n}\, \bU_1$ into three parts:
\begin{align}
\label{eq:str27-1}
\sqrt{n}\, \bU_1& := \sum_{(i,x)\in \gU_{\ell} } e^{\beta_n  \omega_{i,x}}\sqrt n p(i,x) \\
& = 
\bigg( \sumtwo{(i,x)\in \gU_{\ell}}{|x|> A h_n } 
+ \sumtwo{(i,x)\in \gU_{\ell}}{ i < \gd n, |x|\le  A h_n } 
+ \sumtwo{(i,x)\in \gU_{\ell}}{ i\ge \gd n, |x|\le A h_n }  \bigg) e^{\beta_n  \omega_{i,x}}\sqrt n p(i,x)\, . \notag
\end{align}
The main term is the last one, and we now give three lemmas to control the three terms.

\begin{lemma}
\label{lem:term1}
There exist  constants $c$ and $\nu>0$ such that for all $n$ sufficiently large, for any $A>1$
\begin{equation}
\bbP\Big(\sum_{(i,x)\in \gU_{\ell}, |x|> A h_n }  e^{\beta_n  \omega_{i,x}}\sqrt n p(i,x) >  A \Big(\frac{h_n^2}{n}\Big)^{3\alpha} \Big) \le  c A^{-\nu}\, .
\end{equation}
\end{lemma}

\begin{lemma}
\label{lem:term2}
There exist some $c,\nu>0$ such that, for any $A>1$ and $0<\gd<A^{-1}$, we get that for $n$ sufficiently large,
\begin{equation}
\bbP \bigg( \frac{n}{h_n^2} \log \Big( \sum_{(i,x)\in \gU_{\ell}, i < \gd n, |x|\le  A h_n }  e^{\beta_n   \omega_{i,x}}\sqrt n p(i,x)  \Big) \ge (\gd A)^{\frac{1}{4\ga}} \bigg) \le c (\gd A)^{1/2} \, .
\end{equation}
\end{lemma}

And finally, for last term, we have the convergence.
\begin{lemma}
\label{lem:term3}
We have that
\[
\frac{n}{h_n^2} \log \Big ( \sum_{(i,x)\in \gU_{\ell}, i\ge \gd n, |x|\le A h_n } e^{\beta_n  \omega_{i,x}}\sqrt n p(i,x) \Big ) \stackrel{({\rm d})}{\longrightarrow} W_1(\gd,A)  \, ,\]
with $W_1(\gd,A):= \max\limits_{(w,t,x)\in \cP ,   t>\gd, |x|\le A}\big\{w-\frac{x^2}{2 t}\big\} \, .$
\end{lemma}
Now, let us observe that taking the limit $\gd\downarrow 0$, and $A\uparrow \infty$, we readily obtain that $W_1( \gd, A)\to W_1$ (by monotonicity).
Therefore, combining Lemmas~\ref{lem:term1}-\ref{lem:term2}-\ref{lem:term3}, we conclude the proof of Lemma~\ref{lem:cas4-mainterm}.\qed

\begin{proof}[Proof of Lemma \ref{lem:term1}]
Let us consider the event  
\begin{equation}
\cG(n,A):= \Big \{ \, \beta_n  \omega_{i,x}\le \frac{x^2}{8 i} \, \text{for any}\, |x|> Ah_n,\, 1\le i \le n \Big \}.
\end{equation}
Using this event to split the probability (and Markov's inequality), we have that, recalling the definition \eqref{def:gUell} of $\gU_\ell$
\begin{align}
\label{splitG}
&\bbP \Big(\sum_{(i,x)\in \gU_{\ell},  |x| >A h_n } e^{\beta_n  \omega_{i,x}}\sqrt n p(i,x)   > A \Big(\frac{h_n^2}{n}\Big)^{3\alpha} \Big) 
\\
& \le 
\frac{1}{A}\Big(\frac{h_n^2}{n}\Big)^{-3\alpha}\bbE\Big[\sum_{i=1}^n\sum_{|x|> Ah_n} e^{x^2/8i}\sqrt{n}p(i,x)  \ind_{\{\gb_n \go_{i,x} \ge 1\}}\Big] 
+ \bbP\Big(\cG(n,A)^c\Big) \, .\notag
\end{align}
Using again that  $ \bbP(\go \ge 1/\gb_n) \le (h_n^2/n)^{2\ga} (nh_n)^{-1}$ and that $p(i,x) \le e^{-x^2/4i}$ uniformly in the range considered (provided that $n$ is large enough), we get that the first term is bounded by
\begin{align*}
\frac{1}{A}\Big(\frac{h_n^2}{n}\Big)^{-\alpha} \frac{\sqrt{n}}{n h_n}  & \sum_{i=1}^n\sum_{|x|> Ah_n} e^{-x^2/8i} \le \Big(\frac{h_n^2}{n}\Big)^{-\alpha} \, .
\end{align*}
In the last inequality, we used that the sum over $x$ is bounded by a constant independent of $i$, and also that $\sqrt{n}/h_n \to 0$.
The first term in \eqref{splitG} therefore goes to $0$ as $n\to\infty$, and we are left to control $\bbP(\cG(n,A)^c)$.
A union bound gives
\begin{align}
\bbP \big( \cG(n,A)^c \big)& \le \sum_{i=1}^n \sum_{x =A h_n}^{+\infty} \bbP \Big(  \gb_n \go_{i,x} \ge \frac{x^2}{8i} \Big)
\le n  \sum_{k = 0}^{+\infty} \sum_{x = 2^{k} A h_n}^{2^{k+1} A h_n}  \bbP\Big(  \gb_n  \go  \ge  2^{2k} A^2  \frac{h_n^2}{8n} \Big) \notag\\
& \le 2 A n h_n \sum_{k=0}^{\infty} 2^k  \bbP\Big(   \go   \ge  \frac{1}{10} 2^{2k} A^2  m(nh_n) \Big) \, ,
\end{align}
where we used the definition \eqref{def:hn} of $h_n$ for the last inequality, with $n$ large enough.
Then, using the definition of $m(nh_n)$ and Potter's bound, we obtain that for any $\eta>0$ (chosen such that $1-2\ga+2\eta<0$) there is a constant $c>0$ such that for $n$ large enough
\[\bbP \big( \cG(n,A)^c \big) \le c A n h_n \sum_{k\ge 1} 2^{k} (2^{2k} A^{2})^{- \ga +\eta} \frac{1}{n h_n}\le c' A^{1-2\ga +2\eta }\, ,\]
where the sum over $k$ is finite because $1-2\ga+2\eta<0$.
This concludes the proof of Lemma~\ref{lem:term1}.
\end{proof}

\begin{proof}[Proof of Lemma \ref{lem:term2}]
Decomposing over the event 
\[ \cM_{n}(\gd,A) = \Big\{ \max_{ i < \gd n , |x| \le A h_n } \gb_n \go_{i,x}   \le  \frac12 (\gd A)^{\frac{1}{4\ga}}  \frac{h_n^2}{n} \Big\} \, ,\]
and using Markov's inequality, we get that (similarly to \eqref{splitG})
\begin{align}
&\bbP \bigg( \sum_{ (i,x)\in \gU_{\ell}, i < \gd n, |x|\le  A h_n} e^{\beta_n  \omega_{i,x}}\sqrt n p(i,x) \ge \exp\Big( (\gd A)^{\frac{1}{4\ga}} \frac{h_n^2}{n} \Big)\bigg)\\
&\le e^{- \frac12 (\gd A)^{\frac{1}{4\ga}} \frac{h_n^2}{n} } \bbE\Big[ \sum_{i=1}^{\gd n}\sum_{|x|\le Ah_n} \sqrt n p(i,x)\ind_{\{ \beta_n  \omega_{i,x} \ge 1 \}}  \Big]
+ \bbP\big( \cM_{n}(\gd,A)^{c}  \big)\, . \notag
\end{align}
We use again that $\bbP( \omega\ge 1/ \gb_n ) \le (h_n^2/n)^{2\ga} (nh_n)^{-1}$, and the fact that $\sum_x p(i,x) =1$ for any $i\in \bbN$, to get that the first term is bounded by
\[  e^{- \frac12 (\gd A)^{\frac{1}{4\ga}} \frac{ h_n^2}{n} } \Big( \frac{h_n^2}{n}\Big)^{2\ga} \frac{n\sqrt{n}}{nh_n} \to 0 \quad \text{as } n\to\infty \, .\]
For the remaining term, using that $\gb_n^{-1} h_n^2/n \sim m(n h_n)$, we have by a union bound that for $n$ large enough
\begin{align*}
\bbP \big( \cM_{n}(\gd,A)^{c}  \big)& \le \gd A n h_n \bbP\Big(   \go > \frac14 (\gd A)^{\frac{1}{4\ga}} m(n h_n)\Big) \\
&\le c \gd A n h_n \times \big( (\gd A)^{\frac{1}{4\ga}} \big)^{-2\ga} \frac{1}{nh_n} \, ,
\end{align*}
where we used Potter's bound  (with $(\gd A)^{\frac{1}{4\ga}}$ small) and the definition of $m(nh_n)$ for the last inequality (for $n$ large). This concludes the proof of Lemma~\ref{lem:term2}.
\end{proof}

\begin{proof}[Proof of Lemma \ref{lem:term3}]
The Stone local limit theorem \cite{S67} (see \eqref{LLT}) gives that, for fixed $A>0, \gd>0$, there exists $c>0$ such that uniformly for $\gd n\le i\le n$, $|x|\le A h_n$,
\begin{equation}
\label{LocalLT}
\frac1c\, e^{- x^2/2i}  \le \sqrt{i}\,  p(i,x)  
\le c\, e^{- x^2/2i} \, .
\end{equation}

Since $\sqrt{n/i} \ge 1$ for all $i\le n$, we get the lower bound
\begin{align}
\sum_{i=\gd n }^{n} \sum_{|x| \le A h_n} e^{\gb_n  \go_{i,x}} \sqrt{n} p(i,x) \ind_{\{\beta_n  \omega_{i,x} \ge 1  \}} \ge c \exp\Big(  \gb_n W_{n}( \gd,A)\Big)\, , 
\label{lowbound}
\end{align}
where $W_{n}(  \gd,A)$ is a discrete analogue of $W_1( \gd,A)$, that is
\begin{equation}
\label{def:Wn}
W_{n}(\gd,A) := \max\limits_{ \substack {|x|\le A h_n , \, i =\gd n ,\ldots, n  \\
		\gb_n   \go_{i,x} \ge 1 } } \Big\{   \go_{i,x}-\frac{x^2}{2 \gb_n i} \Big\} \,  .
\end{equation}
On the other hand, we get that $\sqrt{n/i} \le \gd^{-1/2}$ for $i\ge \gd n$, so that from \eqref{LocalLT} we get
\begin{equation}\label{eq024}
\sum_{i=\gd n }^{n} \sum_{|x| \le A h_n} e^{\gb_n   \go_{i,x}} \sqrt{n} p(i,x)  \ind_{\{\beta_n  \omega_{i,x} \ge 1 \}}\le 
\frac{c}{\sqrt{\gd}}\, e^{\gb_n W_{n}( \gd,A)} \sum_{i=1}^n\sum_{|x|\le Ah_n}  \ind_{\{\beta_n \omega_{i,x} \ge 1 \}} \, .
\end{equation}
Now, we have that $\bbP(  \go  > 1/ \gb_n) \le  (h_n^2/n)^{2\ga} (nh_n)^{-1}$ as already noticed, so that 
\begin{equation}
\label{eq024after}
\bbE\Big[\sum_{i=1}^n\sum_{|x|\le Ah_n} \ind_{\{\beta_n  \omega_{i,x} \ge 1 \}} \Big]\le A  \left(\frac{h_n^2}{n}\right)^{2\ga} .
\end{equation}
Overall, combining \eqref{lowbound} with \eqref{eq024}-\eqref{eq024after}, we get that with probability going to $1$ as $n\to\infty$,
\begin{equation*}
\bigg| \log \Big(  \sum_{(i,x)\in \gU_{\ell}, i\ge \gd n, |x|\le A h_n }   e^{\beta_n  \omega_{i,x}}\sqrt n p(i,x)  \Big) - \gb_n W_{n}(\gd,A) 
\bigg| \le (2\ga +1) \log \frac{h_n^2}{n} \, .
\end{equation*}

To conclude the proof of Lemma \ref{lem:term2}, it therefore remains to show that
\begin{equation}\label{W1deltAconv}
\frac{n}{h_n^2}  \times \beta_n W_{n}( \gd, A)  \xrightarrow[n\to\infty]{(\dd)} W_1(  \gd, A),
\end{equation}
where $W_1( \gd,A)$ is defined in Lemma \ref{lem:term2}. 

We fix $\gep>0$ and we consider $\widetilde W_{n}(\gep , \gd,A) $  the truncated version of 
$W_{n}( \gd,A)$ in which we replace the condition $\{\gb_n  \go_{i,x} \ge 1\}$ by $\{\gb_n  \go_{i,x} >\gep \frac{h_n^2}{n}\}$, that is
\begin{equation}
\label{def:Wntilde}
\widetilde W_{n}(\gep, \gd,A) := \max\limits_{ \substack {|x|\le A h_n , \, i =\gd n ,\ldots, n  \\
		\gb_n   \go_{i,x} >\gep \frac{h_n^2}{n} } } \Big\{   \go_{i,x}-\frac{x^2}{2 \gb_n i} \Big\} \,  .
\end{equation}
In such a way, and since $\gep h_n^2/n \ge 1$ for large $n$, we have
\[
\frac{n}{h_n^2}  \beta_n \widetilde W_{n}(\epsilon,\gd, A)\le \frac{n}{h_n^2}  \beta_n W_{n}(\gd, A) \le \frac{n}{h_n^2}  \beta_n\widetilde W_{n}(\epsilon,\gd, A) +\epsilon.
\]
To prove \eqref{W1deltAconv} we need to show that
\begin{equation}
\label{convWtilde}
\frac{n}{h_n^2}  \times \beta_n \tilde W_{n}(\gep, \gd, A)  \xrightarrow[n\to\infty]{(\dd)} \tilde W_1( \gep, \gd, A) :=\max\limits_{\substack{(w,t,x)\in \cP \\   t>\gd, |x|\le A, w>\gep}}\Big\{w-\frac{x^2}{2 t} \Big\}  ,
\end{equation}
and then let $\gep\downarrow 0$ -- notice that we have  $\tilde W_1(\gep, \gd,A)\le W_1(\gd,A)\le \tilde W_1(\gep, \gd,A) +\gep$ so that $\tilde W_1( \gep, \gd, A) \to W_1(\gd,A)$ as $\gep\downarrow 0$.

We observe that a.s. there are only finitely many $\go_{i,x}$ in $\llbracket 1,n \rrbracket \times \llbracket -A h_n, A h_n\rrbracket$ that are larger than $\gep m(nh_n) \sim  \gb_n^{-1} \gep  h_n^2/n$. This is a consequence of Markov's inequality and Borel-Cantelli Lemma. Indeed, for any $K\in \mathbb N$ we have
\[\begin{split}
\bbP\Big( \, \Big|\big\{(i,x)\in \llbracket 1,n \rrbracket \times  \llbracket -A h_n, &A h_n\rrbracket \colon \omega_{i,x} \ge \gep m(nh_n)
\big\}\Big| > 2^K\Big)\\
&\le 2^{-K} (2Anh_n) \bbP\Big(\omega\ge \gep m(nh_n)\Big)\le C_\gep 2^{-K}\, .
\end{split}
\]
Therefore, the convergence \eqref{convWtilde}
is a straightforward consequence of the  Skorokhod representational theorem.
\end{proof}


\section{Case $\ga\in(0,1/2)$}
\label{sec:alpha12}

In the first part of this section we prove \eqref{eq:alpha<12_one}. In the second part, we prove the convergence \eqref{eq:alpha<12_two}.

\subsection{Transversal fluctuations: proof of \eqref{eq:alpha<12_one}}
\label{sec:fluctualpha12}

\subsubsection{Paths cannot be at an intermediate scale}

We start by showing that there exists $c_0,c,\nu>0$ such that for any sequences $C_n>1$ and $\delta_n\in (0,1)$ (which may go to $\infty$, resp.\ $0$, as $n\to\infty$) and for any $n\ge 1$
\begin{equation}
\label{goalalphale12}
\mathbb P \Big(\bP^\omega_{n,\beta_n}\big( \max\limits_{i\leq n} \big| S_i\big| \in [C_n \sqrt{n}, \delta_n n)\big) \leq e^{-c_0 C_n^2}
+e^{-c_0 n^{1/2}}\Big) \geq 1-c\delta_n^{\nu} + n^{- \frac{1-2\ga}{4} +\gep }.
\end{equation}
To prove it, we use a decomposition into blocks, as we did 
in Section \ref{sec:fluctu}. Here, we have to  partition the interval $[C_n \sqrt{n}, \delta_n n)$ into 
$[C_n \sqrt{n}, n^{3/4})\cup [n^{3/4}, \delta_n n)$ (one of these intervals might be empty), obtaining
\begin{align}
\nonumber
&\bP^\omega_{n,\beta_n}\Big( \max\limits_{i\leq n} \big| S_i\big| \in [C_n \sqrt{n}, \delta n)\Big) \\  
& \qquad = \bP^\omega_{n,\beta_n}\Big( \max\limits_{i\leq n} \big| S_i\big| \in [C_n \sqrt{n}, n^{3/4})\Big)+
\bP^\omega_{n,\beta_n}\Big( \max\limits_{i\leq n} \big| S_i\big| \in (n^{3/4}, \delta_n n)\Big).
\label{eqalphale12_1}
\end{align}

For the first term, we partition the interval $[C_n\sqrt{n}, n^{3/4})$ into smaller blocks 
$D_{k,n}:=[2^{k}\sqrt{n}, 2^{k+1}\sqrt{n})$, with
$k= \log_2 C_n, \dots,\log_2 n^{1/4}-1$. 
Let us define
\begin{equation}
\Sigma(n,h) = \sum_{i=1}^n \sum_{x\in \llbracket -h ,h\rrbracket} \go_{i,x}
\end{equation}
the sum of all weights in $\llbracket 1,n \rrbracket \times \llbracket -h ,h\rrbracket$.
Then, we write similarly to \eqref{alph12EQ1} (we also use that $\bZ_{n,\gb_n}^{\go}\ge 1$, which is harmless here since no recentering term is needed)
\begin{align*}
\bP^\omega_{n,\beta_n}\Big( \max\limits_{i\leq n} \big| S_i\big| \in & [C_n \sqrt{n}, n^{3/4})\Big)   \le \sum_{k=\log_2 C_n}^{\log_2 n^{1/4} } \bZ^{\go}_{n,\gb_n}\big( \max_{i\le n} |S_i| \in D_{k,n} \big) \\
& \le \sum_{k=\log_2 C_n}^{\log_2 n^{1/4} -1} e^{\beta_n \Sigma(n,2^{k+1}\sqrt n) }\bP\big( \max_{i\le n} |S_i| \in D_{k,n} \big)\\
& \le  \sum_{k=\log_2 C_n}^{ \log_2 n^{1/4}} \exp\Big( \beta_n \Sigma(n,2^{k+1}\sqrt n) - c  2^{2k}  \Big)
\end{align*}
where for the last inequality 
we used a standard estimate for the deviation probability of a random walk $\bP \big( \max_{i\le n} |S_i| \ge 2^k \sqrt{n} \big) \le e^{ - c 2^{2k}}$, see for example \cite[Prop.~2.1.2-(b)]{LL10}.
Therefore, on the event
\begin{equation}
\label{evA34}
\Big\{ \forall\, k= \log_2 C_n, \dots,\log_2 n^{1/4},\, \beta_n \Sigma(n,2^{k+1}\sqrt n) \leq \frac{c }{2}2^{2k}
\, \Big\}
\end{equation}
we have that 
\begin{equation}
\bP^\omega_{n,\beta}\Big( \max\limits_{i\leq n} \big| S_i\big| \in [C_n\sqrt{n}, n^{3/4})\Big)
\leq \sum_{k=\log_2 C_n}^{\log_2 n^{1/4}} e^{-\frac{c }{2}2^{2k}}\leq c'  e^{-\frac{c }{2} C_n^2}.
\end{equation}

For the second term in \eqref{eqalphale12_1}, we partition the interval
$(n^{3/4}, \delta_n n)$ into blocks $E_{n,k}:=[2^{-k-1}n, 2^{-k}n)$, $k= \log_2(1/\delta_n), \dots,\log_2 n^{1/4}-1$. 
Exactly as above we use the large deviation estimate $\bP \big( \max_{i\le n} |S_i| \ge 2^{-k+1} n \big) \le e^{ - c 2^{-2k} n}$ (see e.g~\cite[Prop.~2.1.2-(b)]{LL10}), and we obtain that on the event
\begin{equation}
\label{ev34delta}
\Big\{ \forall\, k=\log_2(1/\delta_n), \dots,\log_2 n^{1/4},\, \beta_n \Sigma(n,2^{-k n}) \leq \frac{c}{2}2^{-2k}n \,
\Big\}
\end{equation}
we have
\begin{equation}
\bP^\omega_{n,\beta}\Big( \max\limits_{i\leq n} \big| S_i\big| \in (n^{3/4}, \delta_n n)\Big)
\leq \sum_{k=\log_2(1/\delta_n)}^{\log_2 n^{1/4}} e^{ - \frac{c}{2} 2^{-2k} n } \leq c' e^{-\frac{c}{2} n^{1/2}}.
\end{equation}

It now only remains to show that the complementary events of \eqref{evA34} and \eqref{ev34delta} have small probability.
We start with \eqref{ev34delta}. Using that $\gb_n \le 2\gb n /m(n^2)$ for $n$ large, we get by a union bound that
\begin{align}\label{ev34delta2}
\bbP\Big( \exists\, k  \ge \log_2 1/\gd_n &\, , \,  \beta_n  \Sigma(n,2^{-k }n)    > \frac{c}{2} 2^{-2k} n \Big) \\
& \le \sum_{k\ge \log_2 1/\gd_n} \bbP\Big( \Sigma(n,2^{-k }n)  >  c_{\gb} 2^{-2k} m(n^2) \Big) \, .
\notag
\end{align}
Then,  by Potter's bound we have that $m(2^{-k+1}n^2) \le 2^{-2k} m(n^2)$ since $\ga<1/2$ (recall $m(\cdot)$ \eqref{def:m} is regularly varying with exponent $1/\ga$).
As a consequence, the last probability in \eqref{ev34delta2} is in the so-called  one-jump large deviation domain (see \cite[Thm.~1.1]{Nag79}, we are using $\ga<1$ here), that is
\[\bbP\Big( \Sigma(n,2^{-k }n)  >  c_{\gb} 2^{-2k} m(n^2) \Big) \sim 2^{-k+1} n^2 \bbP\big( \go > c_{\gb} 2^{-2k} m(n^2)\big)\, . \]
Therefore, using again Potter's bound, we get that for arbitrary $\eta$ there is some constant $c$ such that
\[\bbP\Big( \Sigma(n,2^{-k }n)  >  c_{\gb} 2^{-2k} m(n^2) \Big) \le  c   (2^{2k})^{\ga+\eta} n^{-2}\]
where we also used that $\bbP(\go>m(n^2)) =n^{-2}$.
Therefore, taking  $\eta$ small enough so that $2\ga -1 +2\eta<0$, we obtain that \eqref{ev34delta2} is bounded by a constant times
\[ \sum_{k\ge \log_2 1/\gd_n}   2^{k (2\ga -1 +2\eta)} \le c \gd_n^{1-2\ga+2\eta}\, .  \]

Similarly, for \eqref{evA34}, we have by a union bound that
\begin{align}
\label{evA342}
\bbP \Big(  \exists\, k  \in   \{ \log_2 C_n, &\dots,\log_2 n^{1/4}\}, \, \beta_n   \Sigma(n,2^{k+1} \sqrt{n} ) > \frac{c}{2}2^{2k}  \Big) \notag\\
&\le \sum_{k=\log_2 C_n}^{\log_2 n^{1/4}} \bbP\Big( \Sigma(n,2^{k+1} \sqrt{n} )  > c_{\gb} 2^{2k} n^{-1} m(n^2) \Big)\, .
\end{align}
Then again, we notice that $m(2^{k+2}n^{3/2}) \le 2^{2k} n^{-1} m(n^2)$  (using Potter's bound, as $\ga< 1/2$). Hence,   the last probability in \eqref{evA342} is in the one-jump large deviation domain (see \cite[Thm.~1.1]{Nag79}), that is
\[\bbP\Big( \Sigma(n,2^{k+1} \sqrt{n} )  > c 2^{2k} n^{-1} m(n^2) \Big) \le c 2^{k}n^{3/2} \bbP\big(\go > c_{\gb}   2^{2k} n^{-1} m(n^2) \big)  
\]
Then, we also get that for any $\eta>0$ we have that there is a constant $c>0$ such that
\[ \bbP\big(\go > c_{\gb}   2^{2k} n^{-1} m(n^2) \big) \le c (2^{2k} n^{-1})^{-\ga-\eta} \, ,\]
so that provided that $1-2\ga -2\eta >0$, \eqref{evA342} is bounded by
a constant times
\[ \sum_{k=\log_2 C_n}^{\log_2 n^{1/4}}  2^{k(1-2\alpha-2\eta)}n^{\ga-\frac12+\eta} \le c n^{- \frac14 (1-2\ga -2\eta)}\, .\]

\subsubsection{Paths cannot be at scale $n$ conditionnaly on $\hat \cT_{\gb}=0$}

We have shown in \eqref{goalalphale12} that paths cannot be on an intermediate scale: it remains to prove that on the event $\hat \cT_{\gb} =0$, paths cannot be at scale $n$.
For this purpose we use \cite[Theorem 2.1]{AL11} and \cite[Theorem 1.8]{T14}, which ensure that for any $\delta$ and $\epsilon>0$ there exists $\nu>0$ such that
\begin{equation}
\label{eq:ALThm2.1}
\mathbb P\Big(\bP_{n,\gb_n}^{\go} \big( \max_{i\leq n} |S_i| \in (\delta n, n] \big) \leq e^{-n\nu}\  \Big| \ \hat \cT_{\gb} =0\Big)\ge 1-\epsilon.
\end{equation}

Therefore, we get that for any $\gep>0$ and $\gd>0$, combining \eqref{goalalphale12} with \eqref{eq:ALThm2.1}, for any sequence $C_n >1$, provided that $n$ is large enough we have
\[
\mathbb P \Big(\bP_{n,\gb_n}^{\go} \big( \max_{i\leq n} |S_i| \ge C_n \sqrt{n} \big) \ge e^{-c_0 C_n^2} + e^{-c_0 n^{1/2}} + e^{-n\nu} \  \Big| \ \hat \cT_{\gb} =0 \Big)  \le c \gd^{\nu}  + 2\gep \, ,
\]
which concludes the proof of \eqref{eq:alpha<12_one}.

\subsection{Convergence in distribution conditionally on $\hat \cT_{\gb} =0$, proof of \eqref{eq:alpha<12_two}} 
\label{sec:convalpha12}

In the following, we  consider the case where $\gb_n n^{-1} m(n^2) \to \beta$ with $\beta<\infty$. In the case $\gb=+\infty$, we would indeed have that $\hat \cT_{\gb} >0$.
The proof follows the same idea as that of \cite[Thm.~1.4]{cf:DZ} (and similar steps as above), but with many adaptations (and simplifications) in our case.
We focus on the case $\gb>0$,  in which $\frac{\sqrt n}{ \gb_n m(n^{3/2})} $ goes to infinity as a regularly varying function with exponent $\frac{2}{\ga}-\frac12 -\frac{3}{2\ga} = \frac{1-\ga}{2\ga} >0$ (If $\gb=0$, it goes to infinity faster).

\subsubsection{Step 1. Reduction of the set of trajectories.}
Equation \eqref{eq:alpha<12_one} (with $C_n=A\sqrt{\log n}$)
gives that,  with $\bbP$ probability larger than $1-\gep$ (conditionally on $\hat \cT_{\gb}=0$), we have $\bP_{n,\beta_n}^\go  \big(\max_{i\le n}|S_i|\leq A\sqrt{n \log n }\big) \ge 1- e^{ - c_0 A \log n }$ provided that $n$ is large enough. We therefore get  
\begin{equation}
\bbP\Big( \big| \log \bZ_{n,\beta_n}^{\omega}-\log \bZ_{n,\beta_n}^{\omega}\big(\cA_n \big)\big|  \le  n^{-c_0 A}  \, \Big |\, \hat \cT_{\gb} =0\Big) \ge 1-\gep \,,
\end{equation}
where $\cA_n$ is defined in \eqref{def:An}.
Note that, provided  $A$ has been fixed large enough, we have that $\frac{\sqrt n}{\beta_n m(n^{3/2})}  n^{-c_0 A} \to 0$ as $n\to\infty$:
we  conclude that, for any $\gep>0$
\begin{equation}
\bbP\Bigg(\frac{\sqrt n}{\beta_n m(n^{3/2})} \big| \log \bZ_{n,\beta_n}^{\omega}-\log \bZ_{n,\beta_n}^{\omega}\big(\cA_n\big)\big|  > \gep  \, \Big |\,  \hat \cT_{\gb} =0\Bigg) \le \gep\, ,
\end{equation}
provided that $n$ is large enough. We will therefore focus on $\log \bZ_{n,\beta_n}^{\omega}\big(\cA_n\big)$.

As in Section~\ref{sec:3b4}, we use the notation $A_n = A\sqrt{n\log n}=C_n \sqrt n$ and $\Lambda_{n,A_n} = \llbracket 1,n \rrbracket \times \llbracket -A_n,A_n \rrbracket$.

\subsubsection{Step 2. Truncation of the weights.}
We let $k_n:=m(n^{3/2} \log n)$ be a sequence of truncation levels, and $\tilde \go_x := \go_{x} \ind_{\{\go_x \le k_n\}}$ be the truncated environment. Then, we have that 
\begin{align*}
\bbP\Big(  \bZ_{n,\beta_n}^\omega (\cA_n)  \neq  \bZ_{n,\beta_n}^{\tilde \omega} (\cA_n)  \Big)& = \bbP\big( \max_{(i,x)\in \Lambda_{n,A_n}} \go_{i,x} > m(n^{3/2} \log n) \big) \\
& \le  \frac{2A}{\sqrt{\log n}} \stackrel{n\to\infty}{\to} 0\, ,
\end{align*} 
where we used a union bound for the last inequality, together  with the definition of $m(\cdot)$ \eqref{def:m}. 
Henceforth we can safely replace $\bZ_{n,\beta_n}^\omega (\cA_n)$ with the truncated partition function $\bZ_{n,\beta_n}^{\tilde \omega} (\cA_n)$.

\subsubsection{Step 3. Expansion of the partition function.}
We write again $p(i,x)=\bP(S_i=x)$ for the random walk kernel, and let $\lambda_n(t) = \log \bbE[e^{t \tilde\go_x}]$. Then, expanding 
\[\exp\Big( \sum_{i=1}^n  \big( \gb_n \go_{i,S_i} -\lambda_n(\gb_n) \big) \Big) = \prod_{(i,x)\in \Lambda_{n,A_n}} \big(1+e^{\gb_n \tilde \go_{i,x} -\lambda_n(\gb_n)}-1 \big)^{\ind_{\{S_{i}=x\}}},\]
we obtain
\begin{align}
\label{CEZomega}
e^{-n \lambda_n(\gb_n)} \bZ_{n,\beta_n}^{\tilde \omega}(\cA_n)= 1 +  \!\!\!\!
\sum_{(i,x)\in  \Lambda_{n,A_n}} \!\!\!\! \big(e^{\beta_n \tilde\omega_{i,x} - \lambda_n(\beta_n)}-1\big)p(i,x)+ \bR_n,
\end{align}
with 
\[
\bR_n:=\sum_{k=2}^\infty \sum_{\substack{{1\leq i_1<\dots<i_k\leq n} \\ { |x_i|\leq A_n,\, i=1,\dots, k}}}
\prod_{j=1}^k\big(e^{\beta_n \tilde\omega_{j,x_j} - \lambda_n(\beta_n)}-1\big)p_n(i_j-i_{j-1},x_j-x_{j-1})\, .
\]

\begin{lemma}
\label{lem:R}
We have that for $n$ large 
\[\bbP\Bigg(  \frac{\sqrt{n}}{\gb_n m(n^{3/2}) } \bR_n \ge n^{-1/4} \Bigg) \le   \frac{(\log n)^{4/\ga}}{\sqrt{n}} \to 0\, .\]
In particular, $\bR_n \to 0$ in probability.
\end{lemma}
\begin{proof}
Note that $\bbE[\bR_n]=0$, so it will be enough to  control the second moment of $\bR_n$.
Since the $\tilde\go_{i,x}$ are independent and $\bbE[e^{\gb_n \tilde\go_{i,x} -\lambda_n(\gb_n)} - 1] =0$, 
\begin{align*}
\mathbb E[\bR_n^2] & = \sum_{k = 2}^{\infty} \sum_{\substack{{1\leq i_1<\dots<i_k\leq n} \\ { |x_i|\leq A_n,\, i=1,\dots, k}}} \big(e^{\lambda_n(2\beta_n)-\lambda_n(\beta_n)}-1\big)^k\prod_{j=1}^k p_n(i_j-i_{j-1},x_j-x_{j-1})^2  \\
& \le \sum_{k =2}^{\infty} \big( e^{\lambda_n(2\gb_n)} -1\big)^k \Big(\sum_{i=1}^n \sum_{x\in\mathbb Z}  p(i,x)^2 \Big)^k .\notag
\end{align*}
First, we have that 
\[\sum_{i=1}^n \sum_{x\in\mathbb Z}  p(i,x)^2 =\bE^{\otimes 2}\Big[ \sum_{i=1}^n \ind_{\{S_n = S'_n\}} \Big] \le c\sqrt{n} \, ,  \]
where $S$ and $S'$ are two independent simple random walks.
Then, since $\gb_n \tilde\go \le \gb_n k_n \to 0$,  we can write   $e^{2\gb_n \tilde\go} \le 1+ 3\gb_n \tilde\go$ for $n$ large, so that
\begin{align}
e^{\lambda_n( 2 \gb_n) } -1 &\le 3 \gb_n \bbE[\tilde \go] = 3\gb_n \int_0^{k_n} \bbP(\go>u) {\rm d}u  \notag\\
&\le c \gb_n L(k_n) k_n^{1-\ga} \le  \frac{c \gb_n k_n }{n^{3/2} \log n}\, .
\label{elambdan}
\end{align}
To estimate the integral we used the tail behavior of $\bbP(\go>u)$ \eqref{eq:DisTail} (see \cite[Theorem 1.5.8]{BGT89}), while 
for the last inequality, we used that $k_n = m(n^{3/2} \log n)$ and the definition \eqref{def:m} of $m(\cdot)$, so that $L(k_n) k_n^{-\ga} \sim n^{-3/2} (\log n)^{-1}$.
We therefore get that for $n$ large enough
\[
\bbE[\bR_n^2] \le \sum_{k\ge 2} \Big( \frac{ \gb_n k_n }{n} \Big)^k 
\le 2  \Big( \frac{\gb_n k_n}{ n } \Big)^2\, .
\]
To conclude, by Potter's bounds we get that $k_n \le m(n^{3/2}) (\log n)^{2/\ga}$  for $n$ large, so that 
\begin{equation}
\bbE[\bR_n^2] \le \Big( \frac{\gb_n m(n^{3/2})}{\sqrt{n}} \Big)^2 \times \frac{(\log n)^{\frac{4}{\ga} }}{n}\, ,
\end{equation}
and the conclusion of the lemma follows by using Markov's inequality.
\end{proof}

Going back to \eqref{CEZomega}, we get that
\begin{align*}
&\bZ_{n,\beta_n}^{\tilde \omega} (\cA_n) \\
& = e^{(n-1) \lambda_n(\gb_n)} \Big(  e^{\lambda_n(\gb_n)} + 
\sum_{(i,x)\in  \Lambda_{n,A_n} } \big(e^{\beta_n \tilde\omega_{i,x}} -e^{\lambda_n(\gb_n)}\big)p(i,x)+ e^{\lambda_n(\gb_n)} \bR_n \Big) \notag\\
&=e^{(n-1) \lambda_n(\gb_n)} \Big(  1 + 
\mathbf{V}_n+ \mathbf{W}_n+ e^{\lambda_n(\gb_n)} \bR_n \Big)\, ,
\notag
\end{align*}
with
\[
\mathbf{V}_n:= \!\! \sum_{(i,x)\in \Lambda_{n,A_n}}\!\!  \big(e^{\beta_n \tilde\omega_{i,x}} -1\big)p(i,x) \ \text{ and }\ \mathbf{W}_n := (e^{\lambda_n(\gb_n)}-1)\big( 1 - \!\!\!\!  \sum_{(i,x)\in \Lambda_{n,A_n}} \!\!  p(i,x) \big)\, .
\]
We show below that $\lim_{n\to\infty} \mathbf{W}_n = 0$ and that $ \mathbf{V}_n$ converges in probability to $0$, so that using also Lemma \ref{lem:R}, we get
\begin{align}
\label{almostthere}
&\frac{\sqrt{n}}{\gb_n m(n^{3/2})}  \log \bZ_{n,\beta_n}^{\tilde \omega} (\cA_n)  \\
&=  \frac{\sqrt{n}}{\gb_n m(n^{3/2})} \mathbf{V}_n   + \frac{\sqrt{n}}{\gb_n m(n^{3/2})} \Big( (n-1) \lambda_n(\gb_n) + \mathbf{W_n}  \Big) + o(1)\, . 
\notag
\end{align}

Before we prove the convergence of the first term (see Lemma~\ref{lem:convergence}), we show that the second term goes to $0$---note that this  implies that $\mathbf{W}_n \to 0$ since $\gb_n n^{-1/2} m(n^{3/2}) \to 0$.
We write that
\begin{align}
\label{rewriteW}
\big| (n-1) \lambda_n(\gb_n) + \mathbf{W_n} \big| \le  (n-1) \big| e^{\lambda_n(\gb_n)} -& 1- \lambda_n(\gb_n) \big|\\
& + \Big| n- \!\!\!\! \sum_{(i,x) \in \Lambda_{n,A_n}} \!\!  p(i,x)  \Big| \, .\notag
\end{align}
For the second term, using standard large deviation for the simple random walk (e.g.\ \cite[Prop.~2.1.2-(b)]{LL10}),  there is a constant $c>0$ such that
\begin{equation}
\label{difference}
n-\sum_{(i,x)\in \Lambda_{n,A_n}}  p(i,x) =\sum_{i=1}^n \bP(S_i > A \sqrt{n\log n})  \le n e^{- c A^2 \log n}\, .
\end{equation}
For the first term, since we have $\lambda_n(\gb_n) \to 0$, we get that for $n$ large enough
\begin{equation}
\label{expolambda}
\big| e^{\lambda_n(\gb_n)} - 1- \lambda_n(\gb_n) \big| \le \lambda_n(\gb_n)^2 \le \Big( \frac{\gb_n m(n^{3/2})}{n^{3/2}}  (\log n)^{2/\ga} \Big)^2 \, , 
\end{equation}
where for the second inequality we used \eqref{elambdan} (note that $\lambda_n(\gb_n) \le e^{\lambda_n(\gb_n)}-1$), together with the fact that $k_n \le m(n^{3/2}) (\log n)^{2/\ga}$.

Hence plugging \eqref{difference} and \eqref{expolambda} into \eqref{rewriteW}, we get that provided that $A$ is large enough,
\[ \frac{\sqrt{n}}{\gb_n m(n^{3/2})} \Big| (n-1) \lambda_n(\gb_n) + \mathbf{W_n} \Big| \le  \frac{\gb_n m(n^{3/2})}{ n^{3/2} } (\log n)^{4/\ga} + o(1)  \xrightarrow{n\to\infty} 0 \, . \]
so that the second term in \eqref{almostthere} goes to $0$ as $n\to\infty$, proving also that $\mathbf{W}_n\to 0$ (recall also $\gb_n n^{-1/2} m(n^{3/2}) \to 0$).

\subsubsection{Step 4. Convergence of the main term.}
We conclude the proof by showing the convergence in distribution of the first term in \eqref{almostthere} -- which proves also that $\mathbf{V}_n$ goes to $0$ in probability, since $\gb_n n^{-1/2} m(n^{3/2}) \to 0$.

\begin{lemma}
\label{lem:convergence}
We have the following convergence in distribution,
\[\frac{\sqrt{n}}{\gb_n m(n^{3/2})}\mathbf{V}_n := \frac{\sqrt{n}}{\gb_n m(n^{3/2})}  \sum_{(i,x)\in \Lambda_{n,A_n}} \big(e^{\beta_n \tilde\omega_{i,x}} -1\big)p(i,x) \ \xrightarrow[n\to\infty]{(\dd)}\ \mathcal{W}_0^{(\ga)} \, ,\]
with $\mathcal{W}_0^{\ga}$ defined in Theorem~\ref{thm:alpha<12}.
\end{lemma}

\begin{proof}
First of all,  since $\beta_n \tilde\go_{i,x} \le \gb_n k_n \to 0$ as $n\to\infty$ (and using that $0\le e^{x}-1-x \leq x^2$ for $x$ small), we have that for $n$ large
\begin{equation}
0\le \mathbf{V}_n  - \gb_n \sum_{(i,x) \in \Lambda_{n,A_n} } \tilde\go_{i,x} p(i,x)  \le   \sum_{(i,x) \in \Lambda_{n,A_n}} \big(\gb_n \tilde\go_{i,x} \big)^2 p(i,x) \, .
\end{equation}
Then, we can estimate the expectation of the upper bound, using that similarly to \eqref{elambdan} we have
$\bbE[(\tilde \go)^2] \le c L(k_n) k_n^{2-\ga}  \sim c k_n^2 /(n^{3/2} \log n) $. Using also that $k_n\le m(n^{3/2}) (\log n)^{2/\ga}$ for $n$ large, we obtain that
\begin{align*}
\frac{\sqrt{n}}{ \gb_n m(n^{3/2})}\bbE\Big[ \sum_{(i,x) \in \bar \Lambda_n} \big(\gb_n \tilde\go_{i,x} \big)^2 p(i,x) \Big] & \le 
c \, \frac{k_n}{m(n^{3/2})}\, \gb_n k_n \, n^{-1}  \sum_{i=1} ^n \sum_{x\in \mathbb{Z}} p(i,x) \\
& \le c (\log n)^{2/\ga} \gb_n k_n \xrightarrow{n\to\infty} 0\, .
\end{align*}

The proof of the lemma is therefore  reduced to showing the convergence in distribution of the following term
\begin{align} \label{splitlastterm}
&\frac{\sqrt{n}}{m(n^{3/2})} \sum_{(i,x) \in \bar \Lambda_n}  \tilde \go_{i,x}  p(i,x)  \\
&=   \sum_{i=1}^n \sum_{|x|\le K\sqrt{n}} \frac{\tilde \go_{i,x}}{m(n^{3/2})}  \sqrt{n} p(i,x) + \sum_{i=1}^n \sum_{  K \sqrt n < |x| \le A_n}   \frac{\tilde \go_{i,x}}{m(n^{3/2})}  \sqrt{n} p(i,x),
\notag
\end{align}
where we fixed some level $K>0$ (we take the limit $K\to\infty$ in the end).

\smallskip
{
	\textit{First term in \eqref{splitlastterm}.}
	First, note that the first term converges in distribution to
	\begin{equation}
	\label{def:W0K}
	\cW_{0,K}^{(\alpha)}:=2\int_{\mathbb{R}_+}  \int_0^1 \int_{-K}^K w \rho(t,x) \cP(d w d t d x),
	\end{equation}
	where $\rho(t,x):=(2\pi t)^{-1/2}e^{-x^2/2t}$ is the Gaussian kernel and 
	$\cP( w,  t,x)$ is a PPP on $[0,\infty)\times [0,1] \times \mathbb R$ of intensity  
	$\mu(\dd w \dd t \dd x )=\frac{\alpha}{2} w^{-\alpha-1}\ind_{\{w>0\}}\dd w \dd t \dd x$.
	The proof of \eqref{def:W0K} is identical to that in \cite[p.~4036]{cf:DZ}, so we omit  details.

	Then, since $\cW_0^{(\ga)} <\infty$ a.s. (see \cite[Lemma~1.3]{cf:DZ}), one readily gets that $\cW_{0,K}^{(\alpha)} \to \cW_{0}^{(\alpha)}$ as $K\to\infty$ (by monotonicity).

}

\smallskip
\textit{Second term in \eqref{splitlastterm}.}
To conclude the proof, it remains to  show that the second term in \eqref{splitlastterm} goes to $0$ in probability as $K\to\infty$, uniformly in $n$: for any $K$ (large), we have for $n$ sufficiently large
\begin{equation}
\label{controlK}
\bbP\Big( \sum_{i=1}^n \sum_{  K \sqrt n < |x| \le A_n}   \frac{\tilde \go_{i,x}}{m(n^{3/2})}  \sqrt{n} p(i,x) \ge K^{-1} \Big)  \le c e^{- \ga K}\, .
\end{equation}
To prove \eqref{controlK}, we split the sum in parts with $|x|\in (2^{k-1} K\sqrt{n} , 2^{k} K\sqrt{n}] $ for $k =1,2\ldots$. By a union bound, we have
\begin{align}
\bbP \Big(\sum_{i=1}^n &\sum_{|x|> K\sqrt{n}}  \frac{\tilde \go_{i,x}}{m(n^{3/2})}  \sqrt{n} p(i,x) \ge K^{-1} \Big) \nonumber\\
& \le \sum_{k = 1}^{\infty}  \bbP \Big(  \sum_{i=1}^n \sum_{|x|= 2^{k-1} K\sqrt{n}}^{2^{k} K \sqrt{n}} \frac{\go_{i,x}}{m(n^{3/2})}  \sqrt{n} p(i,x) \ge K^{-1} 2^{-k} \Big)   \nonumber\\
& \le \sum_{k\ge 1} \bbP \Big(  \sum_{i=1}^n \sum_{|x|\le 2^{k} K \sqrt{n}}  \go_{i,x}  \ge e^{ c' (2^k K)^2 } m\big(  n^{3/2}\big) \Big)
\label{sumlargerK}
\end{align}
In the last inequality, we used that there is a constant $c$ such that for any~$k$, uniformly in $i\in\{1,\ldots, n\}$ and $|x| \ge 2^{k-1} K\sqrt{n}$, we have $\sqrt{n} p(i,x) \le e^{ - c (2^k K)^2 } \le 2^{-k}K^{-1} e^{ - c' (2^k K)^2 } $ (since $K2^k \ge 1$).

Now, we use that $m(2^{k+1} K n^{3/2} ) \ge (2^k K )^{-2/\ga} m(n^{3/2}) $ by Potter's bound, and also that for all $k$, $e^{ c' (2^k K)^2 }(2^k K )^{-2/\ga} \ge e^{2^k K}$ if $K$ is large: the last probability in \eqref{sumlargerK} is in the one-jump large deviation domain (see \cite[Thm.~1.1]{Nag79}, we use here that $\ga<1$): there is a $c>0$ such that for all $k\ge 1$
\begin{align*}
\bbP \Big(  \sum_{i=1}^n \sum_{|x|\le 2^{k} K \sqrt{n}} &  \go_{i,x}  \ge  e^{ 2^k K }   m\big(  2^{k+1}  Kn^{3/2}\big) \Big)\\
& \le c 2^k K n^{3/2} \bbP\Big( \go \ge    e^{ 2^k K }   m\big(  2^k  Kn^{3/2}\big) \Big)  \le c    e^{ - \frac{\ga}{2} 2^k K }\, .
\end{align*}
The second inequality comes from Potter's bound, provided that $ e^{2^k K}$ is large enough, and also  the definition \eqref{def:m} of $m(\cdot)$.
Plugged in \eqref{sumlargerK}, we get 
\[\bbP \Big(\sum_{i=1}^n \sum_{|x|> K\sqrt{n}}  \frac{\tilde \go_{i,x}}{m(n^{3/2})}  \sqrt{n} p(i,x) \ge \gep \Big) \le c \sum_{k\ge 1} e^{- \frac\ga2 2^k K} \le c e^{- \ga K} \, ,\]
which is \eqref{controlK}.
\end{proof}

{\bf Acknowledgements}
We are most grateful to N. Zygouras for many enlightening discussions.

\bibliographystyle{plain}
\bibliography{biblioHTBN.bib}

\end{document}